\documentclass[10pt,reqno]{amsart}

\usepackage{setspace}
\setdisplayskipstretch{1.5}
\usepackage[T1]{fontenc}
\usepackage{lmodern}

\usepackage{mathtools}
\usepackage{amsmath,amsfonts,amsbsy,amsgen,amscd,mathrsfs,amssymb,amsthm}
\usepackage{enumerate}
\usepackage{bm}
\usepackage{calc}

\usepackage{stmaryrd}
\SetSymbolFont{stmry}{bold}{U}{stmry}{m}{n} 

\usepackage[usenames,dvipsnames]{xcolor}
\usepackage[colorlinks=true,citecolor=blue,linkcolor=blue]{hyperref}

\usepackage{tikz}
\usepackage[font=small]{caption}

\newtheorem{thm}{Theorem}[section]
\newtheorem{lem}[thm]{Lemma}

\newtheorem{prop}[thm]{Proposition}
\newtheorem{cor}[thm]{Corollary}

\theoremstyle{definition}

\newtheorem{defn}[thm]{Definition}

\newtheorem{example}[thm]{Example}

\newtheorem{rem}[thm]{Remark}
\newtheorem*{rem*}{Remark}

\numberwithin{equation}{section} 
\numberwithin{figure}{section}
\numberwithin{table}{section}

\makeatletter

\newcommand{\id}{\mathbf{1}}

\newcommand{\gue}{G^N}

\newcommand{\xn}{X^N}
\newcommand{\gn}{G^N}
\newcommand{\hn}{H^N}
\newcommand{\xf}{X_{\rm F}}

\newcommand{\yn}{Y^N}

\newcommand{\haar}{U^N}
\newcommand{\haaradj}{U^{N*}}
\newcommand{\haaro}{U^N}
\newcommand{\haars}{V^N}

\newcommand{\wgo}{\widetilde{\wg}}

\newcommand{\trace}{\mathop{\mathrm{Tr}}}
\newcommand{\tr}{\mathop{\textnormal{tr}}}

\newcommand{\calP}{\mathcal{P}}

\newcommand{\calM}{\mathcal{M}}

\newcommand{\bZ}{\mathbb{Z}}
\newcommand{\bN}{\mathbb{N}}
\newcommand{\bC}{\mathbb{C}}

\renewcommand{\P}{\mathbf{P}}
\newcommand{\E}{\mathbf{E}}

\newcommand{\supp}{\mathop{\mathrm{supp}}}

\newcommand{\les}{\lesssim}

\newcommand{\bR}{\mathbb{R}}
\newcommand{\bH}{\mathbb{H}}

\newcommand{\bF}{\mathbf{F}}

\newcommand{\sobnorm}[3]{\|#1\|_{#3}}

\newcommand{\wg}{\mathrm{Wg}}

\begin{document}

\title[A new approach to strong convergence II]{A new approach to strong 
convergence II.\\The classical ensembles }

\author[Chen]{Chi-Fang Chen}
\address{EECS, University of California, Berkeley, CA 94720, USA
\newline\indent
Center for Theoretical Physics, Massachusetts Institute of 
Technology, Cambridge, MA 02139, USA}
\email{achifchen@gmail.com}

\author[Garza-Vargas]{Jorge Garza-Vargas}
\address{Department of Mathematics, Princeton University, Princeton, NJ
08544, USA}
\email{jgarzavargas@princeton.edu}

\author[Van Handel]{Ramon van Handel}
\address{Department of Mathematics, Princeton University, Princeton, NJ 
08544, USA}
\email{rvan@math.princeton.edu}

\begin{abstract}
The first paper in this series introduced a new approach to strong 
convergence of random matrices that is based primarily on soft arguments. 
This method was applied to achieve a refined qualitative and quantitative 
understanding of strong convergence of random permutation matrices and of 
more general representations of the symmetric group. In this paper, we 
introduce new ideas that make it possible to achieve stronger 
quantitative results and that facilitate the application of the method to 
new models.

When applied to the Gaussian GUE/GOE/GSE ensembles of dimension $N$, these 
methods achieve strong convergence for noncommutative polynomials with 
matrix coefficients of dimension $\exp(o(N))$. This provides a sharp form 
of a result of Pisier on strong convergence with coefficients in a 
subexponential operator space. Analogous results up to logarithmic factors 
are obtained for Haar-distributed random matrices in 
$\mathrm{U}(N)/\mathrm{O}(N)/\mathrm{Sp}(N)$. We further illustrate the 
methods of this paper in the following applications.\\
\begin{minipage}{0.8\textwidth}
\vspace*{.15cm}
\begin{enumerate}[1.]
\addtolength{\itemsep}{.5mm}
\item We obtain improved rates for
strong convergence of random permutations.
\item We obtain a
quantitative form of strong convergence of the model introduced by Hayes
for the solution of the Peterson-Thom conjecture.
\item We prove strong convergence of tensor GUE models of
$\Gamma$-independence.
\item We prove strong
convergence of irreducible representations of $\mathrm{U}(N)$ of
dimension up to $\exp(N^{1/3-\delta})$, improving a result of
Magee and de la Salle.
\end{enumerate}
\end{minipage}
\end{abstract}


\maketitle

\raggedbottom

\section{Introduction}
\label{sec:intro}

A sequence of $r$-tuples $\boldsymbol{X}^N=(X_1^N,\ldots,X_r^N)$
of random matrices is said to \emph{converge strongly} to an
$r$-tuple $\boldsymbol{x}=(x_1,\ldots,x_r)$ of elements of a $C^*$-algebra
if
$$
        \lim_{N\to\infty}\|P(X_1^N,\ldots,X_r^N,X_1^{N*},\ldots,
	X_r^{N*})\|
	=
        \|P(x_1,\ldots,x_r,x_1^*,\ldots,x_r^*)\|
$$
in probability for every noncommutative polynomial $P$.
In recent years, this 
notion has proved to have powerful consequences for 
problems of random graphs, random surfaces, and operator algebras, 
resulting in major breakthroughs in these areas; see, for example, 
\cite{BC24,BC20,chen2024new} for discussion and references.

Our previous paper \cite{chen2024new} introduced a new approach to strong 
convergence that is based primarily on soft arguments and requires limited 
problem-specific inputs, in contrast to earlier approaches that were 
heavily dependent on problem-specific analytic methods and/or delicate 
combinatorial estimates. The replacement of hard analysis by soft 
arguments has made it possible to establish strong convergence in new 
situations, and to access quantitative information about the strong 
convergence phenomenon that was previously out of reach. Both features 
were illustrated in \cite{chen2024new} in the context of random 
permutation matrices and of more general representations of the symmetric 
group. Another illustration is provided by the remarkable results of Magee 
and de la Salle \cite{magee2024quasiexp} and Cassidy \cite{Cas24} that 
establish strong convergence of extremely high-dimensional 
representations of $\mathrm{U}(N)$ and $\mathrm{S}_N$.

In this paper, we develop new ideas that advance the method of 
\cite{chen2024new} in two directions: they achieve considerably stronger 
(and in some respects nearly optimal) quantitative results; and they 
further eliminate the need for problem-specific computations in many 
situations, which facilitates the application of the method to new models. 
These new ingredients, which will be discussed in section \ref{sec:new} 
below, enable a particularly transparent and nearly parallel treatment of 
the classical invariant ensembles of random matrix theory that yields new 
quantitative information on the strong convergence phenomenon for these 
models. Beyond the classical ensembles, we will further illustrate these 
methods in several additional applications.

\subsection{Main results}

\subsubsection{The Gaussian ensembles}

In this paper, a GUE/GOE/GSE random matrix is an $N\times 
N$ self-adjoint Gaussian random matrix $G^N$ whose law is invariant under 
unitary/orthogonal/symplectic transformations, and whose entries above the 
diagonal have variance $\frac{1}{N}$.
The limiting operators associated to independent 
random matrices $G_1^N,\ldots,G_r^N$ from these ensembles form a free 
semicircular family $s_1,\ldots,s_r$.\footnote{%
	The reader is warned that this notation differs from 
	\cite{chen2024new}. In the present paper, a free semicircular 
	family is denoted as $\boldsymbol{s}=(s_1,\ldots,s_r)$, while 
	free Haar unitaries are denoted as $\boldsymbol{u}=(u_1,\ldots,u_r)$.}
Precise definitions of these notions will be recalled in section 
\ref{sec:rmtandfree}.

The following is the main result of this paper on the Gaussian ensembles. 
A more explicit form of the constant $c$ appears in the proof.

\begin{thm}
\label{thm:maingauss}
Let $\boldsymbol{G}^N=(G_1^N,\ldots,G_r^N)$ be i.i.d.\ GUE/GOE/GSE random 
matrices of dimension $N$, and let $\boldsymbol{s}=(s_1,\ldots,s_r)$ be a 
free semicircular family. Let $\varepsilon\in (0,1]$ and $q_0\in\mathbb{N}$. 
Then for every noncommutative polynomial $P\in\mathrm{M}_D(\mathbb{C})\otimes 
\mathbb{C}\langle\boldsymbol{s}\rangle$ of degree $q_0$ with matrix 
coefficients of dimension $D\le e^{cN\varepsilon^2}$, we have
$$
	\mathbf{P}\big[
	\|P(\boldsymbol{G}^N)\|\ge (1+\varepsilon)\|P(\boldsymbol{s})\|
	\big] \le
	\frac{N}{c\varepsilon}
	\,e^{-cN\varepsilon^2},
$$
where $c$ is a constant that depends only on $q_0$ and $r$.
\end{thm}

Theorem \ref{thm:maingauss} settles a question that has motivated many 
recent works on strong convergence. In its basic form, strong convergence 
implies that 
$$
	\|P(\boldsymbol{G}^N)\|=(1+o(1))\|P(\boldsymbol{s})\|
	\quad\text{as}\quad N\to\infty
$$
when the polynomial $P$ and thus the coefficient dimension $D$ is fixed. 
However, much stronger implications could be obtained if $D$ is allowed to 
grow sufficiently rapidly with $N$ (a considerable strengthening of the 
strong convergence property). For example, Hayes \cite{hayes2020random} 
shows that the case $D=N$ already suffices to prove the Peterson-Thom 
conjecture in the theory of von Neumann algebras. The latter was settled 
in \cite{BC22,BC24,magee2024quasiexp}, and Theorem \ref{thm:maingauss} 
provides yet another proof of this conjecture. On the other hand, in his 
study of subexponential operator spaces, Pisier \cite{pisier2014random} 
shows that strong convergence holds \emph{up to a factor $2$} even when 
$D=e^{o(N)}$.

Theorem~\ref{thm:maingauss} closes the gap between these two extremes by 
showing that strong convergence holds whenever $D=e^{o(N)}$, providing a 
sharp form of Pisier's theorem.

\begin{cor}
\label{cor:gausssubexp}
Let $\boldsymbol{G}^N$ and $\boldsymbol{s}$ be defined as in Theorem 
\ref{thm:maingauss}. For any sequence of noncommutative polynomials 
$P_N\in\mathrm{M}_{D_N}(\mathbb{C})\otimes 
\mathbb{C}\langle\boldsymbol{s}\rangle$ of degree $O(1)$ and matrix 
coefficients of dimension $D_N=e^{o(N)}$, we have
$$
	\|P_N(\boldsymbol{G}^N)\|=(1+o(1))\|P_N(\boldsymbol{s})\|
	\quad\text{a.s.} \quad\text{as}\quad N\to\infty.
$$
Consequently, for any noncommutative polynomial 
$Q\in\mathbf{W}\otimes\mathbb{C}\langle\boldsymbol{s}\rangle$ with
coefficients in a subexponential operator space $\mathbf{W}$, we have
$$
	\frac{1}{C(\mathbf{W})}
	\|Q(\boldsymbol{s})\|_{\rm min} \le
	\liminf_{N\to\infty}\|Q(\boldsymbol{G}^N)\|
	\le
	\limsup_{N\to\infty}\|Q(\boldsymbol{G}^N)\|
	\le C(\mathbf{W}) \|Q(\boldsymbol{s})\|_{\rm min}
	\text{ a.s.},
$$
where $C(\mathbf{W})$ denotes the subexponential constant of
$\mathbf{W}$.\footnote{The definition of a subexponential
operator space is recalled in section \ref{sec:subexpon}.}
\end{cor}

Whether strong convergence could hold even beyond the subexponential 
regime is a tantalizing question. While we do not resolve this question, 
our results are optimal in the sense that they achieve the largest regime 
that is accessible by trace statistics, as will be discussed in section 
\ref{sec:discoptdim} below.

\begin{rem}[Previous bounds]
\label{rem:previous}
Prior to the present work, strong convergence of Gaussian 
ensembles was established only for $D=o(N/\log^3 N)$ \cite{BBV23} 
(following earlier works that achieved $D=o(N^{1/4})$ 
\cite{pisier2014random} and $D=o(N^{1/3})$ \cite{CGP22}). Thus even the 
linear dimension regime had remained out of reach in this setting.

For Haar unitary matrices (see section \ref{sec:mainclass}), analogous 
sublinear bounds appear in \cite{Par22,parraud2023asymptoticHaar}. In 
this setting, a major step forward \cite{BC24} achieved
strong convergence for $D\le \exp(N^{1/(32r+160)})$, breaking 
the linear dimension barrier. A significant improvement 
$D\le\exp(N^{1/2-o(1)})$ was obtained in \cite{magee2024quasiexp}. Most 
recently, in work concurrent with the present paper, the GUE case was 
revisited in \cite{Par24} where strong convergence was proved for 
$D=\exp(o(N^{2/3}))$. The methods of the present paper finally
make it possible to reach the entire subexponential regime.
\end{rem}

In the complementary regime where $D=N^{O(1)}$ is polynomial, Theorem 
\ref{thm:maingauss} yields a universal bound on the rate of strong 
convergence: a direct application of Theorem \ref{thm:maingauss} with 
$\varepsilon=C\sqrt{\log(N)/N}$ yields the following.

\begin{cor}
\label{cor:gaussrate}
Let $\boldsymbol{G}^N$ and $\boldsymbol{s}$ be defined as in Theorem
\ref{thm:maingauss}. For any sequence of noncommutative polynomials 
$P_N\in\mathrm{M}_{D_N}(\mathbb{C})\otimes
\mathbb{C}\langle\boldsymbol{s}\rangle$
with $D_N=N^{O(1)}$, we have\footnote{%
	The notation $Z_N = O_{\rm P}(z_N)$ denotes that
	$\{Z_N/z_N\}_{N\ge 1}$ is bounded in probability.}
$$
	\|P_N(\boldsymbol{G}^N)\|\le 
	\bigg(
	1
	+ O_{\mathrm{P}}\bigg(\sqrt{\frac{\log N}{N}}\bigg)
	\bigg)\|P_N(\boldsymbol{s})\|
	\quad\text{as}\quad N\to\infty.
$$
\end{cor}

A similar rate was obtained by Parraud \cite{parraud2023asymptotic} for
GUE (and Haar unitaries \cite{parraud2023asymptoticHaar}) 
with $D=1$, but the method used there does not extend to large $D$.

The $N^{-1/2}$ rate is not expected to be optimal: when $P$ is a linear 
polynomial with scalar coefficients, the optimal rate $N^{-2/3}$ follows 
from classical Tracy-Widom asymptotics. At present, however, Corollary 
\ref{cor:gaussrate} yields the best known rate for arbitrary polynomials 
$P$; see section \ref{sec:discoptrate} for further discussion.

\subsubsection{The classical compact groups}
\label{sec:mainclass}

For Haar-distributed random matrices from the classical compact groups 
$\mathrm{U}(N)/\mathrm{O}(N)/\mathrm{Sp}(N)$, the methods of this paper 
yield nearly parallel results to those obtained for the Gaussian 
ensembles. Our main result in this setting differs from
Theorem \ref{thm:maingauss} by a logarithmic factor.

\begin{thm}
\label{thm:mainhaar}
Let $\boldsymbol{U}^N=(U_1^N,\ldots,U_r^N)$ be i.i.d.\ Haar-distributed
random matrices in $\mathrm{U}(N)/\mathrm{O}(N)/\mathrm{Sp}(N)$,
and let $\boldsymbol{u}=(u_1,\ldots,u_r)$ be free Haar unitaries.
Let $\varepsilon\in [\frac{1}{c\sqrt{N}},1]$ and $q_0\in\mathbb{N}$. 
For every noncommutative polynomial $P\in\mathrm{M}_D(\mathbb{C})\otimes 
\mathbb{C}\langle\boldsymbol{u},\boldsymbol{u}^*\rangle$ of degree $q_0$ 
with matrix coefficients of dimension $D\le e^{cN\varepsilon^2/
\log^2(N\varepsilon^2)}$, we have
$$
	\mathbf{P}\big[
	\|P(
	\boldsymbol{U}^N,
	\boldsymbol{U}^{N*})\|\ge (1+\varepsilon)
	\|P(\boldsymbol{u},\boldsymbol{u}^*)\|
	\big] \le
	\frac{N}{c\varepsilon}
	\,e^{-cN\varepsilon^2/\log^2(N\varepsilon^2)}.
$$
Here $c$ is a constant that depends only on $q_0$ and $r$.
\end{thm}

Theorem \ref{thm:mainhaar} yields strong convergence whenever $D=
e^{o(N/\log^2 N)}$, and yields a universal rate
$O_{\rm P}((\frac{\log N}{N})^{1/2}\log\log N)$ when $D=N^{O(1)}$.

\begin{rem}
\label{rem:cmale}
Theorem \ref{thm:mainhaar} falls short by a logarithmic factor of reaching 
the full subexponential regime. However, as was pointed out to us by 
Mikael de la Salle, Corollary 
\ref{cor:gausssubexp} extends \emph{verbatim} to the setting of
Theorem \ref{thm:mainhaar} by using a coupling between the Gaussian and 
Haar-distributed ensembles due to Collins and Male \cite{CM14}.
This argument achieves the full subexponential regime for the
$\mathrm{U}(N)/\mathrm{O}(N)/\mathrm{Sp}(N)$ models, but does not provide 
any quantitative information.

This suggests that the logarithmic factor in Theorem \ref{thm:mainhaar} 
is likely an artefact of the proof and should not be necessary.
The logarithmic factor arises from a single point in 
the proof that is explained in Remark \ref{rem:loselog} below.
\end{rem}

\subsubsection{Further applications}
\label{sec:furtherapp}

While the new ingredients developed in this paper enable a particularly 
sharp treatment of the classical Gaussian and Haar-distributed ensembles, 
they are by no means restricted to this setting. To further illustrate the 
utility of these methods, we will develop four additional applications.
\smallskip
\begin{enumerate}[1.]
\itemsep\smallskipamount
\item
It was shown in \cite{chen2024new} that strong convergence of
random permutation matrices holds with a universal rate 
$O_{\rm P}((\frac{\log N}{N})^{1/8})$. We will obtain
$O_{\rm P}((\frac{\log N}{N})^{1/6})$
with no additional effort, improving the best known rate for this problem.
\item
In an influential paper that motivated much recent work on strong 
convergence, Hayes \cite{hayes2020random} proved a reduction of the 
Peterson-Thom conjecture to the statement that the family of
$N^2$-dimensional random matrices
$$
	G_1^N\otimes\id_N,~\ldots,~G_r^N\otimes\id_N,~\id_N\otimes 
	\tilde G_1^N,~\ldots,~\id_N\otimes \tilde G_r^N
$$
converges strongly as $N\to\infty$ to 
$$
	s_1\otimes\id,~\ldots,~s_r\otimes\id,~\id\otimes 
	s_1,~\ldots,~\id\otimes s_r,
$$
where $G_i^N,\tilde G_i^N$ are independent GUE matrices of dimension $N$ 
and $s_i$ are free semicircular variables. This strong convergence 
property follows readily from Corollary \ref{cor:gausssubexp} using 
exactness of the $C^*$-algebra generated by a free semicircular family, 
but this argument provides no quantitative information. We will develop a
quantitative form of strong convergence of Hayes' model.
\item
In contrast to the Hayes model, where GUE matrices act on distinct factors
of a tensor product, models where GUE matrices act on 
overlapping factors of a tensor product arise naturally in the study of 
quantum many-body systems and in random geometry. The limiting model
in this setting is described by the notion of $\Gamma$-independence 
\cite{SW16,CC21}. We will prove strong convergence of general tensor GUE 
models to a $\Gamma$-independent semicircular family, settling open 
problems formulated in \cite[Problem 1.6]{MT23} and \cite{CY24}.
Quantitative bounds on strong convergence for the Hayes 
model play a key role in the proof.
\item
Let $g_1,\ldots,g_r$ be i.i.d.\ Haar distributed elements of 
$\mathrm{SU}(N)$. Given a unitary representation $\pi_N$ of 
$\mathrm{SU}(N)$, we can define random matrices $U_i^{\pi_N} := 
\pi_N(g_i)$ of dimension $\dim(\pi_N)$. It was shown by Magee and de la 
Salle \cite{magee2024quasiexp} that the random matrices 
$U_1^{\pi_N},\ldots, U_r^{\pi_N}$ converge strongly to free Haar unitaries 
$u_1,\ldots,u_r$ uniformly over all nontrivial representations $\pi_N$ 
with $\dim(\pi_N)\le \exp(N^{1/42-\delta})$, achieving the first 
strong convergence result for representations of quasiexponential 
dimension (much lower dimensional representations were considered in 
earlier work of Bordenave and Collins \cite{BC20}). We will improve this 
conclusion to representations of dimension $\dim(\pi_N)\le 
\exp(N^{1/3-\delta})$.
\end{enumerate}
\smallskip
We postpone precise mathematical statements of these results to
section \ref{sec:applications},
where the above applications will be developed.

\subsection{New ingredients}
\label{sec:new}

A detailed overview of the soft approach to strong convergence introduced 
in our previous paper is given in \cite[\S 2.2]{chen2024new}. Here we 
summarily recall only the most basic steps of this method in order to 
enable the discussion of the new ideas developed in this paper. The reader 
who is new to the method is encouraged to review \emph{ibid.}\ prior to 
proceeding.

\subsubsection{Review of the basic method}
\label{sec:introreview}

Let $X^N$ (e.g., $P(\boldsymbol{U}^N,\boldsymbol{U}^{N*})$) be a 
self-adjoint random matrix of dimension $M_N$, and let $X_{\rm F}$ (e.g., 
$P(\boldsymbol{u},\boldsymbol{u}^*)$) be its limiting model in a 
$C^*$-probability space $(\mathcal{A},\tau)$.
For the present discussion, assume for simplicity that $\|X^N\|\le 
K$ a.s. To control the norm of $X^N$, we bound
\begin{equation}
\label{eq:conclosehere}
	\mathbf{P}\big[\|X^N\|\ge\|X_{\rm F}\|+\varepsilon\big]
	\le
	\mathbf{E}[\trace \chi(X^N)] =
	M_N\,\mathbf{E}[\tr \chi(X^N)],
\end{equation}
where $\tr$ denotes the normalized trace and $\chi\ge 0$ is a smooth test 
function so that $\chi(x)$ vanishes for  $|x|\le\|X_{\rm F}\|+ 
\frac{\varepsilon}{2}$ and equals one for $|x|\ge\|X_{\rm F}\|+\varepsilon$. 
Our aim is to show that the right-hand side of this bound is
$o(1)$.
While we are ultimately interested in smooth test functions $\chi$, our 
approach is based on the availability of powerful tools in the analytic 
theory of polynomials; the application of these tools to polynomial test 
functions will yield quantitative bounds
that are so strong that they can be lifted to smooth test functions a 
posteriori.

The
basic input for our approach is that for many models (including all those
considered in this paper), any polynomial $h$ of $X^N$ satisfies
$$
	\mathbf{E}[\tr h(X^N)] = \Phi_h\big(\tfrac{1}{N}\big),
$$
where $\Phi_h$ is a rational function whose degree is bounded in terms of 
that of $h$. Using only this fact and the trivial bound
$|\Phi_h(\frac{1}{N})|\le \|h\|_{[-K,K]}:= \sup_{|x|\le K}|h(x)|$, we can
apply classical polynomial inequalities due to 
A.\ and V.\ Markov to obtain an asymptotic expansion of $\mathbf{E}[\tr 
h(X^N)]$ of the form
\begin{equation}
\label{eq:introexp}
	\Bigg|
	\mathbf{E}[\tr h(X^N)] - \tau(h(X_{\rm F})) -
	\sum_{k=1}^{m-1} \frac{\nu_k(h)}{N^k}
	\Bigg|
	\le
	\frac{C(m)}{N^m}\,q^{\beta m}\|h\|_{[-K,K]}
\end{equation}
for all $N,m,q\in\mathbb{N}$ and real polynomials $h$ of degree at most 
$q$. Here $C(m)$ is a constant that depends only on $m$, $\beta$ is a 
universal constant, and $\nu_k$ are linear functionals on the space of 
real univariate polynomials.

The key feature of \eqref{eq:introexp} is that the error bound is 
sufficiently strong that it can efficiently control the expansion of any 
smooth function into Chebyshev polynomials. In particular, a 
Fourier-analytic argument shows that the linear functionals $\nu_k$
extend to compactly supported distributions, and that
the expansion \eqref{eq:introexp} remains valid for any smooth 
function $h$ when $q^{\beta m}\|h\|_{[-K,K]}$ is replaced by 
$\|h\|_{C^{\lceil\beta m\rceil+1}[-K,K]}$ on the right-hand side.
If we could furthermore show that
\begin{equation}
\label{eq:suppincl}
	\supp \nu_k \subseteq [-\|X_{\rm F}\|,\|X_{\rm F}\|]
	\quad\text{for }k=1,\ldots,m-1,
\end{equation}
then $\chi$ in \eqref{eq:conclosehere} would satisfy 
$\nu_k(\chi)=0$ for $k\le m-1$ and the expansion yields 
$$
	M_N\,\mathbf{E}[\tr \chi(X^N)] = O\bigg(\frac{C'(m)M_N}{N^m}\bigg).
$$
Thus 
we achieve strong convergence provided that \eqref{eq:suppincl} can be 
established for $m$ sufficiently large that the above bound is $o(1)$.
It is shown in \cite{chen2024new} how the problem of bounding the support
of $\nu_k$ can be reduced to a moment computation; the latter is the main 
part of the method that relies on a problem-specific analysis.

\begin{rem}
In the context of strong convergence, asymptotic expansions for smooth 
test functions were first used by Schultz \cite{schultz2005non}, and were 
systematically developed by Parraud 
\cite{parraud2023asymptotic,parraud2023asymptoticHaar,Par24} for GUE and 
Haar unitary matrices. These works rely on specialized analytic 
tools and explicit computations that are available for these models. A key 
feature of the polynomial method is that it applies to models
for which such specialized tools are not available. At the same time, our
main results
yield stronger quantitative information even for the classical ensembles.
\end{rem}

The basic approach described above achieves not only strong 
convergence, but also strong quantitative bounds that 
yield much stronger implications. The quantitative features of the method 
are controlled by three parameters:
\begin{enumerate}[1.]
\item The value of $\beta$ in \eqref{eq:introexp};
\item The dependence of $C(m)$ in \eqref{eq:introexp} on $m$;
\item The largest $m$ for which \eqref{eq:suppincl} can be established.
\end{enumerate}
The main contributions of this paper are several (independent) new 
ingredients that yield significant improvements to the method of 
\cite{chen2024new} in each of these parameters. We describe these 
new ingredients in the remainder of this section. The combination of all 
these ingredients is key to achieving our main results.

\subsubsection{Optimal polynomial interpolation}
\label{sec:introrakh}

The proof of \eqref{eq:introexp} is based on the observation that its 
left-hand side is merely the Taylor expansion to order $m-1$ of the 
rational function $\Phi_h(\frac{1}{N})$, so that it is bounded by the 
remainder term $\frac{1}{m! N^m}\|\Phi_h^{(m)}\|_{[0,\frac{1}{N}]}$. The 
problem with this bound is that it depends on $\Phi_h(x)$ for $x$ not of 
the form $\frac{1}{N}$, so that the connection with random matrices is 
lost. We surmount this using the classical fact that bounding a polynomial 
on a sufficiently dense discrete set already suffices to achieve uniform 
control of its derivatives.

A key step in the argument is that we must bound the rational function 
$\Phi_h$ in between the points $\frac{1}{N}$ by interpolating its values 
at these points. To this end, \cite{chen2024new} relies on a classical 
result on polynomial interpolation, which states that for any real 
polynomial $h$ of degree $q$, we have $\|h\|_{[0,\delta]}\lesssim 
\max_{x\in I}|h(x)|$ for any set $I\subseteq[0,\delta]$ with spacing at 
most $O(\frac{\delta}{q^2})$ between its points. The latter condition is 
optimal for a general set $I$ \cite{CR92}. When applied to the set $I_M 
:=\{\frac{1}{N}:N\ge M\}$ that is of interest in the present setting, the 
spacing condition limits us to considering only $N\gtrsim q^2$, which 
results in a quantitative loss in the analysis.

Surprisingly, this restriction turns out to be suboptimal in the present 
setting due to the special structure of the set $I_M$: even though 
$O(\frac{\delta}{q^2})$ spacing is necessary for general $I$, we will 
prove in section \ref{sec:optimaluniformbounds} that $O(\frac{\delta}{q})$ 
spacing suffices (and is optimal) for $I_M$, so that we can in fact work 
with $N\gtrsim q$ in the analysis. While this is in itself purely a 
statement about polynomials that is unrelated to random matrices, it 
yields a crucial improvement of the constant $\beta$ in 
\eqref{eq:introexp} in essentially every application of the polynomial 
method in the random matrix context.

The special feature of the set $I_M$ is that it becomes increasingly dense 
near zero. This enables us to exploit a result of Rakhmanov 
\cite{rakhmanov2007bounds}, which states that $O(\frac{\delta}{q})$ 
uniformly spaced points suffice to interpolate a real polynomial of degree 
$q$ strictly in the interior of the interval $[0,\delta]$, in a multiscale 
manner.

\subsubsection{High-order expansions}

The strong quantitative results of this paper rely on asymptotic expansion 
to very large order $m$ (e.g., $m\propto N$ for Gaussian ensembles), which 
requires an essentially optimal constant $C(m)$ in \eqref{eq:introexp}. To 
this end, we must overcome two distinct obstacles that arise in different 
models.

For Haar-distributed models, we aim to apply inequalities for polynomials 
to the rational function $\Phi_h$. This is accomplished in 
\cite{chen2024new} by applying the chain rule to express $\Phi_h^{(m)}$ in 
terms of the derivatives of the numerator and denominator and bounding 
each term separately, resulting in lossy estimates. In 
section~\ref{sec:rationalbernstein}, we show instead that the rational 
function $\Phi_h$ can be approximated to very high precision by a 
polynomial of nearly the same degree, which enables us to apply polynomial 
inequalities directly without incurring a quantitative loss.

For Gaussian ensembles, the function $\Phi_h$ is itself a polynomial 
(known as the genus expansion), so that the above issue does not arise. In 
this setting, however, the random matrices are not uniformly bounded, so 
that the assumption $\|X^N\|\le K$ a.s.\ that was made for simplicity in 
section \ref{sec:introreview} does not apply. Surmounting this issue 
requires a truncation argument. The challenge in implementing such an 
argument is that this must be done without incurring any quantitative loss 
in the final bounds. The methods to do so will be developed in section
\ref{sec:thepolymethodforGUE}.

\subsubsection{Support and concentration}
\label{sec:introboot}

The quantitative features of our bounds are controlled not only by the 
asymptotic expansion \eqref{eq:introexp}, but also by the number of 
distributions $\nu_k$ whose supports can be bounded as in 
\eqref{eq:suppincl}. In general, \eqref{eq:suppincl} need not hold for 
arbitrarily large $m$; for example, this is the case for models based 
on random permutation matrices. The latter phenomenon has a precise
probabilistic interpretation: $x\in\supp\nu_k$ for some $|x|>\|X_{\rm 
F}\|$ detects the presence of an outlier in the spectrum of $X^N$ with 
probability $\sim N^{1-k}$, cf.\ \cite[\S 3.2.2]{chen2024new}.

Unlike random permutation matrices, however, the norms of random matrices 
constructed from Gaussians or the classical compact groups are 
subject to the concentration of measure phenomenon \cite{Led01}, which 
ensures that the 
probability that the norm deviates from its median by a fixed amount is 
\emph{exponentially} small in $N$. Thus if strong convergence holds 
in a qualitative sense $\mathrm{med}(\|X^N\|)= \|X_F\|+o(1)$, then the 
presence of an outlier in the spectrum with probability $N^{-c}$ 
is automatically ruled out. In other words, whenever \eqref{eq:introexp}
holds, we have the formal implication
\begin{multline*}
	\text{concentration of measure} + 
	\text{qualitative strong convergence}
	\Longrightarrow \\
	\supp\nu_k \subseteq [-\|X_{\rm F}\|,\|X_{\rm F}\|]
	\quad\text{for all }k\ge 1,
\end{multline*}
cf.\ section \ref{sec:bootsgue}.
When applicable, this simple observation 
controls the supports of $\nu_k$ in a soft manner, avoiding the need for
problem-specific moment computations such as those used in 
\cite{chen2024new} for random permutations. Let us note that a 
variant of this idea appears in the work of Parraud 
\cite[pp.\ 285--286]{parraud2023asymptotic}.

It should be emphasized that the above argument cannot in itself prove 
strong convergence, as it requires strong convergence as input. However, 
establishing strong convergence generally only requires the validity of 
\eqref{eq:suppincl} for small $m$: for example, when $X^N$ has dimension 
$M_N=N$, we must only understand $\nu_1$ ($m=2$) to achieve strong 
convergence. The concentration argument then automatically extends the 
conclusion to $\nu_k$ for all $k>1$, resulting in far stronger 
quantitative bounds. Thus the concentration method serves as a 
powerful bootstrapping argument to deduce strong quantitative bounds from 
weak ones.

\subsubsection{Supersymmetric duality}
\label{sec:supersymm}

In contrast to the above techniques that are broadly applicable, we 
finally discuss an idea that is special to the classical random matrix 
ensembles. Despite its limited range of applicability, a special property 
of these models can be very fruitfully exploited when it is present.

A remarkable property of the classical ensembles is that the rational 
function $\mathbf{E}[\tr h(X^N)] = \Phi_h(\frac{1}{N})$ still has a 
spectral interpretation if we replace $N$ by $-N$: there exists a 
``dual'' random matrix model $Y^N$ so that 
$$
	\mathbf{E}[\tr h(Y^N)] =
	\Phi_h(-\tfrac{1}{N}).
$$
In particular, GUE and $\mathrm{U}(N)$ models are self-dual, while GOE and 
$\mathrm{O}(N)$ models are dual to GSE and $\mathrm{Sp}(N)$ models
\cite{bryc2009duality,magee2024matrix}.
We presently explain how this yields stronger bounds using a classical 
property of polynomials \cite{borwein2012polynomials}.

Recall (cf.\ section \ref{sec:introrakh}) that the proof of 
\eqref{eq:introexp} aims to bound the remainder term 
$\|\Phi_h^{(m)}\|_{[0,\frac{1}{N}]}$ in the Taylor expansion of $\Phi_h$, 
while polynomial interpolation yields a uniform bound on $\Phi_h$ itself. 
This is achieved using the Markov inequality: if a real 
polynomial $h$ of degree $q$ is uniformly bounded on an interval, its 
derivative is bounded by $O(q^2)$ everywhere on that interval.

While the Markov inequality is optimal if we aim to bound the derivative 
of a polynomial everywhere on the interval, a much better $O(q)$ bound 
holds strictly in the interior of the interval by the Bernstein 
inequality. Unfortunately, this is not applicable in our setting, as we 
can only control $\Phi_h$ in a positive interval $[0,\delta]$ and we aim 
to control its derivatives near a boundary point $0$ of this interval. 
However, the existence of a dual random matrix model enables us to bound 
$\Phi_h$ on a symmetric interval $[-\delta,\delta]$. Thus in this case we 
can apply the Bernstein inequality to achieve a crucial improvement to the 
constant $\beta$ in \eqref{eq:introexp}.

Aside from this quantitative improvement, we will also exploit the duality 
property in an entirely different manner: by an elementary observation, 
the existence of a dual random matrix model automatically implies the 
validity of \eqref{eq:suppincl} for $m=2$, cf.\ section 
\ref{sec:nu1supersymmetry}. Somewhat surprisingly, this completely 
eliminates the need for any problem-specific moment estimates from the 
proofs of our main results.

\subsection{Discussion}

\subsubsection{The optimal dimension of matrix coefficients}
\label{sec:discoptdim}

Corollary \ref{cor:gausssubexp} states that the Gaussian ensembles exhibit 
strong convergence for polynomials with matrix coefficients of 
subexponential dimension 
$D=e^{o(N)}$. Whether or not this conclusion is the best possible is a 
tantalizing question of Pisier \cite{pisier2014random}. Pisier 
shows\footnote{%
	As this is not stated explicitly in \cite{pisier2014random}, we 
	include a short self-contained proof in Appendix 
	\ref{app:pisierexpn2}.}
that strong convergence can fail in the subgaussian regime $D=e^{O(N^2)}$;
what happens in between the subexponential and subgaussian regimes 
remains open.

However, the results of the present paper are already optimal in a weaker 
sense: the subexponential regime is the 
largest one that is accessible by trace methods. 
This phenomenon is best illustrated by means of a simple example.
Let $G^N$ be a GUE matrix of dimension $N$, and consider the random matrix 
$$
	X^N = \id_D \otimes G^N.
$$
We aim to show that $\|X^N\|=2+o(1)$, which is obvious due to 
the special structure of the model. However, in 
general we have no way of reasoning directly about the norm;
instead, we bound the norm by a trace statistic such as 
$\mathbf{E}[\trace \chi(X^N)]$ in \eqref{eq:conclosehere}, which is 
amenable to computation.
But in the present example,
$$
	\mathbf{E}[\trace \chi(X^N)] \ge
	\mathbf{E}[\#\{\text{eigenvalues of }X^N\text{ in }
	[-(2+\varepsilon),2+\varepsilon]^c\}] \ge D e^{-cN}
$$
for a universal constant $c$, where we used $\mathbf{P}[\|G^N\|\ge
2+\varepsilon] \ge \mathbf{P}[G^N_{11}\ge 2+\varepsilon] \ge e^{-cN}$ 
and that any eigenvalue of $G^N$ gives rise to an eigenvalue of $X^N$ of 
multiplicity $D$. Thus $\mathbf{E}[\trace \chi(X^N)]=o(1)$ can only occur 
when $D=e^{O(N)}$.

This example shows that the main results of this paper are the best 
possible in the sense that they capture the optimal regime where the 
expected number of outlier eigenvalues is $o(1)$. The reason for this, 
however, is that when matrix coefficients are very high dimensional, 
outlier eigenvalues can appear with very large multiplicity. This does not 
rule out the possibility that strong convergence holds in the 
super-exponential regime, but presents a fundamental obstacle to the 
application of \eqref{eq:conclosehere} or of any other standard trace 
method (e.g., the moment method).

\subsubsection{The optimal rate of convergence}
\label{sec:discoptrate}

While the new ideas of this paper have made it possible to implement most 
of the ingredients of the basic method of \cite{chen2024new} in a nearly 
optimal manner, one significant inefficiency remains: the $N^{-1/2+o(1)}$ 
rate of strong convergence achieved by Corollary \ref{cor:gaussrate} is 
not expected to be optimal. For linear polynomials $P$, the optimal rate
$N^{-2/3}$ follows from classical Tracy-Widom asymptotics, and heuristic 
universality principles of random matrix theory suggest that the same rate 
should extend to arbitrary $P$.

The source of this inefficiency is the parameter $\beta$ in the asymptotic 
expansion \eqref{eq:introexp}. For both Gaussian and Haar-distributed 
ensembles, the methods of this paper yield $\beta=2$. An $N^{-2/3+o(1)}$ 
rate would follow if this could be improved to $\beta=\frac{3}{2}$. In the 
very special
case of linear polynomials $P$ of GUE matrices, such an expansion has in 
fact been established by Haagerup and Thorbj{\o}rnsen \cite{HT12} by using 
explicit differential equations satisfied by the GUE density, lending 
credence to the validity of such an expansion in the more general setting.

In our approach, $\beta$ is ultimately controlled by the Bernstein 
inequality from the analytic theory of polynomials 
\cite{borwein2012polynomials}, which is optimal for arbitrary polynomials. 
The conjectured validity of an improved $\beta$ therefore suggests that 
the polynomials that arise from the function $\Phi_h$ in random matrix 
models are somewhat better behaved than the worst case polynomials. It is 
unclear, however, how the latter may be captured. There are known 
improvements of the Bernstein inequality for special classes of 
polynomials (e.g., those with no roots in the unit disc \cite{Erd40}) that 
would imply $\beta=\frac{3}{2}$ if they were applicable, but numerical 
evidence suggests that the polynomials that arise here do not satisfy the 
requisite assumptions.

\subsubsection{Beyond the classical ensembles}
\label{sec:beyond}

An unexpected feature of the present paper is that the entire analysis of 
the classical Gaussian and Haar-distributed ensembles uses only 
qualitative properties of the rational function $\Phi_h$ and general tools 
such as concentration of measure. Beyond these basic ingredients, no 
problem-specific arguments are used in the analysis. It should be 
emphasized, however, that this simple analysis is enabled by the
serendipitous coincidence of two special properties of the classical 
ensembles: they are of dimension $N$, so that \eqref{eq:suppincl} need 
only be established for $m=2$ to achieve strong convergence; and they 
admit a dual model as in section \ref{sec:supersymm}, which yields the 
latter property automatically.

These properties are by no means needed for the application of our 
approach, and both the methods of \cite{chen2024new} and of the present 
paper are much more broadly applicable. In general, however, it should be 
expected that \eqref{eq:suppincl} must be established at least for small 
$m$ by means of a problem-specific moment computation as in 
\cite{chen2024new}. Such moment estimates are greatly facilitated by a 
technique developed by Magee and de la Salle \cite[\S 
6.2]{magee2024quasiexp}, which shows that it is often possible to reduce 
such estimates to the special case of polynomials $P$ with nonnegative 
coefficients.

Let us finally note that while the unexpected appearance of dual random 
matrix models for the classical ensembles might raise the hope that this 
phenomenon arises more generally, that does not appear to be the case. For 
example, it is unlikely that random permutation models admit dual random 
matrix models, as the Deligne category $\mathrm{Rep}(\mathrm{S}_t)$ is 
semisimple abelian for all $t<0$ \cite[Theorem 2.18 and \S 9.5]{Del07}.

\subsection{Organization of this paper}

The rest of this paper is organized as follows. 

In section \ref{sec:prelim}, we recall some basic definitions and analytic 
tools that will be used throughout the paper. Section 
\ref{sec:optimaluniformbounds} develops an optimal polynomial 
interpolation bound for $\frac{1}{N}$ samples that will be used in all 
results in this paper.

In section \ref{sec:thepolymethodforGUE}, we implement the polynomial 
method to achieve an asymptotic expansion for smooth spectral statistics 
of GUE. This is combined in section \ref{sec:bootstrapping} with a 
bootstrapping argument to prove Theorem \ref{thm:maingauss} in the GUE 
case. The analysis is extended to GOE/GSE matrices in section 
\ref{sec:supersymmetryGOEGSE}, concluding the proof of 
Theorem~\ref{thm:maingauss}. The proof of Theorem \ref{thm:mainhaar} is 
contained in sections \ref{sec:Unitarymatrices} and 
\ref{sec:O(N)andSp(N)}, which develop the corresponding arguments for the 
$\mathrm{U}(N)$ and $\mathrm{O}(N)/\mathrm{Sp}(N)$ models, respectively.

Section \ref{sec:applications} develops applications to subexponential 
operator spaces (Corollary~\ref{cor:gausssubexp}), random permutation 
models, Hayes' model of the Peterson-Thom conjecture, tensor GUE models, 
and high-dimensional representations of $\mathrm{U}(N)$.

The paper concludes with three appendices. Appendix \ref{app:pisierexpn2} 
discusses Pisier's upper bound on the dimension of matrix coefficients for 
which strong convergence can hold. Appendix \ref{app:tensor} develops an 
approximation result for $\Gamma$-independent semicircular families that 
is used in the treatment of tensor models. Appendix \ref{app:magee} 
contains a result of Magee on duality of stable representations of 
$\mathrm{U}(N)$.

\subsection{Notation}

Throughout this paper, $a\lesssim b$ denotes that $a\le Cb$ for a 
universal constant $C>0$. Unless otherwise specified, $C,c>0$ denote 
universal constants that may change from line to line in proofs. We use 
the convention that
$\mathbb{Z}_+:=\{0,1,2,\ldots\}$, $\mathbb{N}:=\{1,2,\ldots\}$, and
$[r]:=\{1,\ldots,r\}$ for $r\in\bN$.

We denote by $\calP$ the space of all real univariate 
polynomials, and by $\calP_q\subset\calP$ the polynomials of degree at 
most $q$. We denote by $h^{(m)}$ the $m$th derivative of $h$, and write 
$\sobnorm{h}{0}{I} := \sup_{x\in I} |h(x)|$ and $\|h\|_{C^m(I)}:= 
\sum_{j=0}^m \|h^{(j)}\|_I$ for $I\subseteq\mathbb{R}$.

We denote by $\mathrm{M}_N(\mathcal{A})$ the space of $N\times N$ matrices 
with entries in $\mathcal{A}$. The unnormalized and normalized traces of 
$M\in \mathrm{M}_N(\bC)$ are denoted as $\trace M$ and $\tr M := 
\frac{1}{N}\trace M$, respectively. The identity matrix or operator is 
denoted as $\id$.

\section{Preliminaries}
\label{sec:prelim}

The aim of this section is to recall a number of basic tools that 
will be used throughout this paper, as well as to recall the precise 
definitions of the classical random matrix ensembles and their limiting 
models.

\subsection{Polynomial inequalities}

If a real polynomial is bounded in a finite interval, we can control its 
derivatives inside that interval and its growth outside the interval. The 
first property is captured by the Bernstein inequality, which plays a 
fundamental role throughout this paper; it replaces the use of the Markov 
brothers inequality in \cite{chen2024new}. We recall here a version for 
higher derivatives.

\begin{lem}[Bernstein inequality]
\label{lem:polybernstein} 
For any $h\in \calP_q$ and $\delta>0$, we have
$$
	|h^{(m)}(x)| \leq 
	\left( \frac{2q}{\delta \sqrt{1-(x/\delta)^2}} \right)^m 
	\sobnorm{h}{0}{[-\delta, \delta]}
	\quad\text{for all }x\in (-\delta, \delta).
$$
\end{lem}

\begin{proof}
The statement is given for $\delta=1$ in 
\cite[p.\ 260]{borwein2012polynomials}, and follows for arbitrary 
$\delta>0$ by a straightforward scaling argument.
\end{proof}

The second property is captured by the following extrapolation lemma.

\begin{lem}
\label{lem:magextrapolation}
For any $h\in \calP_q$ and $K>0$, we have
$$
	|h(x)|   \leq 
	\left(\frac{2|x|}{K}\right)^q  \sobnorm{h}{0}{[-K, K]}
	\quad \text{for all } x\in \bR \backslash [-K, K].
$$
\end{lem}

The proof can be found in \cite[p.\ 247]{borwein2012polynomials}.

\subsection{Some analytic tools}

\subsubsection{Chebyshev expansions}
\label{sec:chebexp}

Let $h\in \calP_q$ and fix $K>0$. Then we can express
\begin{equation}
    \label{eq:chebyshevexpansion}
    h(x) = \sum_{j=0}^q a_j T_j(K^{-1}x)
\end{equation}
for some real coefficients $a_j$, where $T_j$ denotes the Chebyshev 
polynomial of the first kind of degree $j$ defined by 
$T_j(\cos\theta)=\cos(j\theta)$.
The following is classical.

\begin{lem} 
\label{lem:zygmund}
Let $h$ be as in \eqref{eq:chebyshevexpansion} and define
$f(\theta):= h(K \cos(\theta))$. Then
$$
	|a_0| \leq \sobnorm{h}{0}{[-K, K]},
$$
and for every $m\in \bZ_+$
$$
	\sum_{j=1}^q j^m  |a_j| \les \sobnorm{f^{(m+1)}}{0}{[0, 2\pi]}.
$$
\end{lem}

\begin{proof}
Note that $a_j$ in \eqref{eq:chebyshevexpansion} are the Fourier 
coefficients of $f$. Thus the first inequality 
follows from $a_0 = \frac{1}{2\pi}\int_0^{2\pi} f(\theta)\,d\theta$.
The second inequality follows as
$$
	\sum_{j=1}^q j^m |a_j| \le
	\Bigg(\sum_{j=1}^q \frac{1}{j^2}\Bigg)^{1/2}
	\Bigg(\sum_{j=1}^q  j^{2(m+1)} |a_j|^2\Bigg)^{1/2} 
	\lesssim
	\|f^{(m+1)}\|_{L^2[0,2\pi]}
$$
by Cauchy-Schwarz and Parseval, and using that
$\|g\|_{L^2[0,2\pi]}\le \sqrt{2\pi}\|g\|_{[0,2\pi]}$.
\end{proof}

As the bounds of Lemma \ref{lem:zygmund} are independent of $q$, it 
follows that every function $h:[-K,K]\to\mathbb{R}$ with 
$\|h^{(m+1)}\|_{[-K,K]}<\infty$ has a uniformly convergent Chebyshev 
expansion \eqref{eq:chebyshevexpansion} (with $q=\infty$). Thus Lemma 
\ref{lem:zygmund} extends to any such function $h$.

\subsubsection{Taylor expansions}

While Chebyshev expansions are useful for the analysis of smooth 
functions, Taylor expansions are often more convenient for the analysis of 
analytic functions due to the following standard estimate.

\begin{lem}
\label{lem:taylorcoeff}
Let $f:\mathbb{C}\to\mathbb{C}$ be holomorphic in a neighborhood of 
$\{z\in\mathbb{C}:|z|\le r\}$. Then $f(z)=\sum_{k=0}^\infty a_k 
z^k$ is absolutely convergent for $|z|<r$ with
$$
	|a_k| \le r^{-k} \max_{|z|=r} |f(z)|.
$$
\end{lem}

\begin{proof}
The conclusion follows readily by estimating the integrand in the 
Cauchy integral formula
$a_k = \frac{1}{k!}f^{(k)}(0) =
\frac{1}{2\pi i}\oint_{\{|y|=r\}} f(y)\,y^{-(k+1)}\,dy$.
\end{proof}

\subsubsection{Test functions}

The following nearly optimal construction of smooth test functions will be 
used repeatedly throughout this paper.

\begin{lem}
\label{lem:testfunctions}
Fix $m\in \bZ_+$  and $K, \rho, \varepsilon>0$ so that $\rho+\varepsilon < 
K$. Then there exists a function $\chi \in C^\infty(\bR)$ with the 
following properties.
\smallskip
\begin{enumerate}[1.]
\itemsep\abovedisplayskip
\item \label{item:testfunction} 
$\chi(x)=0$ for $|x|\leq \rho+\frac{\varepsilon}{2}$, and $\chi(x)=1$ for 
$|x|\geq \rho+\varepsilon$. 
\item \label{item:differentiability} Let $f(\theta):= \chi( K \cos 
\theta)$. Then for every $k\leq m$, we have
$$
	\sobnorm{f^{(k+1)}}{0}{[0, 2\pi]} \leq 
	8^{k+1} m^k \left( \frac{K}{\varepsilon} \right)^{k+1}.
$$
\end{enumerate}
\end{lem}

\begin{proof}
The result follows readily from the proof of \cite[Lemma 
4.10]{chen2024new}.
\end{proof}

\subsubsection{Distributions}
\label{sec:distributionsprelims} 

Throughout this paper, we only consider distributions on $\mathbb{R}$.
We adopt the following definition; see, e.g., \cite[\S 
2.2--2.3]{hormander2003}.

\begin{defn} 
\label{def:distribution}
A linear functional $\nu$ on $C^{\infty}(\bR)$ is called a \emph{compactly 
supported distribution} if there exist $C,K\geq 0$ and $m\in \bZ_+$ so 
that
$$
	|\nu(f)| \leq 
	C\max_{0\leq k \leq m} \sobnorm{f^{(k)}}{0}{[-K, K]}
	\quad\text{for all }f\in C^{\infty}(\bR).
$$
The \emph{support} $\supp\nu$ of a compactly
supported distribution $\nu$ is the smallest closed set 
$A\subseteq\mathbb{R}$ so that $\nu(f)=0$ for all $f\in C^{\infty}(\bR)$
that vanish in a neighborhood of $A$.
\end{defn}

The linear functionals that arise in this paper are defined \emph{a 
priori} only on the space $\mathcal{P}$ of real polynomials. The following 
criterion enables us to extend these functionals to compactly supported 
distributions, cf.\ \cite[Lemma 4.7]{chen2024new}.

\begin{lem}
\label{lem:extensiontoadistribution}
Let $\nu$ be a linear functional on $\calP$. If there exist
$C,K,m \geq 0$ so that
$$
	|\nu(h)| \leq C q^m \sobnorm{h}{0}{[-K, K]}
	\quad\text{for all } h \in \calP_q,~q\in\bN,
$$
then $\nu$ extends to a compactly supported distribution
with $|\nu(h)| \lesssim \|h\|_{C^{m+1}[-K,K]}$.
\end{lem}

\subsection{Random matrices and asymptotic freeness} 
\label{sec:rmtandfree}

\subsubsection{Unitary and orthogonal invariant ensembles}

We begin by recalling the definitions of the standard complex and real 
Gaussian ensembles. We denote the real Gaussian distribution as 
$N(0,\sigma^2)$, and define the complex Gaussian distribution 
$N_{\mathbb{C}}(0,\sigma^2)$ as the distribution of $\xi_1+i\xi_2$ where 
$\xi_1,\xi_2$ are i.i.d.\ $N(0,\frac{\sigma^2}{2})$.

\begin{defn}
Let $X$ be an $N\times N$ self-adjoint random matrix with independent 
entries $(X_{ij})_{i\le j}$ on and above the diagonal.
\begin{enumerate}[a.]
\itemsep\smallskipamount
\item
$X$ is called a \emph{GUE matrix} if $X_{ij}\sim 
N_\mathbb{C}(0,\frac{1}{N})$ 
for $i\ne j$ and $X_{ii}\sim N(0,\frac{1}{N})$.
\item
$X$ is called a \emph{GOE matrix} if
$X_{ij}\sim N(0,\frac{1}{N})$ for $i\ne j$ and
$X_{ii}\sim N(0,\frac{2}{N})$.
\end{enumerate}
\end{defn}

The defining property of GUE and GOE models is that they are the Gaussian 
ensembles whose distributions are invariant under conjugation by unitary 
and orthogonal matrices, respectively. Aside from the Gaussian ensembles, 
we will develop parallel results for random unitary and orthogonal matrices 
drawn from the normalized Haar measure on $\mathrm{U}(N)$ and 
$\mathrm{O}(N)$, respectively.

\subsubsection{Symplectic invariant ensembles}

To define the symplectic analogues of the above ensembles,
we must first recall some basic facts.

We denote by $\bH$ the skew-field of quaternions. Recall that 
$z\in\bH$ is represented as 
$z=z_0\mathbf{1}+z_1\mathbf{i}+z_2\mathbf{j}+z_3\mathbf{k}$, where 
$z_i\in\bR$ and $\mathbf{i},\mathbf{j},\mathbf{k}$ satisfy the relations
$$
	\mathbf{i}^2=\mathbf{j}^2=\mathbf{k}^2=\mathbf{ijk}=-\mathbf{1}.
$$
The conjugate is $\bar z = 
z_0\mathbf{1}-z_1\mathbf{i}-z_2\mathbf{j}-z_3\mathbf{k}$, and the real 
part is $\mathrm{Re}\,z=z_0$.

For a quaternionic matrix $A\in\mathrm{M}_N(\bH)$, the 
adjoint $A^*$ is defined as the conjugate transpose as for complex 
matrices. We denote by
$$
	\mathrm{Sp}(N) := \{U\in \mathrm{M}_N(\bH):
	UU^*=U^*U=\id\}
$$
the group of $N\times N$ symplectic (i.e., quaternionic unitary) matrices. 

We define the quaternionic Gaussian distribution $N_\bH(0,\sigma^2)$ as 
the distribution of 
$\xi_0\mathbf{1}+\xi_1\mathbf{i}+\xi_2\mathbf{j}+\xi_3\mathbf{k}$ where 
$\xi_0,\ldots,\xi_3$ are i.i.d.\ $N(0,\frac{\sigma^2}{4})$. We can now 
recall the definition of the standard quaternionic Gaussian ensemble.

\begin{defn} 
An $N\times N$ self-adjoint random matrix with 
independent entries $(X_{ij})_{i\ge j}$ is
called a \emph{GSE matrix} if $X_{ij}\sim N_\bH(0,\frac{1}{N})$ for $i\ne 
j$ and $X_{ii}\sim N(0,\frac{1}{2N})$.
\end{defn}

The defining property of the GSE model is that it is the Gaussian ensemble 
whose distribution is invariant under conjugation by symplectic matrices. 
We will develop parallel results for random symplectic matrices drawn from 
the normalized Haar measure on the compact group $\mathrm{Sp}(N)$.

For the purposes of linear algebra, working directly with quaternions
is somewhat awkward; for example, we cannot apply
noncommutative polynomials with complex coefficients to them, as 
the quaternions form an algebra over the reals. Instead, we will identify 
$\mathbb{H}$ with the subring of $\mathrm{M}_2(\bC)$ generated by 
$$
	\mathbf{1} = \begin{bmatrix} 1 & 0\\ 0 & 1 \end{bmatrix},
	\qquad
	\mathbf{i} = \begin{bmatrix} i & 0\\ 0 & -i \end{bmatrix},
	\qquad
	\mathbf{j} = \begin{bmatrix} 0 & 1\\ -1 & 0 \end{bmatrix},
	\qquad
	\mathbf{k} = \begin{bmatrix} 0 & i \\ i & 0 \end{bmatrix}.
$$
In this manner, $\mathrm{M}_N(\mathbb{H})$ is naturally identified with a 
subring of $\mathrm{M}_{2N}(\bC)$. In this paper, we will always interpret 
linear algebra operations on $M\in\mathrm{M}_N(\mathbb{H})$ as being 
applied to the associated complex representations; for example, $\trace 
M$ will denote the trace of the $2N$-dimensional complex representation of 
$M$.

\subsubsection{Asymptotic freeness}

For the purposes of this paper, a $C^*$-probability space 
$(\mathcal{A},\tau)$ is defined by a
unital $C^*$-algebra $\mathcal{A}$ and a faithful trace $\tau$.

\begin{defn}
Let $(\mathcal{A},\tau)$ be a $C^*$-probability space.
\begin{enumerate}[a.]
\itemsep\smallskipamount
\item $s_1,\ldots,s_r\in\mathcal{A}$ form a \emph{free semicircular 
family} if the spectral distribution of each $s_i$ is the standard 
semicircle distribution, and $s_1,\ldots,s_r$ are freely independent.
\item $u_1,\ldots,u_r\in\mathcal{A}$ are \emph{free Haar unitaries} if
the spectral distribution of each $u_i$ is uniformly distributed on the 
unit circle, and $u_1,\ldots,u_r$ are freely independent.
\end{enumerate}
\end{defn}

We do not recall here the definition of free independence, which is 
discussed in detail in the excellent text \cite{NS06}. The significance of 
the above definition is that it provides a limiting model as $N\to\infty$ 
for many random matrix models.

\begin{lem}[Weak asymptotic freeness]
\label{lem:waf}
Let $\boldsymbol{G}^N=(G_1^N,\ldots,G_r^N)$ be i.i.d.\ GUE/GOE/GSE random 
matrices of dimension $N$, and let $\boldsymbol{s}=(s_1,\ldots,s_r)$ be a 
free semicircular family. Then
$$
	\lim_{N\to\infty}
	\mathbf{E}[\tr P(\boldsymbol{G}^N)] =
	(\tr\otimes\tau)[P(\boldsymbol{s})]
$$
for every noncommutative polynomial
$P\in\mathrm{M}_D(\bC)\otimes\bC\langle\boldsymbol{s}\rangle$.

Similarly, let $\boldsymbol{U}^N=(U_1^N,\ldots,U_r^N)$ be i.i.d.\ 
Haar-distributed random matrices in 
$\mathrm{U}(N)/\mathrm{O}(N)/\mathrm{Sp}(N)$, and let 
$\boldsymbol{u}=(u_1,\ldots,u_r)$ be free Haar unitaries. Then
$$
	\lim_{N\to\infty}
	\mathbf{E}[\tr P(\boldsymbol{U}^N,\boldsymbol{U}^{N*})] =
	(\tr\otimes\tau)[P(\boldsymbol{u},\boldsymbol{u}^*)]
$$
for every noncommutative polynomial  
$P\in\mathrm{M}_D(\bC)\otimes\bC\langle\boldsymbol{u},\boldsymbol{u}^*\rangle$.
\end{lem}

Lemma \ref{lem:waf} is a special case of a celebrated result of Voiculescu 
\cite{voiculescu1991limit} (Voiculescu does not consider the symplectic 
ensembles, but the proofs are entirely analogous).

\section{Polynomial interpolation from $\frac{1}{N}$ samples}
\label{sec:optimaluniformbounds}

It is classical \cite{CR92} that for $h\in\mathcal{P}_q$, we have 
$\|h\|_{[0,\delta]}\lesssim \max_{x\in I}|h(x)|$ for any set 
$I\subseteq[0,\delta]$ with spacing at most $O(\frac{\delta}{q^2})$ 
between its points; this is optimal for general sets $I$. When applied to 
the set $I_M :=\{\frac{1}{N}:N\ge M\}$ that arises in random matrix 
problems, this enables us to bound $\|h\|_{[0,\delta]}$ only for 
$\delta=O(\frac{1}{q^2})$. The aim of this section is to prove that a much 
better bound can be achieved in this case.

\begin{prop}[Interpolation from $\frac{1}{N}$ samples]
\label{prop:uniform_bound_1/N}
We have
$$
        \sobnorm{h}{0}{[0, \delta]} \le C
	\sup_{\frac{1}{N}\le 2\delta} |h(\tfrac{1}{N})|
$$    
for every $q\in\mathbb{N}$, $h\in \calP_q$, and $0 \le \delta \le 
\frac{1}{24q}$, where $C$ is a universal constant.
\end{prop}

This is optimal up to the values of the constants.

\begin{example}[Optimality]
Let $h_q(x)=T_q(qx)\prod_{j=1}^q (1-jx)$, where $T_q$ is
the Chebyshev polynomial of the first kind of degree $q$. Then 
$h_q\in\mathcal{P}_{2q}$ and
$|h_q(\frac{1}{N})|\le 1$ for all $N\ge 1$.
But for $x=\frac{1}{m+1/2}$ with
$m\in\{1,\ldots,q-1\}$, we can estimate
$$
	|h_q(x)| = 
	\frac{\big|T_q\big(\frac{q}{m+1/2}\big)\big|}{(4m+2)^q}
	\frac{(2q)!}{q!} \frac{{q\choose m}}{{2q\choose 2m}}
	\ge
	\bigg(\frac{Cq}{m}\bigg)^q
$$
for a universal constant $C$, where we used $|T_q(x)|\ge\frac{1}{2}x^q$ 
for $x\ge 1$. Thus there is a universal constant $c$ so that 
$\|h_q\|_{[0,\frac{c}{q}]} \ge 2^q$ for all
$q$. This shows that the conclusion of Proposition 
\ref{prop:uniform_bound_1/N} must fail if the assumption 
$\delta\le\frac{1}{24q}$ is replaced by $\delta\le\frac{2c}{q}$.
\end{example}

Proposition \ref{prop:uniform_bound_1/N} is based on a powerful result of 
Rakhmanov \cite{rakhmanov2007bounds}: if $h\in\mathcal{P}_q$ is bounded at 
equispaced points $I$ in the interval $[-1,1]$, then $O(\frac{1}{q})$ 
spacing suffices to achieve a uniform bound on $h$ strictly in the 
interior of the interval, even though $O(\frac{1}{q^2})$ spacing is 
necessary for a uniform bound in the entire interval. The distinct 
behavior in the interior and near the edges of the interval is reminiscent 
of the distinction between the Bernstein and Markov inequalities.

\begin{thm}[Rakhmanov]
\label{thm:Rakhmanov}
Let $M, q\in \bN$ with $q \leq M$, and let $h\in \calP_q$.
Then
$$
	\sobnorm{h}{0}{[-\frac{1}{2}, \frac{1}{2}]} 
	\leq C \max_{k=1,\ldots,2M} 
	\left| h\left( -1 + \tfrac{2k-1}{2 M} \right) \right|,
$$
where $C$ is a universal constant.
\end{thm}

\begin{proof}
Apply \cite[eq.\ (1.6)]{rakhmanov2007bounds}  with $\delta = 1$ 
and $N=2M$. 
\end{proof}

The setting of Proposition \ref{prop:uniform_bound_1/N} differs 
considerably from that of Rakhmanov's theorem: the $\frac{1}{N}$ samples 
are highly nonuniform in the interval $[0,\delta]$, while we aim to bound 
$h$ near the endpoint $0$ of the interval rather than strictly in its 
interior. However, the fact that the $\frac{1}{N}$ samples become 
increasingly dense near $0$ will enable us to apply Rakhmanov's theorem 
in a multiscale manner: we can cover $[0,\delta]$ by a sequence of 
intervals so that the $\frac{1}{N}$ samples in each interval are 
approximately uniform, and apply Theorem \ref{thm:Rakhmanov} to each 
interval.

The main difficulty in the proof is that the samples in each 
interval must be mapped to equispaced samples in order to apply Theorem 
\ref{thm:Rakhmanov}. Concretely, suppose we aim to bound
$h\in\mathcal{P}_q$ based on its values at the $2M$ sample points
$$
	\left\{\tfrac{1}{2M+2k-1}: k=1, \ldots, 2M \right\}.
$$
As we have
$$
	h\big(\tfrac{1}{2M+2k-1}\big) =
	r\big(-1 + \tfrac{2k-1}{2M}\big)
	\quad\text{with}\quad
	r(x) :=
	h\big(\tfrac{1}{2M(2+x)}\big),
$$
the problem is equivalent to bounding the rational function $r$ at 
equispaced points as in Theorem \ref{thm:Rakhmanov}. However $r$ is no 
longer a polynomial, so that Theorem \ref{thm:Rakhmanov} does not apply. 
To surmount this issue, we will use that $r$ can be approximated by a 
polynomial while only losing a constant factor in its magnitude and 
degree.

\begin{lem}
\label{lem:polyapprox}
Let $M,q\in\bN$, let $h\in\mathcal{P}_q$, and define $r(x) 
:= h\big(\tfrac{1}{2M(2+x)}\big)$. Then there exists a polynomial
$g\in\mathcal{P}_{8q}$ so that
$$
	\tfrac{4}{7}|g(x)| \le |r(x)| \le 4|g(x)|\quad\text{for all }x\in[-1,1].
$$
\end{lem}

\begin{proof}
We can clearly write $r(x) = \frac{u(x)}{(2+x)^q}$ for a polynomial
$u\in\mathcal{P}_q$. Note that
$$
	\big|\tfrac{1}{(2+z)^q}\big| \le 2^q
	\quad\text{for all }z\in\mathbb{C},~|z|\le \tfrac{3}{2}.
$$
Thus $\tfrac{1}{(2+z)^q}=\sum_{k=0}^\infty a_k z^k$ with $|a_k| \le
2^q (\frac{3}{2})^{-k}$ by Lemma \ref{lem:taylorcoeff}. Therefore
$$
	\big|\tfrac{1}{(2+x)^q} - t(x)|
	\le
	\sum_{k=7q+1}^\infty |a_k| 
	\le 
	4^{-q}
	\le
	\tfrac{3}{4} \tfrac{1}{(2+x)^q}
	\quad\text{for }x\in[-1,1],
$$
where we defined $t(x):=\sum_{k=0}^{7q} a_k x^k$ and we used that
$3^{-q}\le\tfrac{1}{(2+x)^q}$ for $|x|\le 1$.
In particular, we have shown that 
$\frac{4}{7}t(x)\le\tfrac{1}{(2+x)^q} 
\le 4t(x)$ for all $x\in[-1,1]$, and
the conclusion follows readily by choosing
$g(x)=u(x)t(x)$.
\end{proof}

We can now complete the proof of Proposition \ref{prop:uniform_bound_1/N}.

\begin{proof}[Proof of Proposition \ref{prop:uniform_bound_1/N}]
Fix $q\in\mathbb{N}$ and $h\in\mathcal{P}_q$ throughout the proof.
We first apply Theorem 
\ref{thm:Rakhmanov} to the polynomial $g$ in Lemma \ref{lem:polyapprox}
to estimate
$$
	\|h\|_{[\frac{1}{5M},\frac{1}{3M}]} 
	\le
	C\max_{k=1,\ldots,2M} 
	\big|h\big(\tfrac{1}{2M+2k-1}\big)\big|
$$
for any $M\in\mathbb{N}$ with $M\ge 8q$, where $C$ is a universal 
constant.
Now let $m = \lfloor \frac{1}{3\delta}\rfloor$. Then $m\ge 8q$ by the 
assumption $\delta\le \frac{1}{24q}$, and it is readily verified that 
$$
	(0,\delta]\subseteq (0,\tfrac{1}{3m}] 
	= \bigcup_{M\ge m}[\tfrac{1}{5M},\tfrac{1}{3M}]
$$
as there are no gaps between the intervals 
$[\frac{1}{5M},\frac{1}{3M}]$ for $M\ge 2$. We therefore obtain
$$
	\|h\|_{[0,\delta]} \le
	\sup_{M\ge m}
	\|h\|_{[\frac{1}{5M},\frac{1}{3M}]}
	\le
	C\sup_{N\ge 2m+1}
	\big|h\big(\tfrac{1}{N}\big)\big|,
$$
and the conclusion follows as $\frac{1}{2m+1}\le 2\delta$ (because
$2m+1 \ge \frac{2}{3\delta} - 1 \ge \frac{1}{2\delta}$).
\end{proof}

\section{Asymptotic expansion for GUE}
\label{sec:thepolymethodforGUE}

The aim of this section is to establish an asymptotic expansion of smooth 
trace statistics of polynomials of GUE matrices. This expansion will be 
used in section~\ref{sec:bootstrapping} to prove Theorem 
\ref{thm:maingauss} in the GUE case, while the requisite modifications in 
the case of GOE/GSE matrices will be developed in section 
\ref{sec:supersymmetryGOEGSE}.

The following will be fixed throughout this section. Let 
$\boldsymbol{G}^N=(G_1^N,\ldots,G_r^N)$ be independent GUE matrices of 
dimension $N$, and let $\boldsymbol{s}=(s_1,\ldots,s_r)$ be a free 
semicircular family. We will further fix a self-adjoint\footnote{%
	In the present context, we view $x_1,\ldots,x_r$ as 
	\emph{self-adjoint} noncommuting variables. Thus, for
	example, $P(x_1,x_2)=A\otimes x_1x_2 + A^*\otimes x_2x_1$
	is a self-adjoint noncommutative polynomial.
} 
noncommutative 
polynomial $P\in \mathrm{M}_D(\bC) \otimes \bC \langle 
x_1,\ldots,x_r\rangle$ of degree $q_0$ with matrix coefficients of 
dimension $D$. For simplicity of notation, we will denote by
$$
	\xn := P(\boldsymbol{G}^N),\qquad\quad
	X_\mathrm{F} := P(\boldsymbol{s})
$$
the random matrix of interest and its limiting model.

The main result of this section is as follows.

\begin{thm}[Smooth asymptotic expansion for GUE]
\label{thm:smasympexGUE} 
There exist universal constants $C,c>0$, and a compactly supported 
distribution $\nu_k$ for every $k\in\bZ_+$, such that the following 
holds. Fix any bounded $h\in C^\infty(\bR)$, and define
$$
	f(\theta):= h(K\cos\theta)
	\qquad\text{with}\qquad
	K := (Cr)^{q_0}\|X_{\rm F}\|.
$$
Then for every $m, N\in \bN$ with $m\le\frac{N}{2}$, we have
\begin{multline*}
       \left| \mathbf{E}[\tr h(\xn)]
	- \sum_{k=0}^{m-1} \frac{\nu_k(h)}{N^k} \right| 
       \le
	\frac{(Cq_0)^{2m}}{m! N^m} \|f^{(2m+1)}\|_{[0, 2\pi]} 
	\\ +
	Cre^{-cN}\big(\sobnorm{h}{0}{(-\infty,\infty)}
		+ \sobnorm{f^{(1)}}{0}{[0, 2\pi]} \big).
\end{multline*}
\end{thm}

\begin{rem}
The conclusion of Theorem \ref{thm:smasympexGUE} extends readily by 
continuity to any test function $h\in C_b(\mathbb{R})$ so that 
$\|f^{(2m+1)}\|_{[0, 2\pi]}<\infty$. In particular, the proof of Theorem 
\ref{thm:maingauss} will apply this theorem to the test functions provided 
by Lemma \ref{lem:testfunctions}. This observation will be used without 
further comment in the sequel.
\end{rem}

The remainder of this section is devoted to the proof of Theorem 
\ref{thm:smasympexGUE}.

\subsection{A priori bounds}
\label{sec:aprioriGUE}

The general approach to Theorem \ref{thm:smasympexGUE} follows the basic 
method outlined in section \ref{sec:introreview}. However, for the 
Gaussian ensembles, a significant complication arises from the 
unboundedness of the Gaussian distribution. To surmount this issue, we 
begin by proving a priori bounds that will be used to truncate the model. 
The challenge in the remainder of the proof will be to apply these bounds 
without incurring any quantitative loss.

\begin{lem}[A priori bounds]
\label{lem:apriori}
There exist universal constants $C, c>0$ such that 
\begin{equation}
\label{eq:XNlessthanK}
	\mathbf{P}[\|\xn\|>K] \le Cr e^{-cN},
\end{equation}
where $K:=(Cr)^{q_0}\|X_{\rm F}\|$. Moreover, we have
\begin{equation}
\label{eq:firstapriori}
	| \mathbf{E}[ \tr h(\xn)]| \le 2\sobnorm{h}{0}{[-K,K]}
\end{equation}
and
\begin{equation}
\label{eq:secondapriori}
	|\mathbf{E}[\tr  h(\xn) \cdot 1_{\{\|\xn\|>K \}} ] | 
	\leq C \sqrt{r} e^{-cN} \sobnorm{h}{0}{[-K, K]}
\end{equation}
for every $h \in \calP_q$ with $q\le \frac{N}{q_0}$.
\end{lem}

The remainder of this section is devoted to the proof of this result. We 
begin by recalling a crude tail bound on the norm of GUE matrices.

\begin{lem}
\label{lem:subgaussiannorm}
There exist universal constants $C,c,\kappa>0$ such that for any   
GUE matrix $\gue$ of dimension $N$, we have
$$
	\mathbf{P}[\|\gue\| \geq  \kappa + t] 
	\leq C e^{-cNt^2}\quad\text{for all }t\ge 0,
$$
and
$$
	\mathbf{E}[\|\gue\|^p] \le 
	\left(\kappa+ C \sqrt{\frac{p}{N}}\right)^p
	\quad\text{for all }p\in\bN.
$$
\end{lem}

\begin{proof}
The first inequality follows from a simple $\varepsilon$-net argument 
\cite[\S 2.3.1]{Tao12}. The second inequality follows directly using
$\mathbf{E}[\|\gue\|^p]^{\frac{1}{p}} \le \kappa + 
\mathbf{E}[(\|\gue\|-\kappa)_+^p]^{\frac{1}{p}}$ and integrating the
first inequality
(see, e.g., \cite[Proposition 1.10]{Led01}).
\end{proof}

To proceed, it will be useful to choose a convenient representation of the 
noncommutative polynomial $P$. To this end, denote by $U_j$ the Chebyshev 
polynomial of the second kind of degree $j$ defined by 
$U_j(\cos\theta)\sin\theta = \sin((j+1)\theta)$. Moreover, define the 
noncommutative polynomial $U_{\boldsymbol{i},\boldsymbol{j}}\in
\bC\langle x_1,\ldots,x_r\rangle$ as
$$
	U_{\boldsymbol{i},\boldsymbol{j}}(x_1,\ldots,x_r) :=
	U_{j_1}(\tfrac{1}{2}x_{i_1})U_{j_2}(\tfrac{1}{2}x_{i_2})
	\cdots U_{j_k}(\tfrac{1}{2}x_{i_k})
$$
for every $k\ge 0$, $\boldsymbol{i}=(i_1,\ldots,i_k)$, and 
$\boldsymbol{j}=(j_1,\ldots,j_k)$ such that $j_1,\ldots,j_k\in\mathbb{N}$ 
and $i_1,\ldots,i_k\in [r]$ with $i_1\ne i_2,i_2\ne i_3,\ldots, 
i_{k-1}\ne i_k$ (where
$U_{\boldsymbol{i},\boldsymbol{j}}(\boldsymbol{x}):=\id$ for $k=0$).
Then we can represent $P$ uniquely as
$$
	P(x_1,\ldots,x_r) = \sum_{\boldsymbol{i},\boldsymbol{j}}
	A_{\boldsymbol{i},\boldsymbol{j}}\otimes
	U_{\boldsymbol{i},\boldsymbol{j}}(x_1,\ldots,x_r),
$$
where $A_{\boldsymbol{i},\boldsymbol{j}}\in\mathrm{M}_D(\bC)$ are matrix 
coefficients and the sum ranges over all $\boldsymbol{i},\boldsymbol{j}$ 
as above with $0\le k\le q_0$ and $j_1+\cdots+j_k\le q_0$.
The significance of this representation is that when 
$U_{\boldsymbol{i},\boldsymbol{j}}$ are applied to a 
free semicircular family $\boldsymbol{s}$, the operators
$\{U_{\boldsymbol{i},\boldsymbol{j}}(\boldsymbol{s})\}$ form an 
orthonormal system in $L^2(\tau)$,
cf.\ \cite[\S 5.1]{BS98}. The latter enables us to bound the norms of 
the coefficients $\|A_{\boldsymbol{i},\boldsymbol{j}}\|$ by $\|\xf\|$.

\begin{lem}
\label{lem:chebsecondkind}
For every $\boldsymbol{i},\boldsymbol{j}$ as above
and self-adjoint operators $x_1,\ldots,x_r$, we have
\begin{equation}
\label{eq:chebsecondkind1}
	\|U_{\boldsymbol{i},\boldsymbol{j}}(x_1,\ldots,x_r)\| \le
	2^k
	(\|x_{i_1}\|^{j_1}\vee 1)
	(\|x_{i_2}\|^{j_2}\vee 1)\cdots
	(\|x_{i_k}\|^{j_k}\vee 1).
\end{equation}
Moreover, there is a universal constant $C$ so that
\begin{equation}
\label{eq:chebsecondkind2}
	\sum_{\boldsymbol{i},\boldsymbol{j}} 
	\|A_{\boldsymbol{i},\boldsymbol{j}}\|
	\le
	(Cr)^{q_0} \|X_{\rm F}\|.
\end{equation}
\end{lem}

\begin{proof}
Note that $|T_j(\frac{1}{2}x)|\le |x|^j\vee 1$ for all 
$x\in\mathbb{R}$ by Lemma 
\ref{lem:magextrapolation}, where $T_j$ is the Chebyshev polynomial of the 
first kind. As $U_j(x)=\sum_{a=0}^j x^a T_{j-a}(x)$
\cite[p.\ 37]{borwein2012polynomials}, 
$$
	|U_j(\tfrac{1}{2}x)| \le
	\sum_{a=0}^j 2^{-a}(|x|^j\vee |x|^a) 
	\le
	2(|x|^j\vee 1)
$$
for all $x\in\mathbb{R}$. Thus \eqref{eq:chebsecondkind1} follows 
directly.

Next, note that as $X_{\rm F}:=P(\boldsymbol{s})$
and
$\{U_{\boldsymbol{i},\boldsymbol{j}}(\boldsymbol{s})\}$ are orthonormal
in $L^2(\tau)$, we have
$$
	\|X_{\rm F}\|^2 \ge
	\|(\mathrm{id}\otimes{\tau})(X_{\rm F}^*X_{\rm F})\|
	= \Bigg\|
	\sum_{\boldsymbol{i},\boldsymbol{j}} 
	A_{\boldsymbol{i},\boldsymbol{j}}^*
	A_{\boldsymbol{i},\boldsymbol{j}}
	\Bigg\|
	\ge 
	\max_{\boldsymbol{i},\boldsymbol{j}}
	\|A_{\boldsymbol{i},\boldsymbol{j}}\|^2.
$$
Now \eqref{eq:chebsecondkind2} follows as there are at most $(Cr)^{q_0}$ 
terms in the sum.
\end{proof}

We are now ready to prove Lemma \ref{lem:apriori}.

\begin{proof}[Proof of Lemma \ref{lem:apriori}]
Let $\kappa$ be as in Lemma \ref{lem:subgaussiannorm}.
If $\|\gue_i\|\le \kappa+1$ for all $i$, then
$$
	\|\xn\| \le 
	2^{q_0}
	(\kappa+1)^{q_0}
	\sum_{\boldsymbol{i},\boldsymbol{j}}
	\|A_{\boldsymbol{i},\boldsymbol{j}}\|
	\le
	(Cr)^{q_0}\|X_{\rm F}\| =: K
$$
for a universal constant $C$ 
by the triangle inequality and Lemma \ref{lem:chebsecondkind}.
As
$$
	\mathbf{P}[\|\gue_i\|> \kappa+1\text{ for some }i\in[r]]
	\le
	Cr e^{-cN}
$$
by the union bound and the first inequality of Lemma
\ref{lem:subgaussiannorm}, we proved \eqref{eq:XNlessthanK}.

To proceed, fix $h\in\mathcal{P}_q$ with $q\le \frac{2N}{q_0}$. Note
first that
$$
	\max_{\boldsymbol{i},\boldsymbol{j}}
	\E[\|U_{\boldsymbol{i},\boldsymbol{j}}(\boldsymbol{G}^N)\|^q]
	\le
	2^{qq_0} \E[\|\gue\|^{qq_0}\vee 1] \le
	C_1^{qq_0}
$$
for a universal constant $C_1$
using Lemma \ref{lem:chebsecondkind}, H\"older's inequality, and the 
second inequality of Lemma \ref{lem:subgaussiannorm} (here we used that
$q\le \frac{2N}{q_0}$).
Now define the constant
$L:=2C_1^{q_0}\sum_{\boldsymbol{i},\boldsymbol{j}}\|A_{\boldsymbol{i},\boldsymbol{j}}\|$,
and apply Lemma \ref{lem:magextrapolation} to estimate
$$
	|\mathbf{E}[\tr h(\xn)]| \le
	\mathbf{E}\bigg[\sup_{|x|\le\|\xn\|}|h(x)|\bigg] \le
	\bigg(1 + \E\bigg[\bigg(\frac{ 2\|\xn\|}{L}\bigg)^q\bigg]
	\bigg)\sobnorm{h}{0}{[-L,L]}.
$$
As 
$$
	\E\bigg[\bigg(\frac{2\|\xn\|}{L}\bigg)^q\bigg]
	\le
	\frac{1}{C_1^{qq_0}}\,
	\E\left[\left(
	\frac{\sum_{\boldsymbol{i},\boldsymbol{j}}
	\|A_{\boldsymbol{i},\boldsymbol{j}}\|\,
	\|U_{\boldsymbol{i},\boldsymbol{j}}(\boldsymbol{G}^N)\|}
	{\sum_{\boldsymbol{i},\boldsymbol{j}}
	\|A_{\boldsymbol{i},\boldsymbol{j}}\|}
	\right)^q\right]
	\le
	1
$$
by Jensen's inequality, we have shown that $|\mathbf{E}[\tr h(\xn)]|\le 
2\|h\|_{[-L,L]}$. To prove \eqref{eq:firstapriori}, it remains to note 
that $L\le K$ by Lemma \ref{lem:chebsecondkind}, provided that the 
universal constant $C$ 
in the definition of $K$ is chosen sufficiently large.

Finally, we now suppose $h\in\mathcal{P}_q$ and $q\le \frac{N}{q_0}$. Then 
we can estimate
\begin{align*}
	|\E[\tr h(\xn) \cdot 1_{\{\|\xn\|>K \}} ] |
	&\le
	\E[\tr \, h(\xn)^2 ]^{\frac{1}{2}} \, \P[\|\xn\|>K]^{\frac{1}{2}} 
\\	&\le
	\sqrt{2Cr}\,e^{-cN/2} \sobnorm{h}{0}{[-K, K]}
\end{align*}
by Cauchy-Schwarz, \eqref{eq:XNlessthanK}, and \eqref{eq:firstapriori},
where we used that 
$h^2\in\mathcal{P}_{2q}$ and $2q\le\frac{2N}{q_0}$.
Redefining the constants concludes the proof of \eqref{eq:secondapriori}.
\end{proof}

\subsection{The master inequality}

Our next aim is to prove a form of the asymptotic expansion of Theorem 
\ref{thm:smasympexGUE} for polynomial test functions $h$.

\begin{lem}[Master inequality]
\label{lem:ratesofconv}
There exists a linear functional $\nu_m$ on $\calP$ for every $m\in\bZ_+$ 
such that for every $q\in\bN$ and $h\in \calP_q$
\begin{equation}
\label{eq:infinitesimalbound}
	|\nu_m(h)| \le \frac{(Cqq_0)^{2m}}{m!} 
	\sobnorm{h}{0}{[-K, K]},
\end{equation}
and such that if in addition $q \leq \frac{N}{Cq_0}$, we have
\begin{equation}
\label{eq:rateofconvergence}
	\left| \E[\tr h(\xn)]- \sum_{k=0}^{m-1} \frac{\nu_k(h)}{N^k} 
	\right| \leq \frac{(Cqq_0)^{2m}}{m! N^{m}} 
	\sobnorm{h}{0}{[-K, K]}.
\end{equation}
Here $C$ is a universal constant and $K:=(Cr)^{q_0}\|X_{\rm F}\|$.
\end{lem}

The proof is a straightforward application of the polynomial method of
\cite{chen2024new}. We exploit the well-known fact that 
$\E[\tr h(\xn)]$ is a polynomial of $\frac{1}{N}$.

\begin{lem}[Polynomial encoding]
\label{lem:polyencoding}
For any $h\in \calP_q$, there is $\Phi_h\in\calP_{qq_0}$
so that
$$
	\E[\tr (h(\xn))] = \Phi_h(\tfrac{1}{N}) =
	\Phi_h(-\tfrac{1}{N}).
$$
\end{lem}

\begin{proof}
This is an immediate consequence of the genus expansion for GUE, which 
states that $\E[\tr \gue_{i_1}\cdots\gue_{i_k}]$ is a polynomial of 
$\frac{1}{N^2}$ of degree at most $\frac{k}{4}$ for every
$i_1,\ldots,i_k\in[r]$;
for example, see the proof of \cite[\S 1.10, Lemma 9]{mingo2017free}.
\end{proof}

We can now prove Lemma \ref{lem:ratesofconv}.

\begin{proof}[Proof of Lemma \ref{lem:ratesofconv}]
Fix $q\in\bN$ and $h\in \calP_q$.
Throughout the proof, we adopt the notation of Lemma 
\ref{lem:polyencoding} without further comment.
Lemma \ref{lem:apriori} implies that
$$
	|\Phi_h(\tfrac{1}{N})| \leq 2 \sobnorm{h}{0}{[-K, K]}
	\quad\text{for }N\ge qq_0.
$$
Thus Proposition \ref{prop:uniform_bound_1/N} yields
$$
	\sobnorm{\Phi_h}{0}{[0, \delta]} \le
	C\sobnorm{h}{0}{[-K, K]}
$$
with $\delta := \frac{1}{24qq_0}$. Since $\Phi_h(-\tfrac{1}{N}) = 
\Phi_h(\tfrac{1}{N})$, we obtain the same bound for
$\sobnorm{\Phi_h}{0}{[-\delta,0]}$. We can therefore apply 
Bernstein's inequality (Lemma \ref{lem:polybernstein}) to estimate
\begin{equation}
\label{eq:guemasterbern}
	\|\Phi_h^{(m)}\|_{[-\frac{\delta}{2},\frac{\delta}{2}]} \le
	(Cqq_0)^{2m} \sobnorm{h}{0}{[-K, K]}
\end{equation}
for all $m\in\bZ_+$, where $C$ is a universal constant. Now define
$$
	\nu_m(h) := \frac{\Phi_h^{(m)}(0)}{m!},
$$ 
so that
$$
	\left| \E[\tr h(\xn)]- \sum_{k=0}^{m-1} \frac{\nu_k(h)}{N^k} 
	\right| = 
	\left| \Phi_h(\tfrac{1}{N})- \sum_{k=0}^{m-1} 
	\frac{\Phi_h^{(k)}(0)}{k!N^k} 
	\right| 
	\leq 
	\frac{\|\Phi^{(m)}\|_{[0,\frac{\delta}{2}]}}{m!N^m}
$$
whenever $\frac{1}{N}\le\frac{\delta}{2}=\frac{1}{48qq_0}$
by Taylor's theorem.
Both parts of the lemma now follow immediately from 
\eqref{eq:guemasterbern}, concluding the proof.
\end{proof}

\subsection{Extension to smooth functions}
\label{sec:theextensionproblemGUE}

To complete the proof of Theorem \ref{thm:smasympexGUE}, it remains to 
extend the expansion of Lemma \ref{lem:ratesofconv} from polynomial to 
smooth test functions $h$. The difficulty here is that, unlike in 
\cite{chen2024new}, the expansion \eqref{eq:rateofconvergence} cannot hold 
for arbitrarily large $q\in\bN$ due to the unboundedness of the Gaussian 
distribution. To surmount this problem, we must provide a separate 
treatment of the low- and high-degree terms in the Chebyshev expansion of 
$h$.

\begin{proof}[Proof of Theorem \ref{thm:smasympexGUE}]
We first note that the linear functional $\nu_m$ of Lemma~\ref{lem:ratesofconv} 
extends to a compactly supported distribution with 
$|\nu_m(h)| \lesssim \|h\|_{C^{2m+1}[-K,K]}$ for every $m\in\bZ_+$ by 
Lemma \ref{lem:extensiontoadistribution} and the inequality 
\eqref{eq:infinitesimalbound}.

In the following, we fix any bounded $h\in C^\infty(\mathbb{R})$ and 
denote by
\begin{equation}
\label{eq:chebyshevexpansionforh}
	h(x) = \sum_{q=0}^\infty a_q\,T_q(K^{-1}x)\quad 
	\text{for } x\in [-K, K]
\end{equation}
its Chebyshev expansion on the interval $[-K,K]$ (cf.\ section 
\ref{sec:chebexp}).
Note that as $\E[\tr h(\xn)]=c$ is independent of $N$ whenever 
$h\equiv c$ is a constant function, we have $\nu_0(c)=c$ and $\nu_k(c)=0$ 
for all $k\ge 1$. This implies that the theorem statement is invariant 
under the replacement $h\leftarrow h-a_0$. We will therefore assume 
without loss of generality in the rest of the proof that $a_0=0$.

Let $B\le\frac{N}{q_0}$ be the largest integer $q$ for which 
\eqref{eq:rateofconvergence} has been established, and let
$$
	h_0(x) := \sum_{q=1}^B a_q\,T_q(K^{-1}x).
$$
Then we can estimate
\begin{multline*}
	\left| \E[\tr h(\xn)]-  \sum_{k=0}^{m-1} \frac{\nu_k(h)}{N^k} 
	\right| \leq 
	\left| \E[\tr h_0(\xn)]- \sum_{k=0}^{m-1} 
	\frac{\nu_k(h_0)}{N^k} \right| \\ 
	+
	\sum_{k=0}^{m-1}  \frac{|\nu_k(h-h_0)|}{N^k}   + 
	|\E[\tr (h(\xn)-h_0(\xn))]| .
\end{multline*}
We now bound each term on the right-hand side.

\medskip

\noindent\textbf{First term.} Using \eqref{eq:rateofconvergence}, 
we readily estimate
\begin{align*}
	\left| \E[\tr h_0(\xn)]- \sum_{k=0}^{m-1} 
	\frac{\nu_k(h_0)}{N^k} \right| 
	& \leq \sum_{q=1}^{B} |a_q| \left| \E[\tr T_q(K^{-1}\xn) ]- 
	\sum_{k=0}^{m-1} \frac{\nu_k(T_q(K^{-1}x))}{N^k} \right|  
 \\ 
	&\leq \frac{(Cq_0)^{2m}}{ m! N^m} \sum_{q=1}^{B} q^{2m} |a_q|.
\end{align*}

\smallskip

\noindent\textbf{Second term.}
Because $|\nu_k(h)| \lesssim \|h\|_{C^{2k+1}[-K,K]}$ for all 
$k$, we are able to substitute the Chebyshev expansion 
\eqref{eq:chebyshevexpansionforh} into $\nu_k(h-h_0)$. This yields
$$
	\sum_{k=0}^{m-1}  \frac{|\nu_k(h-h_0)|}{N^k}  \leq 
	\sum_{k=0}^{m-1} \frac{1}{N^k} \sum_{q> B} |a_q| 
	\, |\nu_k(T_q(K^{-1}x))| 
	\leq  
	\sum_{k=0}^{m-1} \frac{(Cq_0)^{2k}}{k! N^k} 
	\sum_{q> B} q^{2k}|a_q|
$$
using \eqref{eq:infinitesimalbound}.
Now note that 
$q^{2k} \le \frac{q^{2m}}{B^{2(m-k)}} \le
(\frac{Cq_0}{N})^{2(m-k)}q^{2m}$
for any $k<m$ and $q>B$, 
where we used that $B\gtrsim \frac{N}{q_0}$ by Lemma 
\ref{lem:ratesofconv}. We therefore obtain
$$
	\sum_{k=0}^{m-1}  \frac{|\nu_k(h-h_0)|}{N^k}  \leq 
	\Bigg(\sum_{k=0}^{m-1} \frac{1}{k!N^{m-k}}\Bigg)
	\frac{(Cq_0)^{2m}}{N^m}
	\sum_{q> B} q^{2m}|a_q|.
$$

\smallskip

\noindent\textbf{Third term.}
We estimate
\begin{align*}
	&|\E[\tr (h(\xn)-h_0(\xn))]| 
\le
	|\E[\tr (h(\xn)-h_0(\xn))\cdot 1_{\{\|\xn\|\leq K\}}]|  
\\
	&\qquad\qquad\qquad
	+
	|\E[\tr h(\xn)\cdot 1_{\{\|\xn\|> K\}}]| +
	|\E[\tr h_0(\xn)\cdot 1_{\{\|\xn\|> K\}}]|  
\\
	& \le 
	\frac{(Cq_0)^{2m}}{N^{2m}}
	\sum_{q>B}q^{2m}|a_q| +
	Cre^{-cN}
	\Bigg(
	\sobnorm{h}{0}{(-\infty,\infty)} +
	\sum_{q\le B}|a_q|\Bigg)
\end{align*}
using \eqref{eq:chebyshevexpansionforh} and $1\le 
(\frac{Cq_0}{N})^{2m}q^{2m}$, 
\eqref{eq:XNlessthanK}, and
\eqref{eq:secondapriori}, 
respectively.

\medskip

Combining the above estimates yields
\begin{multline*}
	\left| \E[\tr h(\xn)]-  \sum_{k=0}^{m-1} \frac{\nu_k(h)}{N^k} 
	\right| \\ \leq 
	\Bigg(\sum_{k=0}^{m} \frac{1}{k!N^{m-k}}\Bigg)
	\frac{(Cq_0)^{2m}}{N^m}
	\sum_{q=1}^\infty q^{2m}|a_q|
	+
	Cre^{-cN}
	\Bigg(
	\sobnorm{h}{0}{(-\infty,\infty)} +
	\sum_{q=1}^\infty|a_q|\Bigg).
\end{multline*}
To conclude the proof, it remains to
apply Lemma \ref{lem:zygmund} and to note that we can estimate
$
	\sum_{k=0}^{m} \frac{1}{k!N^{m-k}}
	\le
	\frac{1}{m!}
	\sum_{k=0}^{m} (\frac{m}{N})^{m-k}
	\le
	\frac{2}{m!}
$
for $m\le \frac{N}{2}$.
\end{proof}

\section{Strong convergence for GUE}
\label{sec:bootstrapping}

The aim of this section is to complete the proof of 
Theorem~\ref{thm:maingauss} for the case that 
$\boldsymbol{G}^N=(G_1^N,\ldots,G_r^N)$ are GUE matrices. With the 
asymptotic expansion of Theorem~\ref{thm:smasympexGUE} in hand, it remains 
to control the supports of the infinitesimal distributions $\nu_k$ as in 
\eqref{eq:suppincl}. To this end, we will first prove a basic form of 
strong convergence in section~\ref{sec:guebasiccvg}, which only requires 
control of $\nu_1$. We will then apply a bootstrapping argument in 
section~\ref{sec:bootsgue} to extend the conclusion to all $\nu_k$. 
Finally, we complete the proof of Theorem~\ref{thm:maingauss} for GUE 
in section~\ref{sec:pfmaingue}.

Throughout this section, we will always assume without further comment 
that the setting and notations of section \ref{sec:thepolymethodforGUE} 
are in force.

\subsection{Strong convergence}
\label{sec:guebasiccvg}

Establishing strong convergence for GUE matrices is especially simple due 
to the following observation.

\begin{lem}
\label{lem:guenu01}
In the setting of Theorem \ref{thm:smasympexGUE}, we have
$$
	\nu_0(h) = ({\tr\otimes\tau})(h(\xf))
	\quad\text{and}\quad
	\nu_1(h)=0\quad
	\text{for all }h\in C^\infty(\mathbb{R}).
$$
\end{lem}

\begin{proof}
It suffices to consider $h\in\mathcal{P}$. Then Lemma \ref{lem:waf} yields
$$
	\nu_0(h) = \lim_{N\to\infty} \E[\tr \, h(\xn)] =
	({\tr\otimes\tau})(h(\xf)).
$$
Now note that as 
$\Phi_h(-\frac{1}{N})=\Phi_h(\frac{1}{N})$ in Lemma 
\ref{lem:polyencoding}, the polynomial $\Phi_h(x)$ can contain only 
monomials of even degree. Thus $\nu_1(h) = \Phi_h'(0)=0$.
\end{proof}

This yields the following.

\begin{cor}[Strong convergence]
\label{cor:strongconv}
$\|\xn\|\to\|X_{\rm F}\|$ in probability as $N\to\infty$.
\end{cor}

\begin{proof}
Fix $\varepsilon>0$ sufficiently small, and let $h$ be the test function 
of Lemma \ref{lem:testfunctions} with $m=4$, 
$K=(Cr)^{q_0}\|X_{\rm F}\|$, and $\rho=\|X_{\rm F}\|$. Then 
$\nu_0(h)=\nu_1(h)=0$ by Lemma~\ref{lem:guenu01} and as $h$ vanishes on
$[-\|\xf\|-\frac{\varepsilon}{2},\|\xf\|+\frac{\varepsilon}{2}]$. Thus 
Theorem~\ref{thm:smasympexGUE} with $m=2$ yields
$$
	\mathbf{P}[\|\xn\|>\|\xf\|+\varepsilon]
	\le
	\mathbf{E}[\trace h(X^N)] 
	= DN\,\mathbf{E}[\tr h(X^N)] 
	= O\bigg(\frac{D}{N}\bigg),
$$
where we used that $h(x)\in[0,1]$ for all $x$ and $h(x)=1$ for 
$|x|>\|\xf\|+\varepsilon$ in the first inequality. As $\varepsilon$ may be 
chosen arbitrarily small, the conclusion follows.
\end{proof}

\subsection{Bootstrapping}
\label{sec:bootsgue}

While Corollary \ref{cor:strongconv} has been stated in a qualitative 
form, its proof shows that strong convergence remains valid for sequences 
of noncommutative polynomials $P$ with matrix coefficients of dimension 
$D=o(N)$. The information obtained so far does not suffice, however, to 
capture coefficients of exponential dimension; to this end, the 
higher-order infinitesimal distributions $\nu_m$ must be controlled as 
well. The aim of this section is to prove the following.

\begin{prop}
\label{prop:supportofinfsGUE}
In the setting of Theorem \ref{thm:smasympexGUE}, we have
$$
	\supp\nu_m \subseteq[-\|\xf\|, \|\xf\|]
	\quad\text{for all }
	m\in\bZ_+.
$$
\end{prop}

As was explained in section \ref{sec:introboot}, we will prove this 
theorem by combining Corollary \ref{cor:strongconv} with concentration of 
measure. This enables a bootstrapping argument that uses only information 
on $\nu_0,\nu_1$ in Lemma \ref{lem:guenu01} to achieve control of $\nu_k$ 
for all $k$.

\begin{rem}
Proposition \ref{prop:supportofinfsGUE} for GUE recovers a result of 
Parraud 
\cite{parraud2023asymptotic} in the special case $D=1$ (that is, without 
matrix coefficients). The argument given here holds for any $D\in\bN$
and will extend almost verbatim beyond the GUE case.
\end{rem}

\subsubsection{Concentration of measure}
\label{sec:concentrationaroundthemedian}

Before we develop the bootstrapping argument, we recall the requisite 
concentration property in a convenient form for our purposes. Similar 
results are well known, see, e.g., \cite[Lemma 7.6]{pisier2014random}. We 
include the proof as it is short and carries over to GOE and GSE 
matrices without any changes. Here and in the sequel, $\mathrm{med}(Z)$ 
denotes the median of a random variable $Z$.

\begin{lem}
\label{lem:concentrationaroundthemean}
For any $\varepsilon>0$ and $N\in\bN$, we have
$$
	\P[|\|\xn\|-\mathrm{med}(\|\xn\|)|>\varepsilon] \leq 
	Cr e^{-cN} +  Ce^{-c_P N \varepsilon^2}
$$
for a constant $c_P>0$ that depends only on $P$ and universal 
constants $C,c>0$. 
\end{lem}

The proof uses a Gaussian concentration inequality
for non-Lipschitz functions.

\begin{lem}[Gaussian concentration] 
\label{lem:localconcentration}
Let $Z\sim N(0, \id_d)$ be a 
$d$-dimensional
standard Gaussian vector. 
Let $\Omega\subseteq \bR^d$ be a measurable set with 
$\P[Z\in \Omega] \geq \frac{3}{4}$, and let $f: \bR^{d}\to \bR$ be a
function whose restriction to $\Omega$ is $L$-Lipschitz.
Then
$$
	\P[|f(Z)-\mathrm{med}(f(Z))| > \varepsilon] 
	\leq \P[Z\notin \Omega] + C e^{-c \varepsilon^2/L^2}
$$
for any $\varepsilon>0$, where $C, c>0$ are universal constants. 
\end{lem}

\begin{proof}
The result is stated in \cite[Lemma 2.2]{aubrun2014entanglement} for
random vectors on the unit sphere. The Gaussian case follows immediately 
by the Poincar\'e limit \cite[eq.\ (3.3)]{Bor75}.
\end{proof}

Lemma \ref{lem:concentrationaroundthemean} is then a direct consequence of 
the above result.

\begin{proof}[Proof of Lemma \ref{lem:concentrationaroundthemean}]
When restricted to the set 
$\Omega = \{\|G_i^N\|\leq \kappa+1 \text{ for all }i\in[r]\}$,
it is clear that $\|\xn\|=\|P(G_1^N,\ldots,G_r^N)\|$ is an $L_P$-Lipschitz 
function of the real and imaginary parts of the entries of
$G_1^N,\ldots,G_r^N$ on and above the diagonal, where $L_P$ 
depends only on $P$. As these entries are independent Gaussians with 
variance $\frac{1}{N}$, the conclusion follows readily from Lemma 
\ref{lem:localconcentration} and Lemma \ref{lem:subgaussiannorm}.
\end{proof}

\subsubsection{Proof of Proposition \ref{prop:supportofinfsGUE}} 
\label{sec:proofofsupportofinfsGUE}

The key observation behind the proof is that strong convergence and 
concentration of measure imply that any spectral statistic that vanishes 
in a neighborhood of $[-\|\xf\|,\|\xf\|]$ is exponentially small.

\begin{cor}
\label{cor:smallexpectations}
Fix $\varepsilon>0$ and any bounded function $h\in C^{\infty}(\bR)$ 
that vanishes on the interval
$[-\|\xf\|-\varepsilon,\|\xf\|+\varepsilon]$.
Then we have
$$
	|\E[ \tr h(\xn)]| = O(e^{-cN})
	\quad\text{as}\quad N\to\infty,
$$
where $c>0$ may depend on $P$ and $\varepsilon$.
\end{cor}

\begin{proof}
By the assumption on $h$, we have
$$
	|\mathbb{E}[\tr h(\xn)]| 
	= |\mathbb{E}[\tr h(\xn) 1_{\|\xn\|> \|\xf\|+\varepsilon}]| 
	\le 
	\|h\|_{(-\infty,\infty)}\, \P[\|\xn\|> \|\xf\|+\varepsilon].
$$
Now note that $\mathrm{med}(\|\xn\|)\le \|\xf\|+\frac{\varepsilon}{2}$
for all $N$ sufficiently large by Corollary~\ref{cor:strongconv}. 
It follows that for all $N$ sufficiently large
$$
	|\mathbb{E}[\tr h(\xn)]| 
	\le 
	\|h\|_{(-\infty,\infty)}\, \P[\|\xn\|-\mathrm{med}(\|\xn\|)
	> \tfrac{\varepsilon}{2}],
$$
and the conclusion follows from Lemma \ref{lem:concentrationaroundthemean}. 
\end{proof}

We can now prove Proposition \ref{prop:supportofinfsGUE}.

\begin{proof}[Proof of Proposition \ref{prop:supportofinfsGUE}]
Fix any $\varepsilon>0$ and bounded $h\in C^{\infty}(\bR)$
that vanishes on $[-\|\xf\|-\varepsilon,\|\xf\|+\varepsilon]$.
We show by induction that $\nu_m(h) =0$ for all $m$.

That $\nu_0(h)=0$ follows immediately from Lemma \ref{lem:guenu01}.
Now let $m\geq 1$ and assume that we have shown $\nu_0(h)=\cdots  
=\nu_{m-1}(h) =0$. Then Theorem \ref{thm:smasympexGUE} yields
$$
	\bigg| \E[\tr h(\xn)] - \frac{\nu_m(h)}{N^m}\bigg| =
	O\bigg(\frac{1}{N^{m+1}}\bigg)
$$
as $N\to\infty$. Thus 
$$
	|\nu_m(h)| \le 
	N^m|\E[\tr h(\xn)]| +
	O\bigg(\frac{1}{N}\bigg)
	=
	O\bigg(\frac{1}{N}\bigg)
$$
as $N\to\infty$
by the triangle inequality and Corollary \ref{cor:smallexpectations}.
As the left-hand side is independent of $N$, it follows that $\nu_m(h)=0$.
\end{proof}

\subsection{Proof of Theorem \ref{thm:maingauss}: GUE case}
\label{sec:pfmaingue}

We can now prove Theorem \ref{thm:maingauss} in the case that 
$\boldsymbol{G}^N$ are GUE matrices by combining Theorem 
\ref{thm:smasympexGUE} and Proposition \ref{prop:supportofinfsGUE}.

\begin{proof}[Proof of Theorem \ref{thm:maingauss}: GUE case]
We may assume without loss of generality that $P$ is self-adjoint
(see Remark \ref{rem:selfadjoint} below).
Fix $\varepsilon\in(0,1]$, $K=(Cr)^{q_0}\|\xf\|$, and 
$m\in\bN$ with $m\le\frac{N}{2}$ that will be chosen at the end of the 
proof. 

Let $h$ be the test function provided by Lemma \ref{lem:testfunctions} 
with $m\leftarrow 2m$, $\rho\leftarrow \|\xf\|$, and 
$\varepsilon\leftarrow\varepsilon\|\xf\|$.
Then $\nu_k(h)=0$ for all $k\in\bZ_+$ by Proposition 
\ref{prop:supportofinfsGUE}, and thus
\begin{multline*}
	\mathbf{P}[\|\xn\| \ge (1+\varepsilon)\|\xf\|]
	\le
	\mathbf{E}[\trace h(\xn)] =
	DN\,\mathbf{E}[\tr h(\xn)]
\\
	\le
	DN \bigg[
	\frac{(Cq_0m)^{2m}}{m! N^m} 
	\bigg(\frac{(Cr)^{q_0}}{\varepsilon}\bigg)^{2m+1}
	+
	Cre^{-cN}\bigg(1 + \frac{(Cr)^{q_0}}{\varepsilon}\bigg)
	\bigg]
\end{multline*}
for a universal constant $C$
by Theorem \ref{thm:smasympexGUE} and Lemma \ref{lem:testfunctions}.
In particular, if we assume that $D\le e^m$, then we can further estimate
$$
	\mathbf{P}[\|\xn\| \ge (1+\varepsilon)\|\xf\|]
	\le
	\frac{(Cr)^{q_0+1}N}{\varepsilon}
	\bigg[
	\bigg(\frac{(Cr)^{2q_0}q_0^2m}{N\varepsilon^2}\bigg)^{m}
	+
	e^{m-cN} 
	\bigg]
$$
for a universal constant $C$, where we used that
$\frac{1}{m!}\le(\frac{e}{m})^m$.

Now choose $m=\lfloor \frac{N\varepsilon^2}{L} \rfloor$ with
$L:= \max\{e (Cr)^{2q_0}q_0^2,\frac{2}{c}\}$.
If $m\ge 2$, then we have $\frac{N\varepsilon^2}{2L} \le m \le 
\frac{N\varepsilon^2}{L}$, and the above estimate yields
$$
	\mathbf{P}[\|\xn\| \ge (1+\varepsilon)\|\xf\|]
	\le
	\frac{(Cr)^{q_0+1}N}{\varepsilon}
	\big(e^{-N\varepsilon^2 /2L}
	+ e^{-cN\varepsilon^2/2}
	\big),
$$
when $D\le e^{N\varepsilon^2 /2L}$,
where we used that $m-cN\le (\frac{1}{L}-c)N\varepsilon^2 
\le -\frac{cN\varepsilon^2}{2}$ as $\varepsilon\le 1$.

On the other hand, if $m<2$, then $\frac{N\varepsilon^2}{L}<2$ and thus
$$
	\mathbf{P}[\|\xn\| \ge (1+\varepsilon)\|\xf\|] \le
	1 
	\le 
	e^2 e^{-N\varepsilon^2/L}
	\le 
	\frac{e^2N}{\varepsilon} e^{-N\varepsilon^2/L},
$$
where we used that $\frac{N}{\varepsilon}\ge 1$.
This concludes the proof.
\end{proof}

\begin{rem}
\label{rem:selfadjoint}
Throughout this section, we assumed that $P$ is a \emph{self-adjoint} 
noncommutative polynomial, and proved Theorem \ref{thm:maingauss} in this 
case. However, the conclusion extends immediately to arbitrary $P$ by 
applying the self-adjoint case to $P^*P$. Thus the restriction to 
self-adjoint $P$ entails no loss of generality. We apply the 
same observation in the remainder of the paper without further comment. 
\end{rem}

\section{Strong convergence for GOE and GSE}
\label{sec:supersymmetryGOEGSE}

The aim of this section is to complete the proof of Theorem 
\ref{thm:maingauss} for the GOE and GSE ensembles. Most of the proof in 
the GUE case extends \emph{verbatim} to the present setting, so that we 
will focus attention only on the necessary modifications. The key 
difference in the case of GOE and GSE is that it is no longer true that 
$\Phi_h(x)=\Phi_h(-x)$ as in Lemma \ref{lem:polyencoding}, and thus that 
$\nu_1=0$ as in Lemma \ref{lem:guenu01}. Instead the arguments in the 
proof where these properties were used will be adapted to the GOE 
and GSE cases by exploiting supersymmetric duality.

The following setting and notations will be fixed throughout this section. 
Let $\boldsymbol{G}^N=(G_1^N,\ldots,G_r^N)$ and 
$\boldsymbol{H}^N=(H_1^N,\ldots,H_r^N)$ be independent GOE and GSE 
matrices of dimension $N$, respectively, and  
$\boldsymbol{s}=(s_1,\ldots,s_r)$ be a free semicircular family. We
fix a self-adjoint noncommutative polynomial $P\in 
\mathrm{M}_D(\bC) \otimes \bC \langle x_1,\ldots,x_r\rangle$ of degree 
$q_0$ with matrix coefficients of dimension $D$, and denote by
$$
	\xn := P(\boldsymbol{G}^N),\qquad\quad
	\yn := P(\boldsymbol{H}^N),\qquad\quad
	X_\mathrm{F} := P(\boldsymbol{s})
$$
the random matrices of interest and the limiting model.

\subsection{Supersymmetric duality}
\label{sec:goegsesupersymm}

The key fact we will use is that the GOE and GSE ensembles are dual in the 
sense that their moments are encoded by the same polynomial at 
positive and negative values. This is captured by the following result, 
which replaces Lemma \ref{lem:polyencoding} in the present setting.

\begin{lem}[Polynomial encoding]
\label{lem:polyencodingGOEGSE}
For any $h\in \calP_q$, there is $\Phi_h\in\calP_{qq_0}$
so that
\begin{alignat*}{2}
	&\E[\tr h(\xn)] &&= \Phi_h(\tfrac{1}{N}),\\
	&\E[\tr h(\yn)] &&= \Phi_h(-\tfrac{1}{2N}).
\end{alignat*}
\end{lem}

\begin{proof}
This is a direct consequence of the genus expansions for GOE and GSE 
established in \cite{bryc2009duality}. In the present notation,
\cite[Theorem B]{bryc2009duality} states that
$$
	\E[\tr \gn_{i_1}\cdots \gn_{i_{2n}}] =
	\sum_\Gamma N^{\chi(\Gamma)-2}
$$
for every $n\in\bN$ and $i_1,\ldots,i_{2n}\in[r]$, while
\cite[Theorem 4.1]{bryc2009duality} states that
$$
	\E[\tr \hn_{i_1}\cdots \hn_{i_{2n}}] =
	\sum_\Gamma (-2N)^{\chi(\Gamma)-2}
$$
for every $n\in\bN$ and $i_1,\ldots,i_{2n}\in[r]$.\footnote{%
	The reader should beware that the notation of 
	\cite{bryc2009duality} differs from that of the present paper: in 
	\cite{bryc2009duality}, GOE and GSE matrices are normalized so 
	that their off-diagonal elements have variance $1$ and $4$,
        respectively, and $\tr$ denotes the sum of
	the diagonal entries viewed as elements of $\mathbb{H}$.}
Here the sums range over a certain family (that is determined by the 
choice of $i_1,\ldots,i_{2n}$) of closed connected surfaces
$\Gamma$ that have a cell decomposition with
$1$ vertex and $n$ edges, and 
$\chi(\Gamma)$ denotes the Euler characteristic.
Since the stated properties of $\Gamma$ ensure that
$-n\le\chi(\Gamma)-2 \le 0$, 
the right-hand sides of the above equations are defined by the
same polynomial of degree at most $n$ applied to
$\frac{1}{N}$ and $-\frac{1}{2N}$, respectively.

Now note that, by linearity, $\E[\tr h(\xn)]$ and $\E[\tr h(\yn)]$ are 
linear combinations of expected traces of words of length 
at most $qq_0$ in $\boldsymbol{G}^N$ and $\boldsymbol{H}^N$, respectively. 
We have shown that words of even length yield a polynomial as in the 
statement, while words of odd length vanish as the Gaussian distribution 
is symmetric.
Finally, note that $\E[\tr h(\xn)]$ is real as 
$P$ is self-adjoint, so $\Phi_h$ is a real polynomial.
\end{proof}

\subsection{Asymptotic expansion}

With Lemma \ref{lem:polyencodingGOEGSE} in hand, we can now repeat the 
proof of Theorem \ref{thm:smasympexGUE} with only trivial modifications.

\begin{thm}[Smooth asymptotic expansion for GOE/GSE]
\label{thm:smasympexGOEGSE} 
There exist universal constants $C,c>0$, and a compactly supported 
distribution $\nu_k$ for every $k\in\bZ_+$, such that the following 
hold. Fix any bounded $h\in C^\infty(\bR)$, and define
$$
	f(\theta):= h(K\cos\theta)
	\qquad\text{with}\qquad
	K := (Cr)^{q_0}\|X_{\rm F}\|.
$$
Then for every $m, N\in \bN$ with $m\le\frac{N}{2}$, we have
\begin{multline*}
       \left| \mathbf{E}[\tr h(\xn)]
	- \sum_{k=0}^{m-1} \frac{\nu_k(h)}{N^k} \right| 
	~\vee~
       \left| \mathbf{E}[\tr h(\yn)]
	- \sum_{k=0}^{m-1} \frac{\nu_k(h)}{(-2N)^k} \right| 
\\
       \le
	\frac{(Cq_0)^{2m}}{m! N^m} \|f^{(2m+1)}\|_{[0, 2\pi]} 
	+
	Cre^{-cN}\big(\sobnorm{h}{0}{(-\infty,\infty)}
		+ \sobnorm{f^{(1)}}{0}{[0, 2\pi]} \big).
\end{multline*}
\end{thm}

\smallskip

\begin{proof}
The only modification that must be made to the arguments of section 
\ref{sec:thepolymethodforGUE} is that we apply 
Lemma~\ref{lem:polyencodingGOEGSE} instead of
Lemma~\ref{lem:polyencoding} in the proof of Lemma 
\ref{lem:ratesofconv}; the master inequalities for GOE and GSE are then 
obtained by Taylor expanding
$\Phi_h(x)$ at the points $x=\frac{1}{N}$ and $x=-\frac{1}{2N}$, 
respectively. The remaining results and proofs in section 
\ref{sec:aprioriGUE} extend \emph{verbatim} to the case of GOE and GSE.
\end{proof}

\subsection{The first-order distribution}
\label{sec:nu1supersymmetry}

While Lemma \ref{lem:waf} directly yields
$$
	\nu_0(h) = \lim_{N\to\infty} \E[\tr \, h(\xn)] =
	({\tr\otimes\tau})(h(\xf))
$$ 
as for GUE, it is no longer true in the present setting that we have 
$\nu_1(h)=0$ as in Lemma~\ref{lem:guenu01}. Nonetheless, we can exploit 
the duality between the GOE and GSE ensembles to give a very simple proof 
of the following fact.

\begin{lem}
\label{lem:GOEGSEnu1}
In the setting of Theorem \ref{thm:smasympexGOEGSE}, we have
$$
	\supp \nu_1 \subseteq [-\|\xf\|,\|\xf\|].
$$
\end{lem}

\smallskip

In the proof we will need the following basic fact about distributions.

\begin{lem}
\label{lem:distv}
Let $\nu$ be a compactly supported distribution and $A\subseteq\mathbb{R}$
be a closed set. If $\nu(h)=0$ for all \emph{nonnegative} bounded 
functions $h\in C^\infty(\mathbb{R})$ that vanish in a 
neighborhood of $A$, then we have $\supp\nu\subseteq A$.
\end{lem}

\begin{proof}
The assumption implies \emph{a fortiori} that
$\nu(h)\ge 0$ for all nonnegative functions $h\in C^\infty_0(\bR\backslash 
A)$. The restricted distribution $\nu|_{\bR\backslash A}$ (cf.\ \cite[\S 
2.2]{hormander2003}) is therefore a 
positive measure by \cite[Theorem 2.1.7]{hormander2003}.
Thus $\nu(h)=0$ for all nonnegative functions $h\in 
C^\infty_0(\bR\backslash A)$ then implies that $\nu|_{\bR\backslash A}=0$, 
which yields the conclusion.
\end{proof}

We can now prove Lemma \ref{lem:GOEGSEnu1}.

\begin{proof}[Proof of Lemma \ref{lem:GOEGSEnu1}]
Let $h$ be a bounded nonnegative function $h\in C^{\infty}(\bR)$, $h\ge 
0$ that vanishes in a neighborhood of $[-\|\xf\|,\|\xf\|]$. Then
Theorem \ref{thm:smasympexGOEGSE} yields
\begin{alignat*}{2}
	& 0 \le \E[\tr h(\xn)] && =
	\frac{\nu_1(h)}{N} + O\bigg(\frac{1}{N^2}\bigg), \\
	& 0 \le \E[\tr h(\yn)] && =
	-\frac{\nu_1(h)}{2N} + O\bigg(\frac{1}{N^2}\bigg),
\end{alignat*}
where we used that $\nu_0(h)=({\tr\otimes\tau})(h(\xf))=0$.
Taking $N$ sufficiently 
large then yields both $\nu_1(h)\ge 0$ and $-\nu_1(h)\ge 0$, 
which implies that we must have $\nu_1(h)=0$. The conclusion now follows 
from Lemma \ref{lem:distv}.
\end{proof}

\subsection{Proof of Theorem \ref{thm:maingauss}: GOE/GSE case}
\label{sec:pfmaingoegse}

The remainder of the proof of Theorem \ref{thm:maingauss} is now 
essentially identical to the proof of the GUE case.

\begin{proof}[Proof of Theorem \ref{thm:maingauss}: GOE/GSE case]
The results and proofs of section \ref{sec:bootstrapping} 
extend \emph{verbatim} to the case of GOE and GSE, provided that 
Theorem~\ref{thm:smasympexGUE} and Lemma~\ref{lem:guenu01} are replaced by 
Theorem~\ref{thm:smasympexGOEGSE} and Lemma~\ref{lem:GOEGSEnu1}, 
respectively.
\end{proof}

\section{Strong convergence for $\mathrm{U}(N)$}
\label{sec:Unitarymatrices}

We now turn our attention to Haar-distributed random matrices. While the 
structure of the proofs is similar to those for the Gaussian 
ensembles, there are distinct complications that arise in the two settings 
that must be surmounted to reach matrix coefficients of (nearly) 
exponential dimension. Unlike in the Gaussian case, no truncation is 
needed in the Haar setting as the random matrix model is already bounded. 
However, in the Haar setting polynomial spectral statistics are no longer 
polynomials, but rather rational functions, of $\frac{1}{N}$. These 
require a careful analysis in order not to incur any quantitative loss in 
the final result.

The aim of this section is to give a complete proof of Theorem 
\ref{thm:mainhaar} for Haar-distributed matrices in $\mathrm{U}(N)$.
The requisite modifications for the $\mathrm{O}(N)/\mathrm{Sp}(N)$ 
models will subsequently be developed in section \ref{sec:O(N)andSp(N)}.

The following will be fixed throughout this section. Let
$\boldsymbol{U}^N=(\haar_1, \dots, \haar_r)$ be independent
Haar-distributed random matrices in $\mathrm{U}(N)$, and let
$\boldsymbol{u}=(u_1,\ldots,u_r)$ be free Haar unitaries.
We further fix a self-adjoint noncommutative polynomial
$P\in\mathrm{M}_D(\bC)\otimes \bC\langle 
x_1,\ldots,x_r,x_1^*,\ldots,x_r^*\rangle$ of degree $q_0$ with matrix 
coefficients of dimension $D$. We denote the random matrix of interest and 
its limiting model as
$$
	\xn := P(\boldsymbol{U}^N,\boldsymbol{U}^{N*}),\qquad\quad
	\xf := P(\boldsymbol{u},\boldsymbol{u}^*).
$$
Finally, we will denote by
$$
	K := \|P\|_{\mathrm{M}_D(\bC)\otimes C^*(\bF_r)} =
	\sup_{n\in\bN}
	\sup_{W_1,\ldots,W_r\in\mathrm{U}(n)}
	\|P(W_1,\ldots,W_r,W_1^*,\ldots,W_r^*)\|
$$
the norm of $P$ in the full $C^*$-algebra of the free group
$\bF_r$ with $r$ free generators. The significance of this definition is
that $\|\xn\|\le K$ for every $N$.

\subsection{Polynomial encoding}

For GUE matrices, Lemma \ref{lem:polyencoding} showed that polynomial 
spectral statistics can be expressed as polynomials of $\frac{1}{N}$. The 
aim of this section is to prove an analogue of this property for 
$\mathrm{U}(N)$ matrices. In this case, however, we obtain rational 
functions rather than polynomials.

In the sequel, we always fix the following special polynomial
\begin{equation}
\label{eq:defofgL}
	g_q(x):= 
	\prod_{j=1}^q (1-(jx)^2)^{\lfloor \frac{q}{j} \rfloor}
\end{equation}
that will arise as the denominator in the rational 
expressions that appear for Haar-distributed matrices. 
We aim to prove the following.

\begin{lem}[Polynomial encoding]
\label{lem:rationalencoding}
For every $h \in \calP_q$, there is a rational function of the form 
$\Psi_h:=\frac{f_h}{g_{qq_0}}$ with 
$f_h,g_{qq_0}\in\mathcal{P}_{\lfloor 3qq_0(1+\log qq_0)\rfloor}$
so that
$$
	\E[\tr h(\xn)] = \Psi_h(\tfrac{1}{N}) =
	\Psi_h(-\tfrac{1}{N})
$$
for all $N\in\bN$ such that $N>qq_0$.
\end{lem}

\begin{rem}
\label{rem:loselog}
Lemma \ref{lem:rationalencoding} and its counterpart for 
$\mathrm{O}(N)/\mathrm{Sp}(N)$ models are solely responsible for the loss 
of the logarithmic factor in Theorem \ref{thm:mainhaar} as compared to Theorem 
\ref{thm:maingauss}, which arises from the logarithmic factor in 
the degree of the numerator and denominator $f_h,g_{qq_0}$ of 
the rational function $\Psi_h$. It is unclear whether this logarithmic
factor is necessary, as there could be cancellations between the numerator 
and denominator that are not captured by the crude analysis below.
\end{rem}

Lemma \ref{lem:rationalencoding} for $\mathrm{U}(N)$ is a special case of
\cite[Theorem 3.1]{magee2024quasiexp}. However, as the argument will be 
needed below also for $\mathrm{O}(N)/\mathrm{Sp}(N)$, we spell out the 
proof here. We begin with some definitions. For every $L\in\bN$, we denote 
by $\mathrm{S}_L$ the symmetric group on $L$ letters. In view of
\cite[Theorem 4.3]{collins2022weingarten}, we introduce the following.

\begin{defn}[Unitary Weingarten functions]
\label{defn:unitaryweingarten}
For any $L\in \bN$ and $\alpha\in\mathrm{S}_L$
\begin{equation}
   \label{eq:weingartenfunction}
   \wg_L(\alpha, N) := 
	\frac{1}{L!} \sum_{\lambda \vdash L} 
	\frac{d_\lambda \chi^{\lambda}(\alpha)}{\prod_{\square \in \lambda }
	(N + c(\square))}  
\end{equation}
where the product runs over all boxes $\square=(i,j)$ in the Young diagram 
associated to $\lambda$ and $c(\square):=j-i$.
Here $d_\lambda$ denotes the dimension and $\chi^{\lambda}$ denotes the 
character of 
the irreducible representation of $\mathrm{S}_L$ associated to $\lambda$.
\end{defn}

For our purposes, the only relevant feature of the function $\wg_L(\alpha, N)$
is that it is a rational function of $N$ and that we know its 
poles.

\begin{lem}
\label{lem:denominatorsofwg}
For every $\alpha\in \mathrm{S}_L$, there exists
$a_\alpha\in\mathcal{P}$ so that
$$
	\wg_L(\alpha, N)=\frac{a_\alpha(N)}{N^L
	\prod_{k=1}^L (N^2-k^2)^{\lfloor\frac{L}{k}\rfloor}}
	\quad\text{for all }N>L.
$$
\end{lem}

\begin{proof}
Fix a partition $\lambda\vdash L$.
Since $\lambda$ has at most $L$ columns and at most $L$ rows, we clearly 
have $-L \leq c(\square)\leq L$ for all $\square \in \lambda$. Thus
$$
	\prod_{\square \in \lambda} (N+ c(\square)) =
	\prod_{k=-L}^{L}
	(N+k)^{\omega_k(\lambda)},
$$
where $\omega_k(\lambda)$ denotes the number of
boxes $\square=(i,j)$ in the Young diagram associated to $\lambda$ with 
$j-i=k$. As for any such box, the Young diagram must contain the rectangle 
with side lengths $i,j$, respectively, we have $ij\le L$ and thus 
\begin{align*}
	\omega_k(\lambda) 
	&\le \#\{ (i,j)\in [L]^2 : ij \le L,~j-i=k\}
\\
	&\le \#\{ j\in [L] : (|k|+j)j \le L\}
	\le\lfloor\tfrac{L}{|k|+1}\rfloor.
\end{align*}
Thus the numerator of
$$
	\wg_L(\alpha, N)
	=
	\frac{
	\sum_{\lambda \vdash L} 
	\frac{d_\lambda \chi^{\lambda}(\alpha)}{L!}
	N^{L-\omega_0(\lambda)} 
	\prod_{k=1}^L 
	(N-k)^{\lfloor\frac{L}{k}\rfloor-\omega_{-k}(\lambda)}
	(N+k)^{\lfloor\frac{L}{k}\rfloor-\omega_k(\lambda)}
	}
	{N^L \prod_{k=1}^L (N^2-k^2)^{\lfloor\frac{L}{k}\rfloor}}
$$
is polynomial in $N$. It remains to 
note that the numerator is in fact a real polynomial, as all characters 
$\chi^\lambda$ of the symmetric group are real-valued.
\end{proof}

We can now prove Lemma \ref{lem:rationalencoding}. 

\begin{proof}[Proof of Lemma \ref{lem:rationalencoding}]
Let $w(\haar_1,\ldots,\haar_r)$ be a reduced word of length at most $L$ in 
the Haar unitary matrices $\haar_i$ and their adjoints $\haaradj_i$.
Suppose that $w\ne\id$ and that it is balanced, 
that is, $\haar_i$ and $\haaradj_i$ appear an equal number of times.
Let $L_i$ be the number of appearances of $\haar_i$ in $w$. Then
by \cite[Theorem 2.8]{magee2019matrix} 
\begin{multline*}
	\E[\tr w(\haar_1, \dots, \haar_r)] =
\\
	\sum_{(\alpha_1, \beta_1) \in \mathrm{S}_{L_1}^2, \dots, 
	(\alpha_r, \beta_r) \in \mathrm{S}_{L_r}^2} 
	\left(\prod_{i=1}^r 
	\wg_{L_i}(\alpha_i^{-1}\beta_i ,N)\right) \ell(N ; \alpha_1, 
	\beta_1, \dots, \alpha_r, \beta_r)
\end{multline*}
for all $N>\max_i L_i$, where each 
$\ell(N ; \alpha_1, \beta_1, \dots, \alpha_r, \beta_r)$
is either identically zero or a non-negative integer power of
$N$. Thus Lemma \ref{lem:denominatorsofwg} yields
$$
	\E[\tr w(\haar_1, \dots, \haar_r)] =
	\frac{b_w(N)}{
	N^L\prod_{k=1}^L (N^2-k^2)^{\lfloor\frac{L}{k}\rfloor}}
$$
for some $b_w\in\mathcal{P}$,
where we used that $\sum_{i=1}^r L_i \le L$ and therefore
$\sum_{i=1}^r \lfloor \frac{L_i}{j}\rfloor \le \lfloor\frac{L}{j}\rfloor$.
Now note that as $|\E[\tr w(\haar_1, \dots, \haar_r)]|\le 1$ for all $N$,
the degree of the numerator is at most the degree
$\Sigma= L+2\sum_{k=1}^L \lfloor \frac{L}{k}\rfloor \le 3L(1+\log L)$
of the denominator.
Dividing both numerator and denominator by $N^\Sigma$ therefore yields
\begin{equation}
\label{eq:unwordrational}
	\E[\tr w(\haar_1, \dots, \haar_r)] = 
	\frac{f_w(\frac{1}{N})}{g_L(\frac{1}{N})}
\end{equation}
for all $N>L$, where $g_L$ was defined in \eqref{eq:defofgL} and 
$f_w,g_L\in\mathcal{P}_\Sigma$.

If $w$ is not balanced, 
it is readily seen that $\E[\tr w(\haar_1,\ldots,\haar_r)]=0$ as $\haar_i$
has the same distribution as $e^{i\theta}\haar_i$ for every $\theta$,
while clearly $\E[\tr w(\haar_1,\ldots,\haar_r)]=1$ for $w=\id$.
Thus \eqref{eq:unwordrational} remains valid for all reduced words $w$.

Now note that $\mathbf{E}[\tr h(\xn)]$ is a linear combination of terms 
$\E[\tr w(\haar_1, \dots, \haar_r)]$ for words $w$ of length at most $L= 
qq_0$. As we have shown that each term is real and as $\mathbf{E}[\tr 
h(\xn)]$ is also real, the representation $\E[\tr 
h(\xn)]=\Psi_h(\frac{1}{N})$ follows from \eqref{eq:unwordrational}. 
That $\Psi_h(\frac{1}{N})=\Psi_h(-\frac{1}{N})$ follows from 
\cite[Remark 1.9]{magee2019matrix}, which implies that the power series 
expansion of $\Psi_h$ contains only even powers of $\frac{1}{N}$. 
\end{proof}

\subsection{A rational Bernstein inequality}
\label{sec:rationalbernstein}

As $\Psi_h$ in Lemma \ref{lem:rationalencoding} is a rational function, we 
can no longer apply Bernstein's inequality directly to control the 
remainder term in its Taylor expansion as we did in section 
\ref{sec:thepolymethodforGUE}, where the analogous expression in Lemma 
\ref{lem:polyencoding} was polynomial. This issue could be surmounted by 
applying Bernstein's inequality to the numerator and denominator 
separately using the chain rule; see, e.g., \cite[Lemma 4.3]{chen2024new}. 
However, this naive approach turns out to be lossy: it results in a 
multiplicative factor $m!$ in the bound on the $m$th derivative, which 
prevents us from reaching coefficients with exponential dimension.

Instead, we develop here a more specialized argument that exploits the 
fact that we have strong control over the denominator of the rational 
expressions.

\begin{lem}[Rational Bernstein inequality]
\label{lem:rationalbernstein}
Let $p,q\in\bN$ with $p\ge q$, let
$f\in \mathcal{P}_p$, and define the rational function 
$r:=\frac{f}{g_q}$ where $g_q$ is as defined in \eqref{eq:defofgL}. Then
$$
        \frac{1}{m!}\|r^{(m)}\|_{[-\frac{1}{cp},\frac{1}{cp}]}
        \le 
        \bigg( e^{-p} (Cp)^m  +
        \frac{(Cp)^{2m}}{m!}\bigg) \|r\|_{I_q}
$$
for all $m\ge 1$, where $c,C$ are universal constants and
$I_q := \{\frac{1}{N}:N\in\bZ,|N|>q\}$.
\end{lem}

The key idea behind the proof is that the function $\frac{1}{g_q}$ and its 
derivatives can be approximated very precisely by a polynomial.

\begin{lem}
\label{lem:polyratber}
For every $b\in\mathbb{N}$, 
there is a polynomial $s\in\mathcal{P}_{2bq}$ so that
$$
	\frac{1}{k!} 
	\|(\tfrac{1}{g_q}-s)^{(k)}\|_{[-\frac{1}{8q},\frac{1}{8q}]}
        \le 2^{-bq} (8q)^k
$$
for all $k\ge 0$. In particular, 
$\frac{1}{2}\frac{1}{g_q(x)} \le s(x) \le \frac{3}{2}\frac{1}{g_q(x)}$
for all $|x|\le \frac{1}{8q}$.
\end{lem}

\begin{proof}
Let $\frac{1}{g_q(z)} = \sum_{i=0}^\infty a_i z^i$ be the power series 
expansion of $\frac{1}{g_q}$ around zero. As
$|1-z^2| \ge 1-|z|^2 \ge e^{-2|z|^2}$ for $|z|\le\frac{1}{2}$,
we can estimate using \eqref{eq:defofgL}
$$
        \left|\tfrac{1}{g_q(z)}\right| \le
        e^{\sum_{j=1}^q 2j^2|z|^2\lfloor\frac{q}{j}\rfloor} \le
        e^{q^2(q+1)|z|^2} \le e^{q/2}
$$
for all $z\in\mathbb{C}$ with $|z|\le\frac{1}{2q}$. Thus
Lemma \ref{lem:taylorcoeff} yields
$|a_i| \le e^{q/2} (2q)^i$ for all $i$.

Now let $s(x):=\sum_{i=0}^{2bq} a_i x^i$. Then we can estimate
$$
        \frac{|(\frac{1}{g_q}-s)^{(k)}(x)|}{k!} \le
        e^{q/2} (8q)^k
        \sum_{i=\max\{2bq+1,k\}}^\infty 
        {i\choose k}
        \frac{1}{4^i} 
$$
for $|x|\le \frac{1}{8q}$, where we used 
$(\frac{1}{g_q}-s)^{(k)}(x)=
\sum_{i=\max\{2bq+1,k\}}^\infty a_i \frac{i!}{(i-k)!} x^{i-k}$.
The first part of the lemma follows as
${i\choose k}\le 2^i$ and $2^{-2}e^{1/2} \le \frac{1}{2}$. The second 
part follows from the first
using that $|\frac{1}{g_q(x)}-s(x)| \le \frac{1}{2} \le
\frac{1}{2} \frac{1}{g_q(x)}$ 
for $|x|\le \frac{1}{8q}$ as $g_q(x)\le 1$ 
\end{proof}

We can now complete the proof of Lemma \ref{lem:rationalbernstein}.

\begin{proof}[Proof of Lemma \ref{lem:rationalbernstein}]
Let $s$ be the polynomial of Lemma \ref{lem:polyratber} 
(we will choose $b\in\mathbb{N}$ at the end of the proof).
The product formula yields
$$
        \frac{r^{(m)}}{m!} =
        \frac{(fs)^{(m)}}{m!} +
        \sum_{k=0}^m \frac{f^{(k)}}{k!}
        \frac{(\frac{1}{g_q}-s)^{(m-k)}}{(m-k)!}.
$$
As $fs$ has degree $p' := p+2bq$, applying
Lemma \ref{lem:polybernstein} and
Proposition \ref{prop:uniform_bound_1/N} yields
\begin{align*}
        &\|f^{(k)}\|_{[-\frac{1}{cp},\frac{1}{cp}]} 
        \le C(Cp)^{2k} \|f\|_{I_q}
        \le C(Cp)^{2k} \|r\|_{I_q}, \\
        &\|(fs)^{(m)}\|_{[-\frac{1}{cp'},\frac{1}{cp'}]} 
        \le (Cp')^{2m}\|fs\|_{I_q}
        \le (Cp')^{2m}\|r\|_{I_q}.
\end{align*}
Here we used $|f|\le |r|$ in the last inequality on the first 
line, and that $|fs|\le \frac{3}{2}|r|$ by the second part of Lemma 
\ref{lem:polyratber} in the last inequality on the second line.

To conclude the proof, note that
\begin{multline*}
        \sum_{k=0}^m \frac{\|f^{(k)}\|_{[-\frac{1}{cp'},\frac{1}{cp'}]}
        }{k!}
        \frac{\|(\frac{1}{g_q}-s)^{(m-k)}\|_{[-\frac{1}{cp'},\frac{1}{cp'}]
        }}{(m-k)!}
\\ 
        \le 
        2^{-bq} (Cp)^m \|r\|_{I_q}
        \sum_{k=0}^m \frac{p^{k}}{k!} 
        \le
        e^{p} 2^{-bq} (Cp)^m \|r\|_{I_q},
\end{multline*}
where we used the first part of Lemma \ref{lem:polyratber}. 
If we now choose $b=\lceil \frac{2}{\log 2}\frac{p}{q}\rceil$,
then $e^{p} 2^{-bq}\le e^{-p}$ and
$p'\le Cp$, and the proof is readily completed.
\end{proof}

\subsection{The master inequality} 
\label{sec:taylorstheoremHaarU}

We now have the necessary ingredients to prove the analogue of Lemma 
\ref{lem:ratesofconv} for the $\mathrm{U}(N)$ model. Note that as no 
truncation is needed, the result holds for all $q,N$ and not merely for 
$q\lesssim N$ as in the Gaussian setting.

\begin{lem}[Master inequality]
\label{lem:ratesHaarU}
There exists a linear functional $\mu_m$ on $\calP$ for every $m\in\bZ_+$ 
such that for every $q\in\bN$ and $h\in \calP_q$
\begin{equation}
\label{eq:infiboundHaarU}
	|\mu_m(h)| 
	\le 
	\bigg(
	(C\tilde q\tilde q_0)^m
	+\frac{(C\tilde q\tilde q_0)^{2m}}{m!} 
	\bigg)
	\sobnorm{h}{0}{[-K, K]},
\end{equation}
and such that for all $N\in\bN$
\begin{equation}
\label{eq:ratesHaarU}
	\left| \E[\tr h(\xn)]- \sum_{k=0}^{m-1} \frac{\mu_k(h)}{N^k} 
	\right| \leq 
	\frac{1}{N^m}
	\bigg(
	(C\tilde q\tilde q_0)^m
	+\frac{(C\tilde q\tilde q_0)^{2m}}{m!} 
	\bigg)
	\sobnorm{h}{0}{[-K, K]}.
\end{equation}
Here $C$ is a universal constant,
$\tilde q:=q(1+\log q)$, and $\tilde q_0:=q_0(1+\log q_0)$.
\end{lem}

\begin{proof}
Let $\Psi_h$ be as in Lemma \ref{lem:rationalencoding}. Then
the degree of the numerator of $\Psi_h$ is 
bounded by $p=\lfloor 3\tilde q\tilde q_0\rfloor$, and as
$\|\xn\|\le K$ a.s.\ we have
$$
	|\Psi_h(\tfrac{1}{N})|=|\E[\tr h(\xn)]|\le\|h\|_{[-K,K]}.
$$
Applying Lemma \ref{lem:rationalbernstein} therefore yields
$$
        \frac{1}{m!}\|\Psi_h^{(m)}\|_{[-\frac{1}{c\tilde q\tilde q_0},\frac{1}{c\tilde q\tilde q_0}]}
        \le 
        \bigg( (C\tilde q\tilde q_0)^m  +
        \frac{(C\tilde q\tilde q_0)^{2m}}{m!}\bigg) \|h\|_{[-K,K]},
$$
for every $m,N\in\bN$, where $C,c>0$ are universal 
constants. Now define
$$
	\mu_m(h) := \frac{\Psi_h^{(m)}(0)}{m!}.
$$
Then \eqref{eq:infiboundHaarU} follows immediately from the previous 
equation display for $m\ge 1$, while for $m=0$ we have 
$|\mu_0(h)|=\lim_{N\to\infty}|\E[\tr 
h(\xn)]|\le\|h\|_{[-K,K]}$. Moreover, 
\eqref{eq:ratesHaarU} follows for $\frac{1}{N}\le \frac{1}{c\tilde q\tilde 
q_0}$ by Taylor expanding $\Psi_h$ as in the proof of Lemma 
\ref{lem:ratesofconv}.

On the other hand, in the case $\frac{1}{c\tilde q\tilde q_0}<\frac{1}{N}$,
we first estimate
\begin{align*}
	\left| \E[\tr h(\xn)]- \sum_{k=0}^{m-1} \frac{\mu_k(h)}{N^k} 
	\right| 
	&\leq 
	\sum_{k=0}^{m-1} 
	\frac{1}{N^k}
	\bigg(
	(C\tilde q\tilde q_0)^k
	+\frac{(C\tilde q\tilde q_0)^{2k}}{k!} 
	\bigg)
	\sobnorm{h}{0}{[-K, K]}
\\
	&\leq 
	\frac{(C\tilde q\tilde q_0)^m}{N^m}
	\sum_{k=0}^{m-1} 
	\bigg(
	1
	+\frac{(C\tilde q\tilde q_0)^{k}}{k!} 
	\bigg)
	\sobnorm{h}{0}{[-K, K]}
\end{align*}
using the triangle inequality and \eqref{eq:infiboundHaarU} on the first 
line, and $1<\frac{c\tilde q\tilde q_0}{N}$ on the second line.
We now consider two cases. If $C\tilde q\tilde q_0<m$, the sum
on the second line
can be 
estimated by $m+e^{C\tilde q\tilde q_0}\le C^m$.
If $C\tilde q\tilde q_0\ge m$, we use that
$\frac{a^k}{k!}\le \frac{a^m}{m!}$ for $m\le a$ to estimate the sum by
$m(1+\frac{(C\tilde q\tilde q_0)^{m}}{m!})$. In each case, the proof is 
readily concluded.
\end{proof}

\subsection{Extension to smooth functions}
\label{sec:extUN}

We now proceed to prove an analogue of Theorem~\ref{thm:smasympexGUE} for 
the $\mathrm{U}(N)$ model. The proof in the present setting is somewhat 
simpler as Lemma \ref{lem:ratesHaarU} holds without constraint on $q$.

\begin{thm}[Smooth asymptotic expansion for $\mathrm{U}(N)$]
\label{thm:smasympexUN} 
There is a universal constant $C>0$, and a compactly supported 
distribution $\mu_k$ for every $k\in\bZ_+$, such that the following 
hold. Fix any $h\in C^\infty(\bR)$, and let
$f(\theta):= h(K\cos\theta)$. Then
$$
	\left| \E[\tr h(\xn)]- \sum_{k=0}^{m-1} \frac{\mu_k(h)}{N^k} 
	\right| \leq 
	\frac{(C\tilde q_0)^m}{u^mN^m}
	\|f^{(m'+1)}\|_{[0,2\pi]} 
	+\frac{(C\tilde q_0)^{2m}}{m!\, u^{2m} N^m} 
	\|f^{(2m'+1)}\|_{[0,2\pi]} 
$$
for every $m, N\in \bN$ and $u\in (0,1)$, where
$m':=\lceil (1+u)m\rceil$ and $\tilde q_0 := q_0(1+\log q_0)$.
\end{thm}

We will need a variant of Lemma \ref{lem:zygmund} with logarithmic terms.

\begin{lem} 
\label{lem:zygmundwithlogs}
Fix $k,l\in\bZ_+$ and $u>0$, and let $s := \lceil k+ul\rceil$.
Let $h\in\mathcal{P}_q$ with Chebyshev expansion
$h(x) = \sum_{j=0}^q a_j T_j(K^{-1}x)$, and let
$f(\theta):= h(K \cos(\theta))$.  Then
$$
	\sum_{j=1}^q j^k (1+\log j)^l  
	|a_j| \lesssim
	\bigg(\frac{e^u}{u}\bigg)^l
	\sobnorm{f^{(s+1)}}{0}{[0, 2\pi]}.
$$
\end{lem}
 
\begin{proof}
This follows from Lemma \ref{lem:zygmund} using
$1+\log j = \frac{1}{u}\log((ej)^u) \le \frac{e^u}{u} j^u$.
\end{proof}

We can now prove Theorem \ref{thm:smasympexUN}.

\begin{proof}[Proof of Theorem \ref{thm:smasympexUN}]
We first note that the linear functional $\mu_m$ of 
Lemma~\ref{lem:ratesHaarU} extends to a compactly supported distribution 
for every $m\in\bZ_+$ by Lemma \ref{lem:extensiontoadistribution} and the 
inequality \eqref{eq:infiboundHaarU}. Now fix $h\in\mathcal{P}_q$ with
Chebyshev expansion $h(x) = \sum_{j=0}^q a_j T_j(K^{-1}x)$. Then
\eqref{eq:ratesHaarU} and the triangle inequality yield
\begin{multline*}
	\left| \E[\tr h(\xn)]- \sum_{k=0}^{m-1} \frac{\mu_k(h)}{N^k} 
	\right|
\\	 \leq 
	\frac{(C\tilde q_0)^m}{N^m}
	\sum_{j=0}^q
	j^m(1+\log j)^m |a_j|
	+
	\frac{(C\tilde q_0)^{2m}}{m!N^m} 
	\sum_{j=0}^q j^{2m}(1+\log j)^{2m}|a_j|.
\end{multline*}
Applying Lemma \ref{lem:zygmundwithlogs} and noting that
$e^u\le C$ for $u\le 1$ yields the conclusion for any $h\in\mathcal{P}$.
The general result follows as polynomials are dense in $h\in 
C^\infty(\mathbb{R})$.
\end{proof}

\subsection{Bootstrapping} 
\label{sec:haarbootstrap}

We now aim to prove the following.

\begin{prop}
\label{prop:supportofinfsHaarU}
In the setting of Theorem \ref{thm:smasympexUN}, we have
$$
	\supp\mu_m \subseteq[-\|\xf\|, \|\xf\|]
	\quad\text{for all }
	m\in\bZ_+.
$$
\end{prop}

The proof of the analogous Gaussian result in Proposition 
\ref{prop:supportofinfsGUE} transfers to the present setting with minimal 
modifications. The main input we need is an appropriate concentration 
inequality for Haar unitary matrices.

\begin{lem}
\label{lem:HaarConcentration}
For any $\varepsilon>0$ and $N\in\bN$, we have
$$
	\P[|\|\xn\|-\mathrm{med}(\|\xn\|)|>\varepsilon] \leq 
	Ce^{-c_P N \varepsilon^2}
$$
for a constant $c_P>0$ that depends only on $P$ and a universal 
constant $C>0$. 
\end{lem}

\begin{proof}
As $\|\xn\|=\|P(\boldsymbol{U}^N,\boldsymbol{U}^{N*})\|$ is 
$L_P$-Lipschitz as a function of $\haar_1,\ldots,\haar_r$ where $L_P$ 
depends only on $P$, concentration around 
the mean follows directly from \cite[Theorem 5.17]{Mec19}.
It is a standard fact that the concentration around the mean is equivalent 
to concentration around the median \cite[Proposition 1.8]{Led01}.
\end{proof}

The proof of Proposition \ref{prop:supportofinfsHaarU} now
follows exactly as in the Gaussian case.

\begin{proof}[Proof of Proposition \ref{prop:supportofinfsHaarU}]
All the proofs in sections \ref{sec:guebasiccvg} and 
\ref{sec:proofofsupportofinfsGUE} extend directly to the present setting, 
provided that we replace Theorem \ref{thm:smasympexGUE} and
Lemmas \ref{lem:polyencoding}~and~\ref{lem:concentrationaroundthemean} by
Theorem \ref{thm:smasympexUN} and Lemmas 
\ref{lem:rationalencoding}~and~\ref{lem:HaarConcentration}, respectively.
\end{proof}

\subsection{Proof of Theorem \ref{thm:mainhaar}: $\mathrm{U}(N)$ case}
\label{sec:pfmainHaarU}

We can now prove Theorem \ref{thm:mainhaar} in the case 
that $\boldsymbol{U}^N$ are Haar-distributed in $\mathrm{U}(N)$.
We first note the following.

\begin{lem}
\label{lem:cstarred}
We can estimate
$K := \|P\|_{\mathrm{M}_D(\bC)\otimes C^*(\bF_r)} \le (Cr)^{q_0}\|\xf\|$.
\end{lem}

\begin{proof}
We can represent the noncommutative polynomial $P$ as
$$
	P(\boldsymbol{u},\boldsymbol{u}^*)  =
	\sum_w A_w \otimes w(u_1,\ldots,u_r),
$$
where $A_w\in \mathrm{M}_D(\bC)$ are matrix coefficients and
where the sum is over all reduced words in 
$\boldsymbol{u},\boldsymbol{u}^*$ of length at most $q_0$.
As there are at most $(Cr)^{q_0}$ such words, we immediately obtain
$K\le (Cr)^{q_0}\max_w \|A_w\|$ by the triangle inequality.
The conclusion follows as in the second part of the proof of Lemma 
\ref{lem:chebsecondkind} using that the reduced words 
$\{w(u_1,\ldots,u_r)\}$ are orthonormal in $L^2(\tau)$.
\end{proof}

We now complete the proof. The argument is very similar to the one used in 
section \ref{sec:pfmaingue}, except that we must take care of the 
logarithmic terms.

\begin{proof}[Proof of Theorem \ref{thm:mainhaar}: $\mathrm{U}(N)$ case]
Let $\varepsilon\in[\frac{2}{\sqrt{N}},1]$, and
fix $m\in\bN$ and $u\in[\frac{1}{m},1]$ that will be chosen at the end of 
the proof. Define $m':=\lceil (1+u)m\rceil$.

Let $h$ be the test function provided by Lemma \ref{lem:testfunctions} 
with $m\leftarrow 2m'$, $\rho\leftarrow \|\xf\|$, and 
$\varepsilon\leftarrow\varepsilon\|\xf\|$.
Then $\mu_k(h)=0$ for all $k\in\bZ_+$ by Proposition 
\ref{prop:supportofinfsHaarU}. Thus
\begin{multline*}
	\mathbf{P}[\|\xn\| \ge (1+\varepsilon)\|\xf\|]
	\le
	\mathbf{E}[\trace h(\xn)] =
	DN\,\mathbf{E}[\tr h(\xn)]
\\
	\le
	\frac{(Cr)^{q_0}DN}{\varepsilon}
	\bigg[
	\bigg(\frac{(Cr)^{3q_0} \tilde q_0  m^{1+2u}
	}{\varepsilon^{1+2u}uN}
	\bigg)^m
	+
	\bigg(
	\frac{
	(Cr)^{6q_0}
	\tilde q_0^2 m^{1+4u} 
	}{\varepsilon^{2(1+2u)}u^2 N
	}
	\bigg)^m
	\bigg]
\end{multline*}
for a universal constant $C$
by Theorem \ref{thm:smasympexUN}, Lemma \ref{lem:testfunctions}, and 
Lemma \ref{lem:cstarred},
where $\tilde q_0 := q_0(1+\log q_0)$ and we used that
$m'\le (1+2u)m \le 3m$ and $\frac{1}{m!}\le(\frac{e}{m})^m$.

Now assume that $D\le e^m$, and choose $u=\frac{1}{1+\log m}$
and $m=\lfloor \frac{N\varepsilon^2}{L\log^2(N\varepsilon^2)}\rfloor$.
Then $m(1+\log m)^2 \le \frac{N\varepsilon^2}{L}$ and
$\varepsilon^u \ge \frac{1}{C}$ (as $\varepsilon \ge \frac{2}{\sqrt{N}}$),
so the inequality simplifies to
$$
	\mathbf{P}[\|\xn\| \ge (1+\varepsilon)\|\xf\|]
	\le
	\frac{(Cr)^{q_0}N}{\varepsilon}
	\bigg(
	\frac{
	(Cr)^{6q_0}
	\tilde q_0^2 
	}{L}
	\bigg)^m.
$$
The conclusion follows by choosing $L=e(Cr)^{6q_0}\tilde q_0^2$ provided 
that $m\ge 2$, which is ensured by assuming that $\varepsilon\ge 
\frac{1}{c\sqrt{N}}$ for a constant $c$ that depends only on $L$.
\end{proof}

\section{Strong convergence for $\mathrm{O}(N)$ and $\mathrm{Sp}(N)$}
\label{sec:O(N)andSp(N)}

The aim of this section is to complete the proof of Theorem 
\ref{thm:mainhaar} for Haar-distributed random matrices in $\mathrm{O}(N)$ 
and $\mathrm{Sp}(N)$. As in the Gaussian case
(section~\ref{sec:supersymmetryGOEGSE}), most of the proof in the 
$\mathrm{U}(N)$ case extends \emph{verbatim} to the present setting, so 
that we will focus attention here only on the necessary modifications.

The following setting and notations will be fixed throughout this section. 
Let
$\boldsymbol{U}^N=(\haaro_1, \dots, \haaro_r)$
and
$\boldsymbol{V}^N=(\haars_1, \dots, \haars_r)$
be independent
Haar-distributed random matrices in $\mathrm{O}(N)$ and
$\mathrm{Sp}(N)$, respectively, and let
$\boldsymbol{u}=(u_1,\ldots,u_r)$ be free Haar unitaries.
We further fix a self-adjoint noncommutative polynomial
$P\in\mathrm{M}_D(\bC)\otimes \bC\langle 
x_1,\ldots,x_r,x_1^*,\ldots,x_r^*\rangle$ of degree $q_0$ with matrix 
coefficients of dimension $D$, and denote the random matrices of interest 
and their limiting model as
$$
	\xn := P(\boldsymbol{U}^N,\boldsymbol{U}^{N*}),\qquad
	\yn := P(\boldsymbol{V}^N,\boldsymbol{V}^{N*}),\qquad
	\xf := P(\boldsymbol{u},\boldsymbol{u}^*).
$$
Finally, we define $K$ as in section \ref{sec:Unitarymatrices}.

\subsection{Polynomial encoding and supersymmetric duality}

We begin by proving an analogue of Lemmas 
\ref{lem:polyencodingGOEGSE} and \ref{lem:rationalencoding} in the present 
setting. We will always denote by $g_q$ the polynomial 
defined in \eqref{eq:defofgL} without further comment.

\begin{lem}[Polynomial encoding]
\label{lem:rationalencodingOSp}
For every $h \in \calP_q$, there is a rational function of the form 
$\Psi_h:=\frac{f_h}{g_{qq_0}}$ with 
$f_h,g_{qq_0}\in\mathcal{P}_{\lfloor 6qq_0(1+\log qq_0)\rfloor}$
so that
\begin{alignat*}{2}
	&\E[\tr h(\xn)] &&= \Psi_h(\tfrac{1}{N}),\\
	&\E[\tr h(\yn)] &&= \Psi_h(-\tfrac{1}{2N})
\end{alignat*}
for all $N\in\bN$ such that $N>qq_0$.
\end{lem}

In the proof, we require the orthogonal counterparts of the 
Weingarten functions of Definition \ref{defn:unitaryweingarten}.
It follows from \cite[Theorem 3.1]{collins2009some} that these are given
by
\begin{equation}
\label{eq:weingartenfunctionO}
   \wgo_L(m_1,m_2, N) = 
   \sum_{\lambda \vdash L} 
   \frac{C_{\lambda,m_1,m_2}}
   {\prod_{(i,j) \in \lambda } (N + 2j-i-1)}
\end{equation}
for all $L,N\in\bN$ with $L\le N$ and $m_1,m_2\in\calM_{2L}$. Here
$\calM_{2L}$ denotes the set of perfect matchings (i.e., pair partitions) 
of $[2L]$, and $C_{\lambda,m_1,m_2}$ is a real constant that depends only 
on $\lambda,m_1,m_2$ whose precise form is irrelevant for our analysis.
The following lemma is the counterpart of Lemma \ref{lem:denominatorsofwg}
in the present setting.

\begin{lem}
\label{lem:denominatorsofwgO}
For every $m_1,m_2\in\calM_{2L}$, there exists
$a_{m_1,m_2}\in\mathcal{P}$ so that
$$
	\wgo_L(m_1,m_2, N)
	=\frac{a_{m_1,m_2}(N)}{N^L
	\prod_{k=1}^{2L} (N^2-k^2)^{\lfloor\frac{2L}{k}\rfloor}}
	\quad\text{for all }N>2L.
$$
\end{lem}

\begin{proof}
Fix a partition $\lambda\vdash L$.
Since $\lambda$ has at most $L$ columns and at most $L$ rows, we clearly 
have $-L \leq 2j-i-1\leq 2L$ for all $(i,j) \in \lambda$. Thus
$$
	\prod_{(i,j) \in \lambda} (N+ 2j-i-1) =
	\prod_{k=-2L}^{2L}
	(N+k)^{\omega_k(\lambda)},
$$
where $\omega_k(\lambda)$ denotes the number of $(i,j)\in\lambda$ with
$2j-i-1=k$. As any $(i,j)\in\lambda$ must satisfy $ij\le L$
(cf.\ the proof of Lemma \ref{lem:denominatorsofwg}), we can 
estimate
$$
	\omega_k(\lambda) \le \#\{ (i,j)\in [L]^2 : ij \le L,~
	2j-i-1=k\}.
$$
If $k>0$, we can further estimate
$$
	\omega_k(\lambda) \le \#\{ i\in [L] : i(k+i+1) \le 2L\}
	\le \lfloor \tfrac{2L}{k+2}\rfloor.
$$
Similarly, if $k\le 0$ we can estimate
$$
	\omega_k(\lambda) \le \#\{ j\in [L] : (2j-1+|k|)j \le L\}
	\le \lfloor \tfrac{L}{|k|+1}\rfloor.
$$
Therefore $\omega_0(\lambda) \le L$ and 
$\omega_k(\lambda) \le \lfloor \frac{2L}{|k|} \rfloor$ for all $k\ne 0$.

It follows from the above observations and \eqref{eq:weingartenfunctionO}
that the numerator of
\begin{multline*}
   \wgo_L(m_1,m_2, N) =  \\
   \frac{
   \sum_{\lambda \vdash L} C_{\lambda,m_1,m_2}
	N^{L-\omega_0(\lambda)} \prod_{k=1}^{2L} 
	(N-k)^{\lfloor\frac{2L}{k}\rfloor - \omega_{-k}(\lambda)}
	(N+k)^{\lfloor\frac{2L}{k}\rfloor - \omega_k(\lambda)}
   }{N^L \prod_{k=1}^{2L} (N^2-k^2)^{\lfloor\frac{2L}{k}\rfloor}}
\end{multline*}
is a real polynomial in $N$, concluding the proof.
\end{proof}

We can now complete the proof of Lemma \ref{lem:rationalencodingOSp}.

\begin{proof}[Proof of Lemma \ref{lem:rationalencodingOSp}]
The proof for the $\mathrm{O}(N)$ case is nearly identical to the
proof of Lemma \ref{lem:rationalencoding}. Specifically,
let $w(\haaro_1,\ldots,\haaro_r)$ be a reduced word of length 
$L\le qq_0$ 
in the Haar orthogonal matrices $\haar_i$ and their adjoints $\haaradj_i$.
If $w\ne\id$ and $w$ is even---that is, each $\haaro_i$ appears an 
even number of times (with or without adjoint)---then a rational 
representation of $\E[\tr w(\haaro_1,\ldots,\haaro_r)]$ as in the 
statement of the lemma follows by the identical argument as in the proof 
of Lemma \ref{lem:rationalencoding} by using \cite[Theorem 
3.4]{magee2024matrix} and Lemma \ref{lem:denominatorsofwgO} instead of
\cite[Theorem 2.8]{magee2019matrix} and
Lemma \ref{lem:denominatorsofwg}, respectively.
If $w$ is not even, then $\E[\tr w(\haaro_1,\ldots,\haaro_r)]=0$ as
$\haaro_i$ has the same distribution as $-\haaro_i$. The proof is now 
readily completed.

The proof for the $\mathrm{Sp}(N)$ case follows directly from the
$\mathrm{O}(N)$ case and the supersymmetric duality property given in
\cite[Theorem 1.2]{magee2024matrix}.
\end{proof}

\subsection{Asymptotic expansion}

With Lemma \ref{lem:rationalencodingOSp} in hand, we can now repeat the 
proof of Theorem \ref{thm:smasympexUN} with only trivial modifications.

\begin{thm}[Smooth asymptotic expansion for $\mathrm{O}(N)/\mathrm{Sp}(N)$]
\label{thm:smasympexOSp} 
There is a universal constant $C>0$, and a compactly supported 
distribution $\mu_k$ for every $k\in\bZ_+$, so that the following 
hold. Fix any $h\in C^\infty(\bR)$, and let
$f(\theta):= h(K\cos\theta)$. Then
\begin{multline*}
	\left| \E[\tr h(\xn)]- \sum_{k=0}^{m-1} \frac{\mu_k(h)}{N^k} 
	\right| 
	~\vee~
	\left| \E[\tr h(\yn)]- \sum_{k=0}^{m-1} \frac{\mu_k(h)}{(-2N)^k} 
	\right| 
\\
	\leq 
	\frac{(C\tilde q_0)^m}{u^mN^m}
	\|f^{(m'+1)}\|_{[0,2\pi]} 
	+\frac{(C\tilde q_0)^{2m}}{m!\, u^{2m} N^m} 
	\|f^{(2m'+1)}\|_{[0,2\pi]} 
\end{multline*}
for all $m, N\in \bN$ and $u\in (0,1)$, where
$m':=\lceil (1+u)m\rceil$ and $\tilde q_0 := q_0(1+\log q_0)$.
\end{thm}

\begin{proof}
The only modification that must be made to the arguments of sections
\ref{sec:taylorstheoremHaarU} and \ref{sec:extUN} is that we apply
Lemma~\ref{lem:rationalencodingOSp} instead of
Lemma~\ref{lem:rationalencoding} in the proof of Lemma 
\ref{lem:ratesHaarU}; the master inequalities for $\mathrm{O}(N)$ and
$\mathrm{Sp}(N)$ are then obtained by Taylor expanding 
$\Psi_h(\frac{x}{2})$ at the points $x=\frac{2}{N}$ and $x=-\frac{1}{N}$,
respectively.
\end{proof}

\subsection{Bootstrapping}

We now aim to prove the following.

\begin{prop}
\label{prop:supportofinfsOSp}
In the setting of Theorem \ref{thm:smasympexOSp}, we have
$$
	\supp\mu_m \subseteq[-\|\xf\|, \|\xf\|]
	\quad\text{for all }
	m\in\bZ_+.
$$
\end{prop}

\begin{proof}
For $m=0,1$, the proofs of section
\ref{sec:nu1supersymmetry} extend \emph{verbatim} to the present setting, 
provided that we use Theorem \ref{thm:smasympexOSp} instead of Theorem 
\ref{thm:smasympexGOEGSE}.

For $m\ge 2$, we may consider only the $\mathrm{Sp}(N)$ case without loss 
of generality, as the expansion in the $\mathrm{O}(N)$ case is defined by 
the same distributions $\mu_m$. The proof is then identical to that of 
Proposition \ref{prop:supportofinfsHaarU},
provided that we replace Theorem \ref{thm:smasympexUN} and Lemma
\ref{lem:rationalencoding}
by Theorem \ref{thm:smasympexOSp} and Lemma
\ref{lem:rationalencodingOSp}, respectively, and we note that the proof of
Lemma~\ref{lem:HaarConcentration} extends \emph{verbatim} to the 
$\mathrm{Sp}(N)$ model.
\end{proof}

\begin{rem}
Curiously, the proof of Lemma~\ref{lem:HaarConcentration} does not extend 
directly to the $\mathrm{O}(N)$ model, as $\mathrm{O}(N)$ has two disjoint 
connected components and thus cannot exhibit a Lipschitz concentration 
principle. This minor issue is easily surmounted, but we need not do so as 
we can work with $\mathrm{Sp}(N)$ without loss of generality.
\end{rem}

\subsection{Proof of Theorem \ref{thm:mainhaar}: 
$\mathrm{O}(N)/\mathrm{Sp}(N)$ case}
\label{sec:pfmainOSp}

The remainder of the proof of Theorem \ref{thm:mainhaar} is now 
essentially identical to the proof of the $\mathrm{U}(N)$ case.

\begin{proof}[Proof of Theorem \ref{thm:mainhaar}: 
$\mathrm{O}(N)/\mathrm{Sp}(N)$ case]
The proof in section \ref{sec:pfmainHaarU} extends \emph{verbatim} to the 
present setting, provided that Theorem \ref{thm:smasympexUN} and
Proposition \ref{prop:supportofinfsHaarU} are replaced by
Theorem \ref{thm:smasympexOSp} and
Proposition \ref{prop:supportofinfsOSp}, respectively.
\end{proof}

\section{Applications}
\label{sec:applications}

\subsection{Subexponential operator spaces}
\label{sec:subexpon}

The aim of this section is to prove Corollary~\ref{cor:gausssubexp}. Let 
us begin by recalling some basic definitions \cite[\S 4]{pisier2014random}.

\begin{defn}[Operator spaces]
\label{defn:subexp}
An \emph{operator space} is a closed subspace of a 
$C^*$-algebra. A finite-dimensional operator space $\mathbf{W}$ is called
\begin{enumerate}[$\bullet$]
\itemsep\smallskipamount
\item 
\emph{exact} if for every $C>1$, there exists $S\in\bN$ and a
linear embedding $u:\mathbf{W}\to\mathrm{M}_S(\mathbb{C})$ such that
for every $N\in\bN$ and $x\in \mathbf{W}\otimes\mathrm{M}_N(\bC)$, we have
$$
	\|(u\otimes\mathrm{id})(x)\| \le
	\|x\| \le C\|(u\otimes\mathrm{id})(x)\|;
$$
\item \emph{$C$-subexponential} if there exists
$D_N\in\bN$ with $D_N=e^{o(N)}$ and a linear embedding 
$f_N:\mathbf{W}\to\mathrm{M}_{D_N}(\mathbb{C})$ such that
for every $N\in\bN$ and $x\in \mathbf{W}\otimes\mathrm{M}_N(\bC)$, we have
$$
	\|(f_N\otimes\mathrm{id})(x)\| \le
	\|x\| \le C\|(f_N\otimes\mathrm{id})(x)\|.
$$
\end{enumerate}
An operator space $\mathbf{W}$ is called exact or $C$-subexponential if 
every finite-dimensional subspace of $\mathbf{W}$ is exact or 
$C$-subexponential, respectively. The subexponential constant of an
operator space
$\mathbf{W}$ is $C(\mathbf{W}) := \inf\{C:\mathbf{W}\text{ is }C\text{-subexponential}\}$.
\end{defn}

An important fact that will be used in the sequel is that the 
$C^*$-algebra $\mathcal{A}$ generated by a free semicircular family 
$\boldsymbol{s}=(s_1,\ldots,s_r)$ is 
exact (this follows, for example, from \cite[Corollary 17.10]{Pis03} and 
the proof of \cite[Theorem 2.4]{Haa14}). Examples of non-exact 
subexponential operator spaces are given in \cite{pisier2014random}.

\begin{rem}
Given a noncommutative polynomial
$Q\in\mathbf{W}\otimes\mathbb{C}\langle\boldsymbol{s}\rangle$
with coefficients in an operator space $\mathbf{W}$, we will always
view $Q(\boldsymbol{s})$ as an element of the minimal tensor product
$\mathbf{W}\mathop{\otimes_{\rm min}}\mathcal{A}$ whose norm is
denoted as $\|\cdot\|_{\rm min}$; cf.\ \cite[\S 2.1]{Pis03}.
\end{rem}

We now turn to the proof of Corollary~\ref{cor:gausssubexp}. The main 
difficulty in the proof is in fact to prove the lower bound on 
$\|P_N(\boldsymbol{G}^N)\|$, as the upper bound is immediate from Theorem 
\ref{thm:maingauss}. 
For completeness, we begin by recalling the standard argument for
the case that $P_N=P$ is independent of $N$.

\begin{lem}
\label{lem:gausssubexplowerfixed}
Let $\boldsymbol{G}^N$ and $\boldsymbol{s}$ be as in Theorem
\ref{thm:maingauss}. 
For any $P\in\mathrm{M}_{D}(\mathbb{C})\otimes 
\mathbb{C}\langle\boldsymbol{s}\rangle$, we have
$\|P(\boldsymbol{G}^N)\|\ge (1-o(1))\|P(\boldsymbol{s})\|$
a.s.\ as $N\to\infty$.
\end{lem}

\begin{proof}
We may assume that $P$ is self-adjoint 
(Remark \ref{rem:selfadjoint}). Lemma \ref{lem:waf} yields
$$
	\E[\|P(\boldsymbol{G}^N)\|^{2p}] \ge
	\E[\tr P(\boldsymbol{G}^N)^{2p}] =
	(1+o(1))\,
	({\tr\otimes\tau})(P(\boldsymbol{s})^{2p})
$$
as $N\to\infty$ for every $p\in\bN$. Thus
Lemma 
\ref{lem:concentrationaroundthemean} and the Borel-Cantelli lemma yield
$$
	\|P(\boldsymbol{G}^N)\|^{2p} \ge (1+o(1))\,
	({\tr\otimes\tau})(P(\boldsymbol{s})^{2p})\quad\text{a.s.}
$$
as $N\to\infty$.
It 
remains to note that 
$[({\tr\otimes\tau})(P(\boldsymbol{s})^{2p})]^{\frac{1}{2p}}
\to \|P(\boldsymbol{s})\|$ as $p\to\infty$.
\end{proof}

We can now complete the proof of Corollary~\ref{cor:gausssubexp}.

\begin{proof}[Proof of Corollary~\ref{cor:gausssubexp}]
Let $P_N$ and $D_N=e^{o(N)}$ be as in the statement.
Applying Theorem \ref{thm:maingauss} with $P\leftarrow P_N$
and $\varepsilon\leftarrow\varepsilon_N:=
\max\{(\frac{\log D_N}{cN})^{1/2},N^{-1/4}\}=o(1)$ yields
$$
	\mathbf{P}\big[
	\|P_N(\boldsymbol{G}^N)\|\ge (1+\varepsilon_N)\|P_N(\boldsymbol{s})\|
	\big] \le
	\frac{N^{5/4}}{c}
	\,e^{-cN^{1/2}}.
$$
Thus
$\|P_N(\boldsymbol{G}^N)\|\le (1+o(1))\|P_N(\boldsymbol{s})\|$
a.s.\ by the Borel-Cantelli lemma.

To prove the corresponding lower bound, note first that by a compactness 
argument, the conclusion of Lemma \ref{lem:gausssubexplowerfixed} extends 
readily to the case that $P_N$ may depend on $N$ but with both degree 
$O(1)$ and matrix coefficients of dimension $D_N=O(1)$. The lower bound 
for general $D_N$ reduces to the case $D_N=O(1)$ by 
\cite[Lemma~5.12]{magee2024quasiexp} and the fact that the $C^*$-algebra 
generated by 
$\boldsymbol{s}$ is exact. We have therefore proved the lower bound
$\|P_N(\boldsymbol{G}^N)\|\ge (1-o(1))\|P_N(\boldsymbol{s})\|$
a.s.

For the second part, let $\mathbf{W}$ be a subexponential operator space, 
and fix any $Q\in\mathbf{W}\otimes\mathbb{C}\langle\boldsymbol{s}\rangle$ and
$C>C(\mathbf{W})$. Denote by $\mathbf{V}\subseteq\mathbf{W}$ the 
finite-dimensional operator space spanned by the coefficients of $Q$.
For every $N\in\bN$, let 
$f_N:\mathbf{V}\to\mathrm{M}_{D_N}(\bC)$
with $D_N=e^{o(N)}$ be the embedding provided by 
Definition \ref{defn:subexp}. Then
$$
	(1-o(1))\|P_N(\boldsymbol{s})\|=
	\|P_N(\boldsymbol{G}^N)\| \le
	\|Q(\boldsymbol{G}^N)\| \le
	C\|P_N(\boldsymbol{G}^N)\|
	= C(1+o(1))\|P_N(\boldsymbol{s})\|
$$
a.s., where $P_N = (f_N\otimes\mathrm{id})(Q)
\in \mathrm{M}_{D_N}(\bC)\otimes 
\mathbb{C}\langle\boldsymbol{s}\rangle$. 

Now let $\mathbf{A}$ be the operator space spanned by the monomials of
$Q$, and let $a>1$. As the $C^*$-algebra generated by
$\boldsymbol{s}$ is exact, so is $\mathbf{A}$.
Let $S\in\bN$ and the embedding $u:\mathbf{A}\to\mathrm{M}_S(\bC)$
be as in Definition \ref{defn:subexp} with $C\leftarrow a$.
Then
$$
	\frac{1}{Ca}
	\|Q(\boldsymbol{s})\|_{\rm min}
	\le
	\frac{1}{C}
	\|(\mathrm{id}\otimes u)(Q(\boldsymbol{s}))\|
	\le
	\|P_N(\boldsymbol{s})\| \le
	a\|(\mathrm{id}\otimes u)(Q(\boldsymbol{s}))\|
	\le
	a\|Q(\boldsymbol{s})\|_{\rm min}
$$
for all $N\ge S$, where we used \cite[Proposition 2.1.1]{Pis03}
for the first and last inequality. We conclude the proof by
taking $a\downarrow 1$ and $C\downarrow C(\mathbf{W})$.
\end{proof}

\subsection{Improved rates for random permutations}
\label{sec:permu}

Let $\tilde\Pi_1^N,\ldots,\tilde\Pi_r^N$ be i.i.d.\ uniformly distributed 
random permutation matrices of dimension $N$, and denote by $\Pi_i^N := 
\tilde\Pi_i^N|_{1^\perp}$ their restriction to the orthogonal complement 
of invariant vector $1$. A breakthrough result of Bordenave and 
Collins \cite{BC19} shows that
$\boldsymbol{\Pi}^N=(\Pi_1^N,\ldots,\Pi_r^N)$ converges strongly to free 
Haar unitaries $\boldsymbol{u}=(u_1,\ldots,u_r)$.

As for the Gaussian and classical compact group ensembles, it is expected 
that the rate of convergence of 
$\|P(\boldsymbol{\Pi}^N,\boldsymbol{\Pi}^{N*})\|$ to 
$\|P(\boldsymbol{u},\boldsymbol{u}^*)\|$ is of order $N^{-2/3}$, in 
agreement with Tracy-Widom asymptotics. This was proved up to a 
logarithmic factor in the recent work \cite{HMY24} in the special case 
$P(\boldsymbol{x},\boldsymbol{x}^*)=x_1+x_1^*+\cdots+x_r+x_r^*$, but 
remains well outside the reach of current methods for general $P$. In this 
setting, the best known rate of $(\frac{\log N}{N})^{1/8}$ was proved in 
\cite{chen2024new}, considerably improving the $\frac{\log\log N}{\log N}$ 
rate established by Bordenave and Collins in \cite{BC24}.

The main result of this section is a further quantitative improvement.

\begin{thm}
\label{thm:permrate}  
For any noncommutative polynomial
$P\in\mathrm{M}_{D}(\mathbb{C})\otimes
\mathbb{C}\langle\boldsymbol{u},\boldsymbol{u}^*\rangle$, 
$$
	\|P(\boldsymbol{\Pi}^N,\boldsymbol{\Pi}^{N*})\| \le
	\|P(\boldsymbol{u},\boldsymbol{u}^*)\|
        + O_{\mathrm{P}}\Big(\big(\tfrac{\log N}{N}\big)^{1/6}\Big)
        \quad\text{as}\quad N\to\infty.
$$
\end{thm}

\smallskip

The result of Theorem \ref{thm:permrate} is only a modest improvement on 
that of \cite[\S 3.3]{chen2024new}. We include it here to illustrate that 
the methods of this paper yield quantitative improvements, essentially for 
free, even for models such as random permutations for which good 
concentration and duality properties are not available.

In the remainder of this section, we fix 
$P\in\mathrm{M}_{D}(\mathbb{C})\otimes 
\mathbb{C}\langle\boldsymbol{u},\boldsymbol{u}^*\rangle$ of degree $q_0$, 
and write $X^N:=P(\boldsymbol{\Pi}^N,\boldsymbol{\Pi}^{N*})$ and $X_{\rm 
F}:= P(\boldsymbol{u},\boldsymbol{u}^*)$. In this setting, the analogue of 
Lemma \ref{lem:rationalencoding} is stated in \cite[Lemma 
5.1]{chen2024new}: for every $h\in\mathcal{P}_q$, there is a rational 
function of the form $\Psi_h := \frac{f_h}{\tilde g_{qq_0}}$ with 
$f_h,\tilde g_{qq_0}\in\mathcal{P}_{\lfloor qq_0(1+\log r)\rfloor}$ so 
that
$$
	\E[\tr h(X^N)] = \Psi_h(\tfrac{1}{N})\quad
	\text{for all }N\ge qq_0,
$$
where 
$$
	\tilde g_q(x) := 
	\prod_{j=1}^{q-1} (1-jx)^{\min\{r,\lfloor\frac{q}{j+1}\rfloor\}}.
$$
We emphasize that we have no control of $\Psi_h(-x)$ here (cf.\ section 
\ref{sec:beyond}). Thus we must use the Markov rather 
than Bernstein inequality in the proof.

The improved rate of Theorem \ref{thm:permrate} arises by replacing the 
classical polynomial interpolation argument used in \cite{chen2024new} by 
the optimal interpolation inequality of Proposition 
\ref{prop:uniform_bound_1/N}, which enables us to interpolate $\Psi_h$ 
between the $\frac{1}{N}$ samples for $N\gtrsim M=qq_0(1+\log r)$. The 
difficulty that then arises is that $\tilde g_{qq_0}$ is not uniformly 
bounded away from zero on the interval $[0,\frac{1}{M}]$, so that the 
elementary rational Markov inequality of \cite[Lemma 4.3]{chen2024new} 
cannot be applied. To surmount this issue, we prove a variant of Lemma 
\ref{lem:rationalbernstein} in the present setting.

\begin{lem}[Rational Markov inequality]
\label{lem:rationalmarkov}
Let $p,q\in\bN$ with $p\ge q$, let
$f\in \mathcal{P}_p$, and define the rational function 
$r:=\frac{f}{\tilde g_q}$. Then
$$
        \frac{1}{m!}\|r^{(m)}\|_{[0,\frac{1}{cp}]}
        \le 
        \bigg( e^{-p} (Cp)^m  +
        \frac{(Cp)^{3m}}{(m!)^2}\bigg) 
	\sup_{N\ge q}|r(\tfrac{1}{N})|
$$
for all $m\ge 1$, where $c,C$ are universal constants.
\end{lem}

\begin{proof}
We first note that the statement and proof of Lemma \ref{lem:polyratber} 
extend readily to the setting where $g_q$ is replaced by $\tilde g_q$.
We proceed as in the proof of Lemma \ref{lem:rationalbernstein}.
Applying the Markov inequality \cite[Lemma 4.1]{chen2024new}
instead of Lemma \ref{lem:polybernstein} yields
\begin{align*}
        &k!\,\|f^{(k)}\|_{[-\frac{1}{cp},\frac{1}{cp}]} 
        \le C(Cp)^{3k} \|f\|_{I_q}
        \le C(Cp)^{3k} \|r\|_{I_q}, \\
        &m!\,\|(fs)^{(m)}\|_{[-\frac{1}{cp'},\frac{1}{cp'}]} 
        \le (Cp')^{3m}\|fs\|_{I_q}
        \le (Cp')^{3m}\|r\|_{I_q},
\end{align*}
where we now let $I_q=\{\frac{1}{N}:N\in\bN,N\ge q\}$ and we used that
$(2k-1)!!\ge k!$. The argument is readily concluded as in the proof of
Lemma \ref{lem:rationalbernstein} by noting that
\begin{multline*}
        \sum_{k=0}^m \frac{\|f^{(k)}\|_{[-\frac{1}{cp'},\frac{1}{cp'}]}
        }{k!}
        \frac{\|(\frac{1}{\tilde g_q}-s)^{(m-k)}\|_{[-\frac{1}{cp'},\frac{1}{cp'}]
        }}{(m-k)!}
\\ 
        \le 
        2^{-bq} (Cp)^m \|r\|_{I_q}
        \sum_{k=0}^m \frac{p^{2k}}{(k!)^2} 
        \le
        e^{2p} 2^{-bq} (Cp)^m \|r\|_{I_q},
\end{multline*}
where we used that $\sum_{k\ge 0} \frac{p^{2k}}{(k!)^2} \le
(\sum_{k\ge 0} \frac{p^k}{k!})^2=e^{2p}$.
\end{proof}

We can now complete the proof of Theorem \ref{thm:permrate}.

\begin{proof}[Proof of Theorem \ref{thm:permrate}]
We can repeat the proof of \cite[Theorem 3.9]{chen2024new} 
\emph{verbatim}, with the only modification that we use
Lemma \ref{lem:rationalmarkov} to obtain the estimate
$$
        \frac{\|\Psi_h^{(m)}\|_{C^0[0,\frac{1}{M}]}}{m!}
        \le (Cqq_0(1+\log r))^{3m} 
        \|h\|_{[-K,K]}	
$$
in the proof of \cite[Theorem 6.1 and Corollary 6.2]{chen2024new},
instead of the corresponding estimate in \cite{chen2024new} that has the 
exponent $4m$ rather than $3m$.
\end{proof}

\subsection{Hayes' model}
\label{sec:hayes}

Fix $L,r\in\bN$, and let $G_k^{N,i}$
be independent GUE matrices of dimension $N$ for $i\in [L]$, $k\in[r]$.
The aim of this section is to investigate whether
the family of $N^L$-dimensional random matrices
$$
	\boldsymbol{\tilde G}^N := 
	\big\{ \id_N^{\otimes (i-1)}\otimes G_k^{N,i}\otimes
	\id_N^{\otimes (L-i)} :
	i\in [L],~k\in[r]\big\}
$$
converges strongly as $N\to\infty$ to
$$
	\boldsymbol{\tilde s} := 
	\big\{ \id^{\otimes (i-1)}\otimes s_k\otimes
	\id^{\otimes (L-i)} :
	i\in [L],~k\in[r] \big\}
$$
in $(\mathcal{A}^{\otimes_{\rm min}L},\tau^{\otimes L})$, where
$\mathcal{A}$ denotes the $C^*$-algebra generated by a free semicircular 
family $\boldsymbol{s}=(s_1,\ldots,s_r)$ with respect to the trace $\tau$.
This question was considered in an influential paper of
Hayes \cite{hayes2020random}, whose main result states that strong
convergence of this model in the case $L=2$ implies an affirmative answer 
to a conjecture of Peterson and Thom in the theory of von Neumann 
algebras.

That Hayes' question does indeed have an affirmative answer has now been 
established by a variety of methods 
\cite{BC22,BC24,magee2024quasiexp,Par24}, proving the Peterson-Thom 
conjecture. We provide yet another proof as a 
consequence of Corollary \ref{cor:gausssubexp}.

\begin{lem}
\label{lem:qualihayes}
For any $P\in\mathrm{M}_D(\bC)\otimes\bC\langle
\boldsymbol{\tilde s}\rangle$, we have
$$
	\|P(\boldsymbol{\tilde G}^N)\| =
	(1+o(1))\|P(\boldsymbol{\tilde s})\|_{\rm min}
	\text{ a.s.\ as }N\to\infty.
$$
\end{lem}

\begin{proof}
Denote by $\boldsymbol{\tilde G}^N_{\le j}$ the subset of
$\boldsymbol{\tilde G}^N$ with index $1\le i\le j$,
and denote by $\boldsymbol{\tilde s}_{>j}$ the subset of 
$\boldsymbol{\tilde s}$ with index $j<i\le L$.
It clearly suffices to show that
$$
	\|P(\boldsymbol{\tilde G}^N_{\le j},\boldsymbol{\tilde s}_{>j})\|_{\rm min}
	= (1+o(1))
	\|P(\boldsymbol{\tilde G}^N_{\le j-1},\boldsymbol{\tilde 
	s}_{>j-1})\|_{\rm min}\text{ a.s.\ as }N\to\infty
$$
for $j=1,\ldots,L$. To this end, note that
conditionally on $\boldsymbol{G}^{N,j}=(G_1^{N,j},\ldots,G_r^{N,j})$,
we can interpret $P(\boldsymbol{\tilde G}^N_{\le j},\boldsymbol{\tilde 
s}_{>j})$ as a noncommutative polynomial $P_N(\boldsymbol{G}^{N,j})$
with coefficients in $\mathrm{M}_N(\mathbb{C})^{\otimes(j-1)}\otimes
\mathbf{A}$, where $\mathbf{A}\subseteq\mathcal{A}^{\otimes_{\rm min}(L-j)}$
denotes the operator space spanned by the monomials of
$\boldsymbol{\tilde s}_{>j}$ that appear in $P_N$.

As exactness is stable under the
minimal tensor product \cite[p.\ 297]{Pis03}, it follows that
$\mathbf{A}$ is exact. Fix any $C>1$, and let
$u:\mathbf{A}\to \mathrm{M}_S(\mathbb{C})$ be the embedding provided by 
Definition \ref{defn:subexp}. As the polynomial $(\text{id}\otimes 
u)(P_N)$ has matrix coefficients of dimension $D_N= N^{j-1}S=e^{o(N)}$, 
applying Corollary \ref{cor:gausssubexp} conditionally yields
\begin{align*}
	\|P_N(\boldsymbol{G}^{N,j})\| &\le
	C\|(\text{id}\otimes u)(P_N(\boldsymbol{G}^{N,j}))\| =
	(1+o(1))
	C\|(\text{id}\otimes u)(P_N(\boldsymbol{s}))\|
	\\
	&\le
	C(1+o(1))\|P_N(\boldsymbol{s})\|_{\rm min} =
	C(1+o(1))
	\|P(\boldsymbol{\tilde G}^N_{\le j-1},\boldsymbol{\tilde s}_{>j-1})\|_{\rm min}
\end{align*}
a.s.\ as $N\to\infty$, where we used 
\cite[Proposition 2.1.1]{Pis03} in the second inequality.
Taking $C\downarrow 1$ yields an upper bound of the 
desired form. The corresponding lower bound follows in a completely 
analogous fashion, concluding the proof.
\end{proof}

Even though Lemma \ref{lem:qualihayes} provides a short proof of strong 
convergence of Hayes' model, the exactness argument provides no 
quantitative information. The main result of this section is the following 
quantitative form of Lemma \ref{lem:qualihayes}.

\begin{thm}
\label{thm:hayes}
Let $\varepsilon\in (0,1]$, and fix
$P\in\mathrm{M}_D(\mathbb{C})\otimes 
\mathbb{C}\langle \boldsymbol{\tilde s}\rangle$ of degree $q_0$ 
with matrix coefficients of dimension $D\le e^{cN\varepsilon^2}$.
Then we have 
$$
	\mathbf{P}\big[
	\|P(\boldsymbol{\tilde G}^N)\|
	\ge (1+\varepsilon)
	\|P(\boldsymbol{\tilde s})\|_{\rm min}
	\big] \le
	\frac{N^L}{c\varepsilon}
	\,e^{-cN\varepsilon^2},
$$
where $\frac{1}{c}=(CLr)^{2q_0}q_0^2$ for a universal constant $C$.
\end{thm}

The proof of Theorem \ref{thm:hayes} is very similar to that of Theorem 
\ref{thm:maingauss}, with one twist. In our main results, we could deduce 
qualitative strong convergence merely from the fact that 
$\supp\nu_1\subseteq[-\|\xf\|,\|\xf\|]$; this was then used as input to 
the bootstrapping argument (cf.\ section \ref{sec:introboot}) to control 
the remaining $\nu_k$. However, as the random matrices in the present 
section are $N^L$-dimensional rather than $N$-dimensional, we would need 
to control $\nu_1,\ldots,\nu_L$ to prove strong convergence. In contrast 
to $\nu_1$, control of $\nu_2,\ldots,\nu_L$ cannot be achieved by 
supersymmetric duality, and would ordinarily require a problem-specific 
analysis as in \cite{chen2024new}.

Fortunately, as we already established strong convergence in a different 
manner in Lemma \ref{lem:qualihayes}, we can use the latter as input to 
the bootstrapping argument and avoid any additional computations.
The proof of Theorem \ref{thm:hayes} therefore illustrates the fact that 
the bootstrapping argument can be used to amplify a qualitative strong 
convergence result to a strong quantitative bound.

\begin{proof}[Proof of Theorem \ref{thm:hayes}]
Fix $P\in\mathrm{M}_D(\mathbb{C})\otimes \mathbb{C}\langle 
\boldsymbol{\tilde s}\rangle$ as in the statement, and define $\xn = 
P(\boldsymbol{\tilde G}^N)$ and $\xf = P(\boldsymbol{\tilde s})$.
As any word 
$w(\boldsymbol{\tilde G}^N)$ of length $q$ has the form
$$
	w(\boldsymbol{\tilde G}^N) = 
	\big(G^{N,1}_{k_1}\cdots G^{N,1}_{k_{\ell_1}}\big)\otimes
	\big(G^{N,2}_{k_{\ell_1+1}}\cdots G^{N,2}_{k_{\ell_2}}\big)
	\otimes\cdots\otimes
	\big(G^{N,L}_{k_{\ell_{L-1}+1}}\cdots G^{N,L}_{k_q}\big)
$$
for some $0\le\ell_1\le\cdots\le\ell_{L-1}\le q$ and $k_1,\ldots,k_q\in[r]$,
$$
	\E[\tr w(\boldsymbol{\tilde G}^N)] = 
	\E[\tr
	G^{N,1}_{k_1}\cdots G^{N,1}_{k_{\ell_1}}]\,
	\E[\tr
	G^{N,2}_{k_{\ell_1+1}}\cdots G^{N,2}_{k_{\ell_2}}]
	\,\cdots\,
	\E[\tr
	G^{N,L}_{k_{\ell_{L-1}+1}}\cdots G^{N,L}_{k_q}]
$$
is a polynomial of $\frac{1}{N^2}$ of degree at most $\frac{q}{4}$ as is 
noted in the proof of Lemma \ref{lem:polyencoding}. In particular, the 
statement of Lemma \ref{lem:polyencoding} extends to the present setting.

With this observation in place, the entire proof of Theorem 
\ref{thm:maingauss} for GUE matrices carries over directly to the present 
setting with two minor modifications: we replace the argument of section 
\ref{sec:guebasiccvg} by Lemma \ref{lem:qualihayes}, and we note the 
correct normalization $\mathbf{E}[\trace h(\xn)] 
= DN^L\,\mathbf{E}[\tr h(\xn)]$
in the present setting in section \ref{sec:pfmaingue}.
\end{proof}

We have considered the GUE form of the Hayes model in this section for 
concreteness. It is straightforward to repeat the above arguments to 
obtain the analogous result for the GOE/GSE or 
$\mathrm{U}(N)/\mathrm{O}(N)/\mathrm{Sp}(N)$ ensembles.

\subsection{Tensor GUE models}
\label{sec:graph}

In Hayes' model of the previous section, independent GUE matrices act 
on disjoint factors of a tensor product, that is, they are 
non-interacting. In this section, we investigate tensor GUE models that 
admit a general interaction pattern. Such models arise naturally in the
study of quantum many-body systems \cite{ML19,CY24} and random 
geometry \cite{MT23}.

To this end, fix $L,V\in\bN$ and 
nonempty subsets $K_1,\ldots,K_V\subseteq[L]$. For every $v\in[V]$, we 
define an independent $N^L$-dimensional random matrix 
$$
	\hat G_v^N = H_v^N \otimes \id_{N^{|[L]\backslash K_v|}}
        \quad\text{in}\quad
        \bigotimes_{\ell\in[L]}\mathrm{M}_N(\mathbb{C}) \simeq 
        \bigotimes_{\ell\in K_v}\mathrm{M}_N(\mathbb{C})\otimes 
	\bigotimes_{\ell\in [L]\backslash K_v}\mathrm{M}_N(\mathbb{C}),
$$
where $H_v^N$ is GUE of dimension $N^{|K_v|}$. Thus
$\boldsymbol{\hat G}^N=(\hat G_v^N)_{v\in[V]}$ are independent GUE 
matrices that act on overlapping tensor factors.

To describe the limiting model, let $\Gamma=([V],E)$ be a finite simple 
graph. A \emph{$\Gamma$-independent semicircular family} is a family 
$\boldsymbol{\hat s}=(\hat s_v)_{v\in[V]}$ in a $C^*$-probability space 
with the following properties: each $\hat s_v$ is a semicircular variable;
$\hat s_v,\hat s_w$ are classically independent if $\{v,w\}\in E$; and 
$\hat s_v,\hat s_w$ are freely independent if $v\ne w$, $\{v,w\}\not\in 
E$. We refer to \cite[\S 3]{SW16} for the precise definition.
 
It was shown by Charlesworth and Collins \cite[Theorem 4]{CC21} that 
$\boldsymbol{\hat G}^N$ converges weakly to $\boldsymbol{\hat s}$ as 
$N\to\infty$ (in the sense of Lemma \ref{lem:waf}), where the graph 
$\Gamma$ is defined by placing an edge $\{v,w\}\in E$ if and only if 
$K_v\cap K_w=\varnothing$. We will fix this graph $\Gamma$ in the 
remainder of this section. It is readily seen that any finite simple 
graph $\Gamma$ can be realized in this manner, as is noted in \cite{CC21} 
and \cite{MT23}.

Whether $\boldsymbol{\hat G}^N$ converges strongly to $\boldsymbol{\hat 
s}$ as $N\to\infty$ has remained an open problem, cf.\ \cite[Problem 
1.6]{MT23} and \cite{CY24}. We resolve this problem here.

\begin{thm}
\label{thm:graph}
For every $P\in\mathrm{M}_D(\mathbb{C})\otimes
\mathbb{C}\langle\boldsymbol{\hat s}\rangle$, we have
$$
        \|P(\boldsymbol{\hat G}^N)\| = (1+o(1))
        \|P(\boldsymbol{\hat s})\|\quad\text{a.s.}\quad\text{as}\quad 
	N\to\infty.
$$
\end{thm}

\smallskip 

As the random matrices $\hat G_v^N$ have dimension $N^L$,
a direct application of the polynomial method to the present model would 
require us to control the supports of the infinitesimal distributions 
$\nu_1,\ldots,\nu_L$. Instead, we will take a different approach 
that circumvents the need for such an analysis.

The idea behind the proof is to use the central limit theorem to 
approximate the present model by the Hayes model of the previous section. 
To this end, we define in $\bigotimes_{\ell\in[L]}\mathrm{M}_N(\mathbb{C}) 
\simeq \bigotimes_{\ell\in K_v}\mathrm{M}_N(\mathbb{C})\otimes 
\bigotimes_{\ell\in [L]\backslash K_v}\mathrm{M}_N(\mathbb{C})$ the random 
matrix
$$
        \hat G_v^{N,T} =
        \frac{1}{\sqrt{T}} 
        \sum_{t=1}^T
        \big(G_{v,t}^{N,1} \otimes
        G_{v,t}^{N,2} \otimes\cdots \otimes
        G_{v,t}^{N,|K_v|}\big)\otimes \id_{N^{|[L]\backslash K_v|}},
$$
where $(G_{v,t}^{N,i})_{v,t,i}$ are independent GUE matrices  of 
dimension $N$.
Similarly, we define in\footnote{All tensor products of $C^*$-algebras in this
section are minimal tensor products.}
$\bigotimes_{\ell\in [L]}\mathcal{A} \simeq
\bigotimes_{\ell\in K_v}\mathcal{A}\otimes
\bigotimes_{\ell\in [L]\backslash K_v}\mathcal{A}$ the associated 
limiting model
$$
        \hat s_v^T =
        \frac{1}{\sqrt{T}} 
        \sum_{t=1}^T
        \big(s_{v,t} \otimes\cdots \otimes
        s_{v,t}\big)\otimes \id,
$$
where $(s_{v,t})_{v,t}$ is a free semicircular family in
$\mathcal{A}$.
In the following, we will write
$\boldsymbol{\hat G}^{N,T}=(\hat G_v^{N,T})_{v\in [V]}$ and
$\boldsymbol{\hat s}^T=(\hat s^T_v)_{v\in [V]}$.

\begin{lem}
\label{lem:graphcov}
The mean and covariance of the real and imaginary parts of the entries
of $\boldsymbol{\hat G}^{N,T}$ coincide with those of 
$\boldsymbol{\hat G}^N$ for every $T\in\bN$.
\end{lem}

\begin{proof}
Let $G = G^{1}\otimes\cdots\otimes G^{r}$ where
$G^{1},\ldots,G^{r}$ are independent GUE matrices of dimension $N$, 
and let $G'$ be a GUE matrix of dimension $N^r$ acting on 
$(\mathbb{C}^N)^{\otimes r}$. Both $G$ and $G'$ have zero mean
and are self-adjoint, and
\begin{multline*}
        \mathbf{E}\big[
        G_{(i_1,\ldots,i_r),(j_1,\ldots,j_r)}
        \bar G_{(k_1,\ldots,k_r),(l_1,\ldots,l_r)}\big] =
        \mathbf{E}\big[
        G^{1}_{i_1,j_1}
        \bar G^{1}_{k_1,l_1}\big]\cdots
        \mathbf{E}\big[
        G^{r}_{i_r,j_r}
        \bar G^{r}_{k_r,l_r}\big] \\
        = \frac{1_{i_1=k_1}1_{j_1=l_1}\cdots
        1_{i_r=k_r}1_{j_r=l_r}}{N^r} =
        \mathbf{E}\big[
        G_{(i_1,\ldots,i_r),(j_1,\ldots,j_r)}'
        \bar G_{(k_1,\ldots,k_r),(l_1,\ldots,l_r)}'\big].
\end{multline*} 
Thus the real and imaginary parts of the entries of
$G$ and $G'$ have the same mean and covariance. The proof is
readily completed.
\end{proof}

Lemma \ref{lem:graphcov} ensures that $\boldsymbol{\hat G}^{N,T}$ 
converges in distribution to $\boldsymbol{\hat G}^N$ as $T\to\infty$ by 
the central limit theorem. The following lemma yields the corresponding 
property for the limiting models; its proof will be given in Appendix 
\ref{app:tensor}.

\begin{lem}
\label{lem:graphlimit}
For every 
$P\in\mathrm{M}_D(\mathbb{C})\otimes
\mathbb{C}\langle\boldsymbol{\hat s}\rangle$, we have
$$
	\|P(\boldsymbol{\hat s}^T)\|=(1+o(1))\|P(\boldsymbol{\hat s})\|
	\quad\text{as}\quad T\to\infty.
$$
\end{lem}

\smallskip

On the other hand, strong convergence of
$\boldsymbol{\hat G}^{N,T}$ to $\boldsymbol{\hat s}^T$ as $N\to\infty$
follows from Lemma \ref{lem:qualihayes}, as
for any $P\in\mathrm{M}_D(\mathbb{C})\otimes
\mathbb{C}\langle\boldsymbol{\hat s}\rangle$ of degree $q_0$,
$P(\boldsymbol{\hat G}^{N,T})$ may be viewed as a noncommutative 
polynomial of degree at most $Lq_0$ of the Hayes model of the previous 
section with $r=VT$ free variables.

The above relations between the different models are summarized as 
follows:
\begin{center}
\begin{tikzpicture}
\draw (0,0) node[right] {$\boldsymbol{\hat G}^{N,T}$};
\draw[->] (1.05,-0.05) to (3,-0.05); 
\draw (3,0) node[right] {$\boldsymbol{\hat G}^{N}$};
\draw[->] (0.25,-.35) to (0.25,-1.7);
\draw (0,-2) node[right] {$\boldsymbol{\hat s}^{T}$};
\draw[densely dashed,->] (3.25,-.35) to (3.25,-1.7);
\draw (3,-2) node[right] {$\boldsymbol{\hat s}$};
\draw[->] (0.65,-2.05) to (3,-2.05);

\draw (1.9,0.2) node {\small CLT};
\draw (1.8,-1.8) node {\small Lem.\ \ref{lem:graphlimit}};
\draw (0,-1) node[rotate=90] {\small Lem.\ \ref{lem:qualihayes}};

\end{tikzpicture}
\end{center}
The key to the proof of Theorem \ref{thm:graph} is that we have strong 
quantitative forms of the convergence of $\boldsymbol{\hat G}^{N,T}$ both 
as $N\to\infty$ (by Theorem \ref{thm:hayes}) and as $T\to\infty$ (by the 
universality principle of \cite{BvH24}). This enables us to 
exchange the order of the limits $T\to\infty$ and $N\to\infty$ to deduce 
strong convergence of $\boldsymbol{\hat G}^{N}$ to $\boldsymbol{\hat s}$.

\begin{proof}[Proof of Theorem \ref{thm:graph}]
The lower bound $\|P(\boldsymbol{\hat G}^N)\|\ge 
(1+o(1))\|P(\boldsymbol{\hat s})\|$ a.s.\ follows readily from weak 
convergence \cite{CC21} and concentration of measure as in the proof of 
Lemma \ref{lem:gausssubexplowerfixed}.
It remains to prove the corresponding upper bound.

In order to apply the results of \cite{BvH24}, it is convenient to 
use a classical linearization trick of Haagerup and Thorbj{\o}rnsen: by 
\cite[Lemma 1]{HT05}, it suffices to show that
$$
        \mathrm{sp}(P(\boldsymbol{\hat G}^N))
        \subseteq \mathrm{sp}(P(\boldsymbol{\hat s}))
        +[-\varepsilon,\varepsilon]
        \quad\text{eventually as }N\to\infty\quad\text{ a.s.}
$$
for every $\varepsilon>0$, $D\in\mathbb{N}$, and self-adjoint 
$P\in\mathrm{M}_D(\mathbb{C})\otimes
\mathbb{C}\langle\boldsymbol{\hat s}\rangle$ of degree $q_0=1$. 
We will fix such a polynomial $P$ in the rest of the proof.

Now note that $P(\boldsymbol{\hat G}^{N,T})=
A_0\otimes\id_{N^L} + \frac{1}{\sqrt{T}}
\sum_{v,t} A_{v,t}\otimes Z_{v,t}$,
where each $Z_{v,t}$ is an independent tensor product of at most $L$
GUE matrices. In particular,
$$
	\mathbf{P}\Big[\max_{v,t}\|Z_{v,t}\| > C\Big] \le CT e^{-cN},
	\quad
	\mathbf{E}\Big[\max_{v,t}\|Z_{v,t}\|^2\Big]^{1/2}\le 
	C\bigg(1+\sqrt{\frac{\log T}{N}}\bigg)^L
$$
by Lemma \ref{lem:subgaussiannorm}, where the constant $C$ may depend 
on $L,V$ but not on $N,T$. We can therefore apply \cite[Theorem 
2.8]{BvH24} and Lemma \ref{lem:graphcov} to obtain
$$
        \mathbf{P}\big[
        \mathrm{sp}(P(\boldsymbol{\hat G}^N)) \subseteq
        \mathrm{sp}(P(\boldsymbol{\hat G}^{N,T}))+
        [-\varepsilon(t),\varepsilon(t)]
        \big]
        \ge
        1 - N^{L}e^{-t} - CT e^{-cN} 
$$
for all $t>0$ and $T\le e^{N}$,
where $\varepsilon(t) = 
C_P(N^{-1/2}t^{1/2}+T^{-1/12}t^{2/3}+T^{-1/4}t)$ and
$C_P$ is a constant that depends on $P$.
Choosing $t=(L+2)\log N$ yields
$$
        \mathrm{sp}(P(\boldsymbol{\hat G}^N)) \subseteq
        \mathrm{sp}(P(\boldsymbol{\hat G}^{N,T_N}))+
        [-\varepsilon,\varepsilon]
        \quad\text{eventually as }N\to\infty\quad\text{ a.s.}
$$
for every $\varepsilon>0$ by the Borel-Cantelli lemma, where
$T_N := \lceil \log^9 N\rceil$.

On the other hand, for any $h\in\mathcal{P}_q$, we may view
$h(P(\boldsymbol{\hat G}^{N,T_N}))$ as a noncommutative polynomial
of degree at most $Lq$ of the Hayes model of the previous section with
$r=VT_N=O(\log^9 N)$ variables. Thus Theorem \ref{thm:hayes}
and Borel-Cantelli yield
$$
        \|h(P(\boldsymbol{\hat G}^{N,T_N}))\| \le
        (1+o(1))\|h(P(\boldsymbol{\hat s}^{T_N}))\|
        \quad\text{a.s.}\quad\text{as}\quad N\to\infty
$$
for every $h\in\mathcal{P}$. By Lemma \ref{lem:graphlimit} and
\cite[Proposition 2.1]{CM14}, this implies
$$
        \mathrm{sp}(P(\boldsymbol{\hat G}^{N,T_N}))
        \subseteq
        \mathrm{sp}(P(\boldsymbol{\hat s})) 
        +[-\varepsilon,\varepsilon]
        \quad\text{eventually as }N\to\infty\quad\text{ a.s.}
$$
for every $\varepsilon>0$. Combining the above estimates concludes the 
proof.
\end{proof}

We emphasize that the above argument relies fundamentally on the 
quantitative form of strong convergence for the Hayes model provided by 
Theorem \ref{thm:hayes}. Indeed, $P(\boldsymbol{\hat G}^{N,T_N})$ defines 
a sequence of noncommutative polynomials in the Hayes model with an 
increasing number of variables $r=O(\log^9N)$. We achieve uniform control 
of the error as the constant in Theorem \ref{thm:hayes} depends 
polynomially on $r$.

Now that strong convergence of the tensor GUE model has been established 
in a qualitative sense, Theorem \ref{thm:graph} could be used as input to 
a bootstrapping argument to obtain a quantitative strong convergence 
theorem along the lines of Theorem~\ref{thm:hayes}. In another direction, 
strong convergence of the analogous tensor model where the GUE matrices 
are replaced by Haar unitary matrices can be deduced from 
Theorem~\ref{thm:graph} using the coupling method of Collins and Male 
\cite{CM14}. As these extensions do not require any new ideas, we omit the 
details.

\subsection{High-dimensional representations}

When we considered Haar-distributed random matrices in the classical 
compact groups of dimension $N$, we implicitly identified elements of the 
group with their fundamental representation as $N$-dimensional matrices. 
It is of considerable interest to understand strong convergence of other 
representations whose dimension may be much larger than $N$. The aim of 
this section is to prove the following result in this direction.

\begin{thm}
\label{thm:mageedlsalle}
Let $\boldsymbol{U}^N =  (U_1^N,\ldots,U_r^N)$ be i.i.d.\ 
Haar-distributed elements of $\mathrm{U}(N)$ and let
$\boldsymbol{u}=(u_1,\ldots,u_r)$ be free Haar unitaries.
Then
$$
	\lim_{N\to\infty}
	\sup_{1<\dim(\pi_N)\le \exp(N^{1/3-\delta})}
	\big|
	\|P(\pi_N(\boldsymbol{U}^N),\pi_N(\boldsymbol{U}^{N*}))\| 
	-
	\|P(\boldsymbol{u},\boldsymbol{u}^*)\|
	\big| =0\quad\text{a.s.}
$$
for every $\delta>0$, $D\in\mathbb{N}$, and
$P\in\mathrm{M}_D(\bC)\otimes\bC\langle 
\boldsymbol{u},\boldsymbol{u}^*\rangle$, where the supremum is taken
over all irreducible representations $\pi_N$ of $\mathrm{U}(N)$
with the stated dimensions.
\end{thm}

Theorem \ref{thm:mageedlsalle} improves a result of Magee and 
de la Salle \cite{magee2024quasiexp}, who prove the same statement for 
$\dim(\pi_N) \le \exp(N^{1/42-\delta})$. Our aim here is to showcase 
how the methods of this paper give rise to quantitative 
improvements.\footnote{%
	The direct analogue of Theorem \ref{thm:mageedlsalle} is stated in
	\cite[Corollary~1.2]{magee2024quasiexp} for random matrices and 	
	representations of $\mathrm{SU}(N)$ rather than $\mathrm{U}(N)$.
	The distinction is largely cosmetic, and 
	Theorem~\ref{thm:mageedlsalle} readily implies the analogous result for
	$\mathrm{SU}(N)$ as in \cite[\S 8]{magee2024quasiexp}.}

Recall that the 
distinct irreducible representations of $\mathrm{U}(N)$ are indexed 
by their highest weight vectors, which are $N$-tuples 
$\boldsymbol{z}=(z_1,\ldots,z_N)\in\mathbb{Z}^N$ such that $z_1\ge 
\cdots\ge z_N$. It will be convenient to parametrize $\boldsymbol{z}$ as
$$
	\boldsymbol{z} =
	(\lambda_1,\ldots,\lambda_{\ell(\lambda)},0,\ldots,0,
	-\mu_{\ell(\mu)},\ldots,-\mu_1),
$$
where $\lambda\vdash L$ and $\mu\vdash M$ are integer partitions 
with $\ell(\lambda)+\ell(\mu)\le N$; here and below, we denote by 
$\ell(\lambda)$ the number of elements of a partition $\lambda$. We denote 
the associated unitary representation of $\mathrm{U}(N)$ by 
$\pi_{\lambda,\mu}$, and we define for $N\ge L+M$
$$
	\pi_{L,M} := \bigoplus_{\substack{\lambda\vdash L \\
	\mu\vdash M}} \pi_{\lambda,\mu}.
$$
By a variant of Schur-Weyl duality \cite[\S 2.2]{Mag24}, we have 
$\dim(\pi_{L,M})\le N^{L+M}$. The significance of this definition is that, 
on the one hand, $\pi_{L,M}$ is stable 
(that is, the highest weight vectors of its irreducible components are
independent of $N$) and thus its spectral statistics have
rational expressions for $N\ge L+M$; while on the other hand,
every irreducible representation of
$\mathrm{U}(N)$ of dimension up to $e^{cR}$ is contained in
$\pi_{L,M}$ for some $L+M\le R$ up to a scalar factor.

\begin{lem}
\label{lem:pidimension}
Let $R<N$. Then every irreducible representation $\pi$ of 
$\mathrm{U}(N)$ of dimension $1<\dim(\pi) \le e^{cR}$ is of
the form $\pi(U) = \det(U)^k \pi_{\lambda,\mu}(U)$ for some
$k\in\mathbb{Z}$, $\lambda\vdash L$, $\mu\vdash M$ with
$1\le L+M \le R$, where $c$ is a universal constant.
\end{lem}

\begin{proof}
By \cite[Proposition 2.2]{magee2024quasiexp}, any $\pi$ with
$1<\dim(\pi)\le e^{cR}$ has highest weight 
$$
	(k+\lambda_1,\ldots,k+\lambda_{\ell(\lambda)},k,\ldots,k,
	k-\mu_{\ell(\mu)},\ldots,k-\mu_1)
$$
for some $k\in\mathbb{Z}$, $\lambda\vdash L$, $\mu\vdash M$ with
$1\le L+M\le R$. That this representation has the form given in the 
statement is classical, see, e.g., \cite[\S 13.2]{FH91}.
\end{proof}

In the following, we fix $L,M,N\in\bZ_+$ with $1\le L+M\le N$ 
and a noncommutative polynomial $P\in\mathrm{M}_D(\bC)\otimes\bC\langle
\boldsymbol{u},\boldsymbol{u}^*\rangle$ of degree $q_0$, and let
$$
	\xn := P(\pi_{L,M}(\boldsymbol{U}^N),
	\pi_{L,M}(\boldsymbol{U}^N)^*),\qquad\quad
	\xf := P(\boldsymbol{u},\boldsymbol{u}^*).
$$
Then we have the following analogue of Lemma \ref{lem:rationalencoding}.

\begin{lem}[Polynomial encoding]
\label{lem:rationalencodingMagee}
Let $Q:=(L+M)q_0$. For every $h \in \calP_q$, there is a rational 
function $\Psi_h:=\frac{f_h}{g_{Qq}}$ with 
$f_h,g_{Qq}\in\mathcal{P}_{\lfloor 4Qq(1+\log Qq)\rfloor}$
so that
$$
	N^{-(L+M)} \E[\trace h(\xn)] = \Psi_h(\tfrac{1}{N}) =
	\Psi_h(-\tfrac{1}{N})
$$
for all $N\in\bN$ such that $N>Qq$. Here $g_{Qq}$ is defined as in
\eqref{eq:defofgL}.
\end{lem}

\begin{proof}
The rational representation is an immediate consequence of the first part
of \cite[Theorem 3.1]{magee2024quasiexp}. That $\Psi_h(\tfrac{1}{N}) =
\Psi_h(-\tfrac{1}{N})$ is proved in Appendix \ref{app:magee}.
\end{proof}

With Lemma \ref{lem:rationalencodingMagee} in hand, we can now repeat all 
the arguments in sections 
\ref{sec:taylorstheoremHaarU}--\ref{sec:pfmainHaarU} identically in the 
present setting, except that we use the strong convergence result of 
\cite{magee2024quasiexp} as input to the bootstrapping argument (as in 
section \ref{sec:hayes}, the argument of section \ref{sec:guebasiccvg} 
fails here as $X^N$ is not $N$-dimensional). This yields the following.

\begin{thm}
\label{thm:magee2}
Let $\varepsilon\in[\frac{A}{c\sqrt{N}},1]$, and assume that 
$D\le e^{cN\varepsilon^2/A^2\log^2(N\varepsilon^2)}$. Then
$$
	\mathbf{P}\big[
	\|P(\pi_{L,M}(\boldsymbol{U}^N),
	\pi_{L,M}(\boldsymbol{U}^N)^*)\| \ge
	(1+\varepsilon)\|P(\boldsymbol{u},\boldsymbol{u}^*)\|
	\big]
	\le \frac{N^{L+M}}{c\varepsilon}
	e^{-cN\varepsilon^2/A^2\log^2(N\varepsilon^2)}.
$$
Here $c$ depends only on $q_0$ and $r$,
and $A:=(L+M)(1+\log(L+M))$.
\end{thm}

We can now conclude the proof of Theorem \ref{thm:mageedlsalle}.

\begin{proof}[Proof of Theorem \ref{thm:mageedlsalle}]
For any $1\le R\le N^{1/3}$, we can estimate
\begin{multline*}
	\mathbf{P}\Big[
	\max_{1\le L+M\le R}
	\|P(\pi_{L,M}(\boldsymbol{U}^N),
	\pi_{L,M}(\boldsymbol{U}^N)^*)\| \ge
	(1+\varepsilon)\|P(\boldsymbol{u},\boldsymbol{u}^*)\|
	\Big]
\\
	\le 
	\frac{R^2 N^R}{c\varepsilon}
	e^{-cN\varepsilon^2/R^2(1+\log R)^2\log^2(N\varepsilon^2)}
\end{multline*}
for $\varepsilon\in [\frac{R(1+\log R)}{c\sqrt{N}},1]$ using
Theorem \ref{thm:magee2} and a union bound. In particular,
if we choose $R\le N^{1/3-\delta}$ and $\varepsilon=N^{-\delta}$ for any 
$\delta>0$ sufficiently small, the right-hand side of this inequality is 
bounded by $e^{-cN^{1/3}/\log^4 N}$  and thus
$$
	\max_{1\le L+M\le N^{1/3-\delta}}
	\|P(\pi_{L,M}(\boldsymbol{U}^N),
	\pi_{L,M}(\boldsymbol{U}^N)^*)\| \le
	(1+o(1))\|P(\boldsymbol{u},\boldsymbol{u}^*)\|
	\quad\text{a.s.}
$$
as $N\to\infty$ by the Borel-Cantelli lemma. By
\cite[Lemma 5.13]{magee2024quasiexp}, this implies that
$$
	\lim_{N\to\infty}
	\max_{\substack{\lambda\vdash L,~\mu\vdash M\\
	1\le L+M\le N^{1/3-\delta}}}
	\big|
	\|P(\pi_{\lambda,\mu}(\boldsymbol{U}^N),
	\pi_{\lambda,\mu}(\boldsymbol{U}^{N})^*)\| 
	-
	\|P(\boldsymbol{u},\boldsymbol{u}^*)\|
	\big| =0\quad\text{a.s.}
$$
for every noncommutative polynomial $P$. In order to extend this result
to arbitrary irreducible representations of $\mathrm{U}(N)$ of
dimension at most $\exp(N^{1/3-\delta})$, it remains to account for
the scalar factor that appears in Lemma \ref{lem:pidimension}.

To this end, fix a noncommutative polynomial $P$, and define for every
$\boldsymbol{\theta}=(\theta_1,\ldots,\theta_r)\in [0,2\pi)^r$
the twisted noncommutative polynomial $P_{\boldsymbol{\theta}}$ as
$$
	P_{\boldsymbol{\theta}}
	(u_1,\ldots,u_r,u_1^*,\ldots,u_r^*) =
	P(e^{i\theta_1}u_1,\ldots,e^{i\theta_r} u_r,
	e^{-i\theta_1}u_1^*,\ldots,e^{-i\theta_r} u_r^*).
$$
It is readily seen that $\boldsymbol{\theta}\mapsto 
\|P_{\boldsymbol{\theta}}(U_1,\ldots,U_r,U_1^*,\ldots,U_r^*)\|$
has a uniformly bounded Lipschitz constant for all choices of
unitary operators $U_1,\ldots,U_r$.
Thus we obtain
$$
	\lim_{N\to\infty}
	\max_{\substack{\lambda\vdash L,~\mu\vdash M\\
	1\le L+M\le N^{1/3-\delta}}}
	\sup_{\boldsymbol{\theta}}
	\big|
	\|P_{\boldsymbol{\theta}}(\pi_{\lambda,\mu}(\boldsymbol{U}^N),
	\pi_{\lambda,\mu}(\boldsymbol{U}^{N})^*)\| 
	-
	\|P_{\boldsymbol{\theta}}(\boldsymbol{u},\boldsymbol{u}^*)\|
	\big| =0\quad\text{a.s.}
$$
by the Arzel\`a-Ascoli theorem. Furthermore, we note that
$\|P_{\boldsymbol{\theta}}(\boldsymbol{u},\boldsymbol{u}^*)\|=
\|P(\boldsymbol{u},\boldsymbol{u}^*)\|$ as 
$e^{i\boldsymbol{\theta}}\boldsymbol{u}=
(e^{i\theta_1}u_1,\ldots,e^{i\theta_r}u_r)$ are free Haar unitaries
for every $\boldsymbol{\theta}$. The conclusion now follows from
Lemma \ref{lem:pidimension} by 
choosing $e^{i\theta_j} = \det(U_j^N)^k$ for
$j\in[r]$.
\end{proof}

\begin{rem}
The proof of Theorem \ref{thm:magee2} is not 
independent of \cite{magee2024quasiexp}, as we have used the main result 
of that paper as input to the bootstrapping argument. We note, however, 
that the bootstrapping argument only requires strong convergence of 
$\pi_{K,L}$ as $N\to\infty$ for \emph{fixed} $K,L$. This is considerably 
simpler to achieve than the case of growing $K,L$; in particular, it does 
not require parts 2--3 of \cite[Theorem~3.1]{magee2024quasiexp} whose 
proof is based on delicate group-theoretic cancellations. (Alternatively, 
the main result of \cite{BC20} would also have sufficed for this purpose.)

This considerable simplification was made possible here by exploiting the 
fact that $\mathrm{U}(N)$ has strong concentration of measure properties. 
We emphasize, however, that there are many interesting situations where 
this is not the case. This issue arises in particular in the analogue of 
Theorem \ref{thm:mageedlsalle} where $\mathrm{U}(N)$ is replaced by the 
symmetric group $\mathrm{S}_N$ that was recently proved by Cassidy 
\cite{Cas24}, whose analysis requires the full strength of the methods 
developed in \cite{magee2024quasiexp}.
\end{rem}

\appendix

\section{An upper bound on the coefficient dimension}
\label{app:pisierexpn2}

The aim of this Appendix is to explain the observation, due to Pisier 
\cite{pisier2014random}, that strong convergence of the classical 
ensembles must fail for matrix coefficients of dimension $D=e^{O(N^2)}$. 
As this is not stated explicitly in \cite{pisier2014random}, we provide a 
self-contained proof. We focus on GUE/GOE/GSE matrices for simplicity.

\begin{lem}[Pisier]
Let $\boldsymbol{G}^N$ and $\boldsymbol{s}$ be defined as in Theorem 
\ref{thm:maingauss}. Then there exists $r\in\bN$ and matrices
$A_1^N,\ldots,A_r^N\in\mathrm{M}_{D_N}(\bC)$ with
$D_N = e^{O(N^2)}$ so that
$$
	\Bigg\|
	\sum_{i=1}^r A_i^N\otimes G_i^N\Bigg\|
	\ge
	(2+o(1))\Bigg\| \sum_{i=1}^r A_i^N\otimes s_i \Bigg\|
	\quad\text{a.s.}\quad\text{as}\quad N\to\infty.
$$
\end{lem}

\begin{proof}
Let $\mathcal{B}_N$ be an optimal $\varepsilon$-net 
of $\{M\in\mathrm{M}_N(\bC)_{\rm sa}:\|M\|\le 3\}$
with respect to the operator norm. It is classical 
that $\#\mathcal{B}_N \le (\frac{C}{\varepsilon})^{N^2}$.
We define $A_1^N,\ldots,A_r^N$ to be the block-diagonal matrices whose 
corresponding $N\times N$ blocks range over all $r$-tuples of elements of 
$\mathcal{B}_N$. Thus $D_N \le N(\frac{C}{\varepsilon})^{rN^2}$.

Now note that $\|G_i^N\|=2+o(1)$ a.s. 
Therefore, a.s.\ for all sufficiently large $N$, there are
$B_1^N,\ldots,B_r^N\in\mathcal{B}_N$ so that
$\|\bar G_i^N - B_i^N\|\le\varepsilon$, where $\bar M$ denotes the 
elementwise complex conjugate of $M$. By construction
$$
	\Bigg\|
	\sum_{i=1}^r A_i^N\otimes G_i^N\Bigg\| \ge
	\Bigg\|
	\sum_{i=1}^r B_i^N\otimes G_i^N\Bigg\| \ge
	\Bigg\|
	\sum_{i=1}^r \bar G_i^N\otimes G_i^N\Bigg\| -
	3r\varepsilon
$$
a.s.\ for all sufficiently large $N$. But note that
$$
	\Bigg\|
	\sum_{i=1}^r \bar G_i^N\otimes G_i^N\Bigg\| 
	\ge
	\Bigg\langle v_N,
	\Bigg(\sum_{i=1}^r \bar G_i^N\otimes G_i^N\Bigg)v_N
	\Bigg\rangle
	=
	\sum_{i=1}^r \tr(G_i^{N*}G_i^N) =
	(1+o(1))r
$$
as $N\to\infty$ a.s.\ by Lemma \ref{lem:waf} and concentration of 
measure,
where we defined $v_N := \frac{1}{\sqrt{N}}\sum_{j=1}^N e_j\otimes e_j$.
On the other hand, we can estimate
$$
	\Bigg\|\sum_{i=1}^r A_i^N\otimes s_i\Bigg\|
	\le 
	2\,\Bigg\|\sum_{i=1}^r (A_i^N)^2\Bigg\|^{1/2} \le
	6\sqrt{r}
$$
by the free Khintchine inequality \cite[eq.\ (9.9.8)]{Pis03}. The 
conclusion follows readily, for example, by choosing the parameters 
$\varepsilon=\tfrac{1}{6}$ and
$\sqrt{r}\ge 24$.
\end{proof}

The proof is readily adapted to achieve the same conclusion in the setting 
of Theorem \ref{thm:mainhaar}; then one may use \cite[eq.\ (9.7.1)]{Pis03} 
instead of the free Khintchine inequality. However, the situation for 
other ensembles may be even more restrictive. For example, for random 
permutations as in section \ref{sec:permu}, the set $\mathcal{B}_N$ can be 
replaced by $\mathrm{S}_N$ to show that strong convergence fails
already when $D_N=e^{O(N\log N)}$.

\section{Strong approximation of tensor models}
\label{app:tensor}

In this appendix we adopt the setting and notations of
section \ref{sec:graph}. We aim to prove Lemma \ref{lem:graphlimit}.
We begin by noting the following polynomial encoding.

\begin{lem}
\label{lem:bivpoly}
For every $P\in\mathrm{M}_D(\mathbb{C})\otimes
\mathbb{C}\langle\boldsymbol{\hat s}\rangle$, 
there is a polynomial $\Phi$ so that
$$
	\mathbf{E}[\mathop{\mathrm{tr}}
        P(\boldsymbol{\hat G}^{N,T})] =
	\Phi(\tfrac{1}{N^2},\tfrac{1}{T})
	\quad\text{for all }N,T\in\bN.
$$
\end{lem}

\begin{proof}
Let $Z_{v,t}$ be independent random matrices whose distribution does not 
depend on $t$, and let $Z_v^T = T^{-1/2}\sum_{t=1}^T Z_{v,t}$.
In the following, we fix indices $v_1,\ldots,v_q$.
Then
$c(\pi(t_1,\ldots,t_q)):=\mathbf{E}[\tr Z_{v_1,t_1}\cdots Z_{v_q,t_q}]$ 
only depends on $(t_1,\ldots,t_q)$ through the partition 
$\pi(t_1,\ldots,t_q)\in\mathrm{P}([q])$
whose elements are $\{i\in[q]:t_i=t\}$. Thus
$$
	\E[\tr Z_{v_1}^T\cdots Z_{v_q}^T] =
	\frac{1}{T^{q/2}}
	\sum_{\pi\in \mathrm{P}([q])}
	T(T-1)\cdots (T-|\pi|+1)
	\,c(\pi).
$$
If in addition $Z_{v,t}$ has the same distribution as $-Z_{v,t}$, then
$c(\pi)=0$ unless $q$ is even and $|\pi|\le\frac{q}{2}$. It is then 
evident that $\E[\tr Z_{v_1}^T\cdots Z_{v_q}^T]$ is a polynomial of 
$\frac{1}{T}$.

The special case $Z_v^T = \hat G_v^{N,T}$ satisfies the above conditions, 
and moreover each $c(\pi)$ is a polynomial of $\frac{1}{N^2}$ as noted
in the proof of Theorem \ref{thm:hayes}. The conclusion follows in the 
case $P(\boldsymbol{\hat s})=\hat s_{v_1}\cdots \hat s_{v_q}$,
and extends to general $P$ by linearity.
\end{proof}

Lemma \ref{lem:bivpoly} readily implies the following weak convergence
statement.

\begin{cor}
\label{cor:graphlimitweak}
For every 
$P\in\mathrm{M}_D(\mathbb{C})\otimes
\mathbb{C}\langle\boldsymbol{s}\rangle$, we have
$$
        (\mathrm{tr}\otimes\tau)(P(\boldsymbol{\hat s}^T))=
        (1+o(1))\,(\mathrm{tr}\otimes\tau)(P(\boldsymbol{\hat s}))
        \quad\text{as}\quad T\to\infty.
$$
\end{cor}

\begin{proof}
By Lemma \ref{lem:graphcov}, the central limit theorem, and weak
convergence of $\boldsymbol{\hat G}^N$ \cite{CC21}
$$
        \lim_{N\to\infty}\lim_{T\to\infty}
	\mathbf{E}[\mathop{\mathrm{tr}}
        P(\boldsymbol{\hat G}^{N,T})] =
        \lim_{N\to\infty}\mathbf{E}[\mathop{\mathrm{tr}}
        P(\boldsymbol{\hat G}^N)]
        = (\mathrm{tr}\otimes\tau)(P(\boldsymbol{\hat s})).
$$
On the other hand, Lemma \ref{lem:waf} yields
$$
        \lim_{T\to\infty}\lim_{N\to\infty}
	\mathbf{E}[\mathop{\mathrm{tr}}
        P(\boldsymbol{\hat G}^{N,T})] = 
        \lim_{T\to\infty}(\mathrm{tr}\otimes\tau)(P(\boldsymbol{\hat s}^T)).
$$
To conclude the proof, it remains to note that Lemma \ref{lem:bivpoly} 
enables us to exchange the order of the limits as
$N\to\infty$ and $T\to\infty$.
\end{proof}

To upgrade the weak convergence of Corollary \ref{cor:graphlimitweak} to 
norm convergence, we will use that the models $\boldsymbol{\hat s}^T$ 
satisfy a Haagerup inequality uniformly in $T$.

\begin{lem}
\label{lem:graphhaagerup}
There exist constants $C,a>0$ so that
$\|P(\boldsymbol{\hat s}^T)\| \le C q^a
\|P(\boldsymbol{\hat s}^T)\|_{L^2(\tau)}$
for every $T,q\in\mathbb{N}$ and noncommutative polynomial
$P\in\mathbb{C}\langle\boldsymbol{\hat s}\rangle$
of degree $q$.
\end{lem}

\begin{proof}
Define in $\bigotimes_{\ell\in[L]}C^*_{\rm red}(\mathrm{F}_R)\simeq
\bigotimes_{\ell\in K_v}C^*_{\rm red}(\mathrm{F}_R)\otimes
\bigotimes_{\ell\in [L]\backslash K_v} C^*_{\rm red}(\mathrm{F}_R)$ 
$$
        \hat s_v^{T,S} :=
        \frac{1}{\sqrt{T}} 
        \sum_{t=1}^T
        \big(s_{v,t}^S \otimes\cdots \otimes
        s_{v,t}^S\big)\otimes \id
$$
where $s_{v,t}^S := \frac{1}{\sqrt{2S}}\sum_{s=1}^S(\lambda(g_{v,t,s})+ 
\lambda(g_{v,t,s})^*)$, and let 
$\boldsymbol{\hat s}^{T,S}=(\hat s_v^{T,S})_{v\in 
V}$. Here $R=VTS$, and $\{g_{v,t,s}:v\in[V],t\in[T],s\in[S]\}$ are the 
free generators of $\mathrm{F}_R$.

By the Haagerup inequality for the free group $\mathrm{F}_R$ \cite[Lemma 
1.4]{Haa78} and stability of Haagerup inequalities under direct 
products \cite[Lemma 2.1.2]{Jol90} (see also \cite{CHR13}), there exist 
$C,a>0$ such that
for all $p,T,S\in\mathbb{N}$ and 
$P\in\mathbb{C}\langle\boldsymbol{\hat s}\rangle$
of degree $q$
$$
        \|P(\boldsymbol{\hat s}^{T,S})\|_{L^{2p}(\tau)} \le Cq^a
        \|P(\boldsymbol{\hat s}^{T,S})\|_{L^2(\tau)}.
$$
As $\boldsymbol{\hat s}^{T,S}$ converges weakly to
$\boldsymbol{\hat s}^{T}$ as $S\to\infty$ by the free central limit 
theorem \cite[Theorem 8.17]{NS06},
the conclusion follows by taking $S\to\infty$ and then $p\to\infty$.
\end{proof}

We can now complete the proof of Lemma \ref{lem:graphlimit}.

\begin{proof}[Proof of Lemma \ref{lem:graphlimit}]
We first note that by Corollary \ref{cor:graphlimitweak},
$$
        \|P(\boldsymbol{\hat s}^{T})\| \ge
        \|P(\boldsymbol{\hat s}^{T})\|_{L^{2p}({\tr\otimes\tau})} =
        (1+o(1))
        \|P(\boldsymbol{\hat s})\|_{L^{2p}({\tr\otimes\tau})}
        \quad\text{as }T\to\infty
$$
for every $p\in\mathbb{N}$. Taking $p\to\infty$ yields the lower bound.

Next, we apply Lemma \ref{lem:graphhaagerup} as in the
proof of \cite[Theorem 4.1]{BCSV23} to estimate
$$
        \|P(\boldsymbol{\hat s}^{T})\|
        \le
        D^{3/4p}
        (Cpq_0)^{a/2p}
        \|P(\boldsymbol{\hat s}^{T})\|_{L^{4p}({\tr\otimes\tau})} 
$$
for every $p\in\mathbb{N}$, where $q_0$ is the degree of $P$.
The upper bound follows by first taking $T\to\infty$ on the right-hand 
side using
Corollary \ref{cor:graphlimitweak}, and then $p\to\infty$.
\end{proof}

\section{Supersymmetric duality for stable characters of $\mathrm{U}(N)$}
\label{app:magee}

Lemma \ref{lem:rationalencodingMagee} relies on an unpublished result of 
Magee which states that every stable representation of $\mathrm{U}(N)$ 
exhibits a form of supersymmetric duality. We are grateful to 
Michael Magee for allowing us to include the proof here.

\begin{lem}[Magee]
\label{lem:magee}
Fix $L,M\in\bZ_+$ with $L+M\ge 1$, let $\lambda\vdash L$ and 
$\mu\vdash M$, and let $w(\haar_1,\ldots,\haar_r)$ be any word
in i.i.d.\ Haar-distributed elements of $\mathrm{U}(N)$. 
Then there exists a rational function $\Psi$ such that
\begin{alignat*}{2}
	&\E[\trace \pi_{\lambda,\mu}(w(\haar_1,\ldots,\haar_r))]
	&&= \Psi(\tfrac{1}{N}),\\
	&\E[\trace \pi_{\bar\lambda,\bar\mu}(w(\haar_1,\ldots,\haar_r))]
	&&= (-1)^{L+M}\Psi(-\tfrac{1}{N})
\end{alignat*}
for all $N\ge L+M$, where $\bar\lambda,\bar\mu$ denote the conjugate 
partitions of $\lambda,\mu$.
\end{lem}

Note that the duality statement in Lemma \ref{lem:rationalencodingMagee} 
follows directly from Lemma~\ref{lem:magee}, as both $\pi_{\lambda,\mu}$ 
and $\pi_{\bar\lambda,\bar\mu}$ appear in $\pi_{L,M}$ with multiplicity 
one.

\begin{proof}[Proof of Lemma \ref{lem:magee}]
For every permutation $\sigma$ with cycle type
$(\sigma_1,\ldots,\sigma_r)$, let
$$
	p_\sigma(U) := \trace[U^{\sigma_1}]\trace[U^{\sigma_2}]
	\cdots\trace[U^{\sigma_r}].
$$
By combining \cite[eq.\ (0.3)]{Koi89}
and \cite[p.\ 78, eq.\ (12)]{Ful97}, we can express the character 
$\trace \pi_{\lambda,\mu}(U)$ 
for any $U\in\mathrm{U}(N)$ by the explicit formula
$$
	\trace \pi_{\lambda,\mu}(U) = 
	\sum_{\alpha,\beta,\gamma}
	\frac{
	(-1)^{|\alpha|} c^\lambda_{\alpha,\beta} c^\mu_{\bar\alpha,\gamma}
	}{|\beta|! |\gamma|!}
	\sum_{\sigma\in\mathrm{S}_{|\beta|}}
	\sum_{\sigma'\in\mathrm{S}_{|\gamma|}}
	\chi_\beta(\sigma)\, 
	\chi_\gamma(\sigma')\, 
	p_\sigma(U)\,
	p_{\sigma'}(U^*),
$$
where $c^\lambda_{\alpha,\beta}$ are Littlewood-Richardson coefficients,
$\chi_\beta$ is the character of $\mathrm{S}_{|\beta|}$ associated
to $\beta$,
and the sum is over all integer partitions $\alpha,\beta,\gamma$;
here and below, we write $|\alpha|:=k$ for $\alpha\vdash k$.
The sum is finite as $c^\lambda_{\alpha,\beta}=0$ unless
$|\alpha|+|\beta|=|\lambda|$.

Denote $W^N := w(\haar_1,\ldots,\haar_r)$.
Then \cite[Proposition 1.1 and Remark 1.9]{magee2019matrix}
imply that for every $\sigma\in\mathrm{S}_k$ and 
$\sigma'\in\mathrm{S}_l$, there is a rational function 
$\Psi_{\sigma,\sigma'}$ so that
$$
	\E[p_\sigma(W^N)\,p_{\sigma'}(W^{N*})]
	 = \Psi_{\sigma,\sigma'}(\tfrac{1}{N})
	 = (-1)^{c(\sigma)+c(\sigma')}
		\Psi_{\sigma,\sigma'}(-\tfrac{1}{N})
$$
for all $N\ge k+l$,
where $c(\sigma)$ denotes the number of cycles of $\sigma$.

Now recall that for every $\sigma\in\mathrm{S}_k$ and $\alpha\vdash k$, 
we have
$$
	\chi_{\bar\alpha}(\sigma) = \mathrm{sgn}(\sigma)\chi_\alpha(\sigma) =
	(-1)^{k-c(\sigma)}\chi_\alpha(\sigma).
$$
The proof is concluded by noting that
\begin{multline*}
	(-1)^{L+M}
	\sum_{\alpha,\beta,\gamma}
	\frac{
	(-1)^{|\alpha|} c^\lambda_{\alpha,\beta} c^\mu_{\bar\alpha,\gamma}
	}{|\beta|! |\gamma|!}
	\sum_{\sigma\in\mathrm{S}_{|\beta|}}
	\sum_{\sigma'\in\mathrm{S}_{|\gamma|}}
	\chi_\beta(\sigma)\, 
	\chi_\gamma(\sigma')\, 
	\Psi_{\sigma,\sigma'}(-\tfrac{1}{N}) = \\
	\sum_{\alpha,\beta,\gamma}
	\frac{
	(-1)^{|\alpha|} c^\lambda_{\alpha,\beta} c^\mu_{\bar\alpha,\gamma}
	}{|\beta|! |\gamma|!}
	\sum_{\sigma\in\mathrm{S}_{|\beta|}}
	\sum_{\sigma'\in\mathrm{S}_{|\gamma|}}
	\chi_{\bar\beta}(\sigma)\, 
	\chi_{\bar\gamma}(\sigma')\, 
	\Psi_{\sigma,\sigma'}(\tfrac{1}{N})
	= \E[\trace \pi_{\bar\lambda,\bar\mu}(W^N)],
\end{multline*}
where we used $L=|\alpha|+|\beta|$,
$M=|\alpha|+|\gamma|$, and that
$c^\lambda_{\alpha,\beta}=c^{\bar\lambda}_{\bar\alpha,\bar\beta}$
(cf.\ \cite[p.\ 62]{Ful97}).
\end{proof}

\subsection*{Acknowledgments}

We thank Pierre Deligne, Michael Magee, F\'elix Parraud, and Mikael de la 
Salle for helpful discussions, and the referee for helpful 
suggestions. We are grateful to Michael Magee for 
allowing us to include the argument of Appendix~\ref{app:magee} in this 
paper, and to Mikael de la Salle for suggesting Remark \ref{rem:cmale}. 

This work was completed while RvH was a member of the Institute for 
Advanced Study in Princeton, NJ, which is gratefully acknowledged for 
providing a fantastic mathematical environment.

CFC was supported in part 
by a Simons-CIQC postdoctoral fellowship through NSF grant QLCI-2016245. 
JGV was supported in part by Joel A.\ Tropp through NSF grant FRG-1952777, 
ONR grant N-00014-24-1-2223, and the Carver Mead New Adventures 
Fund, and in part by Caltech IST and a Baer--Weiss CMI Fellowship. RvH 
was supported in part by NSF grant DMS-2347954.


\bibliographystyle{abbrv}
\bibliography{ref}

\end{document}